\documentclass[11pt]{article}

\usepackage{amsmath,amssymb,amsthm,mathtools,bm,fullpage,mathrsfs}
\usepackage{graphicx}
\usepackage[round]{natbib}
\usepackage[colorlinks=true,linkcolor=blue,citecolor=blue,urlcolor=blue]{hyperref}
\usepackage{refcount}
\usepackage[capitalize,noabbrev]{cleveref}
\crefname{auxlemma}{Lemma}{Lemmas}
\Crefname{auxlemma}{Lemma}{Lemmas}
\crefname{lbaux}{Lemma}{Lemmas}
\Crefname{lbaux}{Lemma}{Lemmas}
\crefname{algocf}{Algorithm}{Algorithms}
\Crefname{algocf}{Algorithm}{Algorithms}
\newcommand{\CrefInTitle}[2]{\texorpdfstring{\Cref{#2}}{#1~\getrefnumber{#2}}}
\usepackage{tikz}
\usetikzlibrary{arrows.meta,calc,positioning,patterns}
\usepackage{xcolor}
\definecolor{cI}  {RGB}{31,119,180}
\definecolor{cII} {RGB}{44,160,44}
\definecolor{cIII}{RGB}{255,127,14}
\definecolor{cIV} {RGB}{214,39,40}
\usepackage[ruled,vlined]{algorithm2e}
\usepackage{placeins}
\usepackage{enumerate}

\DontPrintSemicolon
\SetKwInput{KwIn}{Input}
\SetKwInput{KwOut}{Output}

\newtheorem{theorem}{Theorem}
\newtheorem{lemma}[theorem]{Lemma}
\newtheorem{proposition}[theorem]{Proposition}

\newtheorem{corollary}[theorem]{Corollary}

\newcommand{\E}{\mathbb{E}}
\newcommand{\R}{\mathbb{R}}
\newcommand{\Pbb}{\mathbb{P}}

\newcommand{\TV}{\mathrm{TV}}
\newcommand{\KL}{\mathrm{KL}}
\newcommand{\chiTwo}{\chi^2}
\newcommand{\cH}{\mathcal{H}}
\newcommand{\cF}{\mathcal{F}}
\newcommand{\cD}{\mathcal{D}}

\newcommand{\Var}{\mathrm{Var}}
\newcommand{\GS}{\mathrm{GS}}
\renewcommand{\Re}{\operatorname{Re}}
\renewcommand{\Im}{\operatorname{Im}}

\newcommand{\Lap}{\mathrm{Lap}}
\newcommand{\Unif}{\mathrm{Unif}}
\newcommand{\iid}{\mathrm{i.i.d.}}
\newcommand{\Arg}{\operatorname{Arg}}

\title{\LARGE Sharp Minimax Rates for Smooth Two-Sample Testing under Central Differential Privacy}
\author{Ilmun Kim\\
Department of Mathematical Sciences \\ KAIST \\[1ex]
\texttt{ilmunk@kaist.ac.kr}}
\date{\today}

\begin{document}

\maketitle
\begin{abstract}
We establish sharp minimax limits for two-sample testing of H\"older-smooth densities under central differential privacy. Given two independent samples, the goal is to decide whether the underlying distributions are identical or separated in $L_1$ distance, while releasing only an $\varepsilon$-differentially private decision. We show that privacy changes the classical smooth-testing boundary through multiple regimes: the optimal separation radius is the maximum of four terms, consisting of the classical nonprivate rate and three distinct privacy-induced barriers. Which barrier is active depends on the privacy budget and the smoothness-to-dimension ratio, yielding a sharp phase diagram. Our upper bound discretizes the samples, applies a private discrete two-sample test to the resulting histograms, and chooses the bin resolution to balance approximation bias, sampling fluctuations, and privacy noise. The procedure also admits a permutation-calibrated implementation with finite-sample type~I error control. For the lower bounds, we combine smooth perturbation constructions with privacy-specific coupling and transport inequalities, showing that all four terms are unavoidable. Finally, when the smoothness is unknown, we develop a multiscale private test that attains the optimal adaptive rate and prove a matching lower bound. Adaptation costs exactly an iterated-logarithmic factor, and this cost appears only in the classical nonprivate term.
\end{abstract}

\section{Introduction}
\label{sec:intro}

Two-sample testing asks whether two independent samples come from the same distribution. Canonical examples include comparing treatment and control groups, populations observed at different sites or time periods, or demographic cohorts. The same question is central in modern data applications, for example when assessing distribution shift between training and deployment populations \citep{QuioneroCandelaSugiyamaSchwaighoferLawrence2009,KohSagawaMarklundEtAl2021} or auditing algorithmic systems across groups \citep{KleinbergLakkarajuLeskovecLudwigMullainathan2018,ObermeyerPowersVogeliMullainathan2019}. In many of these settings, the data are individual-level and sensitive, including medical records, financial histories, and profiles derived from personal data. Even a binary testing decision may still disclose some information, as changing one record can affect the released output and potentially signal whether an individual participated or influenced the result.

This paper studies two-sample testing under central differential privacy (DP) \citep{DworkRoth2014}. In this model, a trusted data holder, such as a hospital system, platform, agency, or secure data environment, may use the raw samples internally, but the released output must satisfy differential privacy. We investigate how central differential privacy changes the separation required for reliable two-sample testing. For H\"older-smooth densities, we show that the sharp $L_1$ separation radius is the maximum of the classical nonprivate smooth-testing term and three distinct privacy terms. The dominant term changes across privacy regimes, reflecting the interaction among smooth approximation, discrete closeness testing after binning, and the noise needed to privatize the final decision.

\subsection{Problem Setup and Rate Preview} \label{sec:problem-setup}
We compare two unknown smooth distributions from independent samples $X_1,\dots,X_{N_X}\stackrel{\iid}{\sim} f$ and $Y_1,\dots,Y_{N_Y}\stackrel{\iid}{\sim} g$, where $f$ and $g$ are unknown densities on $[0,1]^d$ belonging to a H\"older class of smoothness $s>0$. For real numbers $a$ and $b$, let $a\vee b:=\max\{a,b\}$ and $a\wedge b:=\min\{a,b\}$.  We write $N:=N_X\wedge N_Y$ for the smaller sample size.  The goal is to distinguish
\begin{align*}
H_0:\ f \equiv g
\qquad\text{versus}\qquad
H_1:\ \|f-g\|_1\ge r,
\end{align*}
using a centrally differentially private test, where $\|f-g\|_1:=\int_{[0,1]^d}|f-g|\,dx$ denotes the $L_1$ distance. The procedure may access the raw samples internally, but the released decision must satisfy $\varepsilon$-differential privacy. Formal definitions of the H\"older class, privacy constraint, and minimax separation radius are given in \Cref{sec:setup}.

We measure separation in $L_1$ distance. Since $\|f-g\|_1$ is twice the total variation distance, it has a direct operational interpretation as twice the largest discrepancy between the probabilities assigned by $f$ and $g$ to measurable events. It is also invariant under bijective transformations of the sample space, making it a natural scale-free metric for continuous testing problems~\citep[e.g.,][]{Ingster1987,BalakrishnanWasserman2019,DuboisBerrettButucea2023,ChhorCarpentier2025}.

Our main theorem, \Cref{thm:main-rate}, identifies the sharp $L_1$ separation radius under central DP. In the regime $N\varepsilon\ge1$, the rate is
\begin{align*}
N^{-2s/(4s+d)}
\;\vee\;
(N\sqrt{\varepsilon})^{-2s/(2s+d)}
\;\vee\;
(N^{3/2}\varepsilon)^{-2s/(4s+d)}
\;\vee\;
(N\varepsilon)^{-1}.
\end{align*}
The first term is the classical nonprivate smooth-testing rate. The other three are privacy terms, and different terms dominate in different privacy regimes. The phase transition depends on the smoothness $s$ and dimension $d$, as described in \Cref{cor:phase-transition}.

\subsection{Contributions}
The paper makes four main contributions.

\begin{itemize}
\item We give the minimax phase diagram for smooth two-sample testing under central DP. The optimal $L_1$ separation radius is the maximum of one classical smooth-testing term and three privacy terms, whose dominance changes with the privacy level and the smoothness-to-dimension ratio. In rough or high-dimensional settings all four regimes appear, whereas in smoother, lower-dimensional settings one intermediate privacy regime disappears. The precise phase diagram is stated in \Cref{cor:phase-transition} and illustrated in \Cref{fig:phase-envelope}.

\item We construct tests attaining the optimal rate. The procedure bins the observations, applies a private two-sample test to the resulting histograms, and chooses the resolution to balance approximation error, sampling noise, and privacy noise (\Cref{sec:ub-binning}). Smoothness places the binned distributions in a bounded-histogram subclass; optimizing the corresponding private discrete bounds then yields the continuous four-term rate. A key discrete ingredient is a bounded-histogram refinement of private closeness testing: the heavy-element nonprivate term present in unrestricted discrete closeness testing disappears for histograms induced by bounded densities. We also develop finite-sample implementations based on private permutation calibration (\Cref{alg:mcperm,cor:holder-l1-private-explicit-rate}), including a Rao--Blackwellized variant that averages out the auxiliary splitting randomness, preserves the same sensitivity bound, and admits a linear-time implementation. To our knowledge, this is the first permutation-calibrated test for $L_1$ two-sample closeness to achieve minimax-optimal sample complexity under central differential privacy.

\item We prove matching lower bounds for all four terms. The classical nonprivate term follows from smooth hypercube perturbations, while the linear privacy term follows from a fixed smooth perturbation combined with a DP coupling argument (\Cref{prop:nonprivate-lb,prop:linear-lb}). The two intermediate privacy terms require sharper DP-specific tools, namely a coupling argument \citep{AcharyaSunZhang2018} and a new transport inequality for differentially private tests (\Cref{lem:transport}). These arguments show that each term in the upper bound is individually sharp.

\item We develop an adaptive procedure for testing with unknown smoothness. Over compact smoothness ranges, it evaluates the Rao--Blackwellized statistic over a dyadic grid of resolutions and privately calibrates the resulting multiscale score. The procedure matches the fixed-smoothness privacy terms and pays only an iterated-logarithmic factor in the classical nonprivate term (\Cref{sec:ub-adaptive-smoothness}). We also prove a matching adaptive lower bound for the classical term, even without imposing a privacy constraint, showing that the iterated-logarithmic loss is unavoidable uniformly over the smoothness range.
\end{itemize}

\subsection{Related Work}
\label{sec:intro-related}

We now place the results in the context of work on nonparametric two-sample testing, private distribution testing, and private nonparametric inference.

\medskip\noindent\textbf{Classical nonprivate smooth testing.}
The nonprivate benchmark for smooth goodness-of-fit and two-sample testing is the Ingster-type separation rate $N^{-2s/(4s+d)}$ \citep{Ingster1987,IngsterSuslina2003}. This theory has been studied in multivariate density models \citep{AriasCastroPelletierSaligrama2018} and refined through local minimax radii, adaptive procedures, and permutation-calibrated tests \citep{Spokoiny1996,FromontLaurent2006,ButuceaTribouley2006,BalakrishnanWasserman2019,KimBalakrishnanWasserman2022,ChhorCarpentier2025}. The first term in \Cref{thm:main-rate} recovers this classical benchmark. Our contribution is to characterize the additional separation induced by requiring the released decision to satisfy central differential privacy.

\medskip\noindent\textbf{Discrete distribution testing and private closeness testing.}
The upper bound relies on discrete two-sample testing, commonly known as closeness testing, applied to the histograms obtained by binning the continuous samples. For discrete $L_1$ separation $\tau$, the nonprivate problem over a $k$-point domain has two sample-complexity terms, $\sqrt{k}/\tau^2$ and $k^{2/3}/\tau^{4/3}$ \citep{BatuFortnowRubinfeldSmithWhite2013,ChanDiakonikolasValiantValiant2014,DiakonikolasGouleakisKanePeeblesPrice2021,CanonneSun2022}. Early private goodness-of-fit work developed chi-squared procedures calibrated after privacy noise is added \citep{GaboardiLimRogersVadhan2016,RogersKifer2017}. Under central DP, identity and closeness testing exhibit several distinct privacy-induced terms, rather than a single effective-sample-size correction \citep{CaiDaskalakisKamath2017,AliakbarpourDiakonikolasRubinfeld2018,AcharyaSunZhang2018,Zhang2021}. Permutation-based ideas for private distribution testing appear in \citet{AliakbarpourDiakonikolasKaneRubinfeld2019}. Broader structural results and methods for central-DP hypothesis testing, including simple and high-dimensional hypotheses, have also been developed \citep{CanonneKamathMcMillanSmithUllman2019,CanonneKamathMcMillanUllmanZakynthinou2020,Narayanan2022}, while \citet{KazanShiGroceBray2023} propose a general test-of-tests framework. These results are complementary to ours, as they focus on model classes different from the smooth nonparametric two-sample problem studied here. In our setting, $k=\kappa^d$ is chosen by the procedure. Smoothness places the binned problem in a bounded-histogram subclass, and optimizing the resulting private discrete bounds over $\kappa$ produces the three privacy-dependent continuous terms in \eqref{eq:rate-combined}. Our lower-bound arguments are also related to general private analogues of Le Cam, Fano, and Assouad methods developed by \citet{AcharyaSunZhang2021}.

\medskip\noindent\textbf{Local privacy models.}
For discrete distribution testing under local differential privacy (LDP), foundational work establishes sample-optimal protocols with sharp distinctions between public-coin and private-coin mechanisms \citep{AcharyaCanonneFreitagTyagi2019,AcharyaCanonneFreitagSunTyagi2021}. For H\"older and Besov densities, goodness-of-fit rates under local privacy are given in \citet{LamWeilLaurentLoubes2022} and \citet{DuboisBerrettButucea2023}.  Recent work gives minimax-optimal LDP procedures for continuous two-sample testing over H\"older/Besov classes \citep{MunKwakKim2025}, while \citet{KentBerrettYu2026} give non-interactive and interactive procedures that are optimal up to logarithmic factors for general discrete and smooth continuous classes. These works give the closest local-private analogues of our problem. In contrast, we study the central model, where the raw samples may be used internally and only the released binary decision must be private. The resulting rate has a different phase structure, reflecting the weaker privacy constraint and smaller privacy penalty of the central model. In particular, central privacy allows raw-data binning and permutation calibration before the final release, whereas local protocols must privatize information at the individual-message level. Related phase transitions and adaptive procedures have been studied for federated nonparametric goodness-of-fit testing in a Gaussian white-noise model under distributed DP constraints \citep{CaiChakrabortyVuursteen2024}; that setting differs from the central-DP two-sample density model considered here.

\medskip\noindent\textbf{Private tests and estimation for continuous data.}
Within central privacy, several procedures have been developed for continuous two-sample and related nonparametric testing problems, but they target notions different from the smooth $L_1$ minimax boundary studied here. Kernel-based private two-sample tests build on the MMD/RKHS framework \citep{GrettonBorgwardtRaschScholkopfSmola2012} and use finite-dimensional approximations to kernel mean embeddings with privatized empirical feature moments, or apply private permutation calibration directly to MMD statistics \citep{RajLawSejdinovicPark2019,KimSchrab2026}. Their alternatives are mainly expressed through RKHS metrics rather than H\"older $L_1$ separation. Other central-DP nonparametric tests include private rank-based procedures \citep{CouchKazanShiBrayGroce2019} and Kolmogorov--Smirnov-type tests with finite-sample, distribution-free calibration \citep{AwanWang2025}. These works emphasize practical methodology, whereas our focus is the sharp smooth $L_1$ minimax two-sample boundary. For one-sample continuous goodness-of-fit (GOF) testing, \citet{KwakAhnLeePark2024} propose central-DP procedures based on discretization followed by private discrete GOF testing. Private equivalence testing under $\mathcal{A}_k$-type discrepancy \citep{OmerSheffet2024} is also complementary, addressing related questions rather than the sharp $L_1$ minimax rate over H\"older classes. On the estimation side, private nonparametric density estimation is studied in \citet{WassermanZhou2010,BarberDuchi2014,LalanneGarivierGribonval2023Cost}. More general minimax lower-bound and transport methods for private estimation and testing are developed by \citet{LalanneGarivierGribonval2023Complexity}. Testing and estimation, however, have different objectives and generally different minimax rates. As in nonprivate smooth testing, the testing boundary is not obtained by first estimating the two densities in $L_1$ and comparing the estimates. In the private problem, it is set by binning bias, discrete closeness testing, and privacy noise.

\subsection{Organization of the Paper}
\label{sec:intro-organization}
The paper is organized as follows. \Cref{sec:setup} gives the setup and notation. \Cref{sec:results} states the main minimax rate, sample-complexity bounds, and phase-transition consequences. \Cref{sec:ub} proves upper bounds via binning, private discrete closeness testing, and finite-sample implementations. \Cref{sec:lb} proves matching lower bounds using goodness-of-fit reductions, smooth hypercubes, and privacy-specific transport and coupling arguments. \Cref{sec:ub-adaptive-smoothness} treats adaptation to unknown smoothness and proves the corresponding adaptive lower bound. \Cref{sec:simulation} gives numerical illustrations, and \Cref{sec:discussion} discusses implications and open problems. The appendices collect auxiliary lemmas, tools, and deferred proofs.

\paragraph{Notation.}
The symbols $\lesssim$, $\gtrsim$, and $\asymp$ denote inequalities and equalities up to positive multiplicative constants. Unless stated otherwise, these constants may depend only on fixed problem parameters, such as $(d,s,L,\gamma,\beta)$, and on auxiliary fixed bounds such as $M$ when they are explicitly present, but never on $N$ or $\varepsilon$. For a positive integer $k$, we write $[k]:=\{1,\dots,k\}$ and let $\Delta_k:=\{v\in\R^k:v_i\ge0,\,\sum_i v_i=1\}$ denote the probability simplex. For vectors $v\in\R^k$, $\|v\|_1=\sum_{i=1}^k |v_i|$, $\|v\|_2=(\sum_{i=1}^k v_i^2)^{1/2}$, and $\|v\|_\infty=\max_{i=1,\dots,k} |v_i|$. For functions $u$ on $[0,1]^d$, $\|u\|_p$ denotes the Lebesgue $L_p$ norm. We write $\Lap(b)$ for the centered Laplace distribution with scale parameter $b$, with density $x\mapsto (2b)^{-1}\exp(-|x|/b)$. All logarithms are natural unless a base is specified.

\section{Model and Minimax Formulation}
\label{sec:setup}

\subsection{Smooth Two-Sample Model}

We observe two independent samples
\begin{align*}
X_1,\ldots,X_{N_X}\stackrel{\iid}{\sim} f,
\qquad
Y_1,\ldots,Y_{N_Y}\stackrel{\iid}{\sim} g,
\end{align*}
where $f$ and $g$ are Lebesgue densities on $[0,1]^d$. We assume that both densities belong to the same H\"older density class: for fixed $s>0$ and $L>1$, $f,g\in\cD_s^d(L)$, where $\cD_s^d(L)$ is defined below. The minimax rate is governed by the smaller sample size, and we set $N:=N_X\wedge N_Y$.

Let $d\ge 1$, $s>0$, $L>1$, $q:=\lceil s\rceil-1$, and $\eta:=s-q\in(0,1]$. The H\"{o}lder ball $\cH_s^d(L)$ is the set of all $h:[0,1]^d\to\R$ with
\begin{align*}
\max_{|\alpha|\le q}\|D^\alpha h\|_\infty\le L
\qquad\text{and}\qquad
\max_{|\alpha|=q}[D^\alpha h]_{C^\eta}\le L,
\end{align*}
where $\alpha=(\alpha_1,\ldots,\alpha_d)\in\mathbb{Z}_{\ge 0}^d$ is a multi-index, $|\alpha|=\alpha_1+\cdots+\alpha_d$,
\begin{align*}
D^\alpha h
:=
\frac{\partial^{|\alpha|}h}
{\partial x_1^{\alpha_1}\cdots\partial x_d^{\alpha_d}},
\end{align*}
and $[u]_{C^\eta}:=\sup_{x\ne y}|u(x)-u(y)|/\|x-y\|_2^\eta$. Thus, when $s$ is an integer, the definition requires derivatives of order $s-1$ to be Lipschitz. The H\"{o}lder density class is
\begin{align*}
\cD_s^d(L)
:=
\left\{
h\in\cH_s^d(L):
h\ge 0,\ \int_{[0,1]^d}h(x)\,dx=1
\right\};
\end{align*}
the condition $L>1$ ensures $\cD_s^d(L)$ contains the uniform density.

For $f,g\in\cD_s^d(L)$, write
\begin{align*}
\|f-g\|_1:=\int_{[0,1]^d}|f(x)-g(x)|\,dx.
\end{align*}
The testing problem is the composite null $H_0:f=g$ against alternatives separated in $L_1$:
\begin{align*}
H_1(r):=\{(f,g)\in\cD_s^d(L)^2:\|f-g\|_1\ge r\}.
\end{align*}

\subsection{Differential Privacy}

We adopt the central differential privacy framework of \citet{DworkMcSherryNissimSmith2006}. For an integer $n$, a measurable randomized mechanism $\mathcal M:([0,1]^d)^n\to\mathcal{Y}$ is $\varepsilon$-differentially private ($\varepsilon$-DP) if, for every pair of neighboring datasets $D,D'\in([0,1]^d)^n$ differing in one record and every measurable $S\subseteq\mathcal{Y}$,
\begin{align*}
\Pbb\bigl(\mathcal M(D)\in S\bigr)
\le
e^{\varepsilon}\,\Pbb\bigl(\mathcal M(D')\in S\bigr).
\end{align*}
In the two-sample problem, privacy is imposed on the pooled dataset
\begin{align*}
D=(X_{1:N_X},Y_{1:N_Y})\in([0,1]^d)^{N_X+N_Y},
\end{align*}
so neighboring datasets may differ in any one $X$-record or any one $Y$-record. Thus the central-DP constraint applies to the full inferential procedure as a function of the pooled sample. Tests are randomized binary mechanisms. For such a test $\mathcal M$, write $p_{\mathcal M}(D):=\Pbb(\mathcal M(D)=1)$, where the probability is over the internal randomness of $\mathcal M$. Thus $p_{\mathcal M}(D)$ is the rejection probability on dataset $D$, and differential privacy limits how much this rejection probability can change under a one-record modification of $D$.

For $f,g\in\cD_s^d(L)$, let
\begin{align*}
P_{f,g}^{N_X,N_Y}:=P_f^{\otimes N_X}\otimes P_g^{\otimes N_Y},
\qquad
\E_{f,g}^{N_X,N_Y}:=\E_{P_{f,g}^{N_X,N_Y}}.
\end{align*}
Then $\E_{f,g}^{N_X,N_Y}[p_{\mathcal M}]$ is the rejection probability of $\mathcal M$ when the two samples have densities $f$ and $g$. We write $\Phi_{\gamma,\varepsilon}^{\mathrm{cDP}}$ for the class of $\varepsilon$-DP tests whose type~I error is at most $\gamma$ uniformly over the composite null. The dependence on $N_X,N_Y,d,s,L$ is suppressed:
\begin{align*}
\Phi_{\gamma,\varepsilon}^{\mathrm{cDP}}
:=
\left\{
\mathcal M:([0,1]^d)^{N_X+N_Y}\to\{0,1\}\ :\
\mathcal M\text{ is }\varepsilon\text{-DP},\
\sup_{h\in\cD_s^d(L)}
\E_{h,h}^{N_X,N_Y}[p_{\mathcal M}]\le\gamma
\right\}.
\end{align*}
The type~II error of $\mathcal M$ at $(f,g)\in\cD_s^d(L)^2$ is $1-\E_{f,g}^{N_X,N_Y}[p_{\mathcal M}]$.

\subsection{Minimax Separation Radius}
\label{sec:setup-minimax}

Fix error levels $\gamma,\beta\in(0,1)$ with $\gamma+\beta<1$. The minimax $L_1$ separation radius $r^*_{\mathrm{2samp}}(N_X,N_Y,\varepsilon)$ is the smallest signal size at which some private test attains the prescribed uniform type~I and type~II errors. Formally, it is the infimum over separations $r$ for which there exists $\mathcal M\in\Phi_{\gamma,\varepsilon}^{\mathrm{cDP}}$ with type~II error at most $\beta$ uniformly over all alternatives $(f,g)\in\cD_s^d(L)^2$ satisfying $\|f-g\|_1\ge r$:
\begin{align*}
r^*_{\mathrm{2samp}}(N_X,N_Y,\varepsilon)
:=
\inf\Biggl\{
r>0 :
\exists\,\mathcal M\in\Phi_{\gamma,\varepsilon}^{\mathrm{cDP}}
\text{ s.t.\ }
\sup_{\substack{f,g\in\cD_s^d(L)\\\|f-g\|_1\ge r}}
\Bigl(1-\E_{f,g}^{N_X,N_Y}[p_{\mathcal M}]\Bigr)\le\beta
\Biggr\}.
\end{align*}
We characterize $r^*_{\mathrm{2samp}}(N_X,N_Y,\varepsilon)$ up to constants depending only on $(\gamma,\beta,d,s,L)$, as a function of $N=N_X\wedge N_Y$ and $\varepsilon$. When only this smaller sample size matters, we use the abbreviation $r^*_{\mathrm{2samp}}(N,\varepsilon)$ for the corresponding quantity. In the results below, the rate depends on the two sample sizes only through $N_X\wedge N_Y$, up to constants independent of the imbalance. Thus rate statements suppress dependence on the fixed smoothness, dimension, H\"older radius, and error levels, but not on $N$ or $\varepsilon$.

\section{Main Results}
\label{sec:results}

We now state the main rate theorem and two consequences. The minimax separation rate is the maximum of one classical term, two intermediate privacy terms, and a final linear privacy barrier, created by the interaction among smoothness, binning, and central-DP noise. \Cref{cor:sample-complexity,cor:phase-transition} give the corresponding sample-complexity and phase-transition forms.

\subsection{Minimax Two-Sample Rate under Central Differential Privacy}
\label{sec:results-rate}

The following theorem gives the minimax-optimal $L_1$ separation radius when $N\varepsilon\ge1$. This restriction excludes the regime $N\varepsilon<1$, where the radius is already of constant order. Indeed, when $N\varepsilon<1$, the linear privacy term $(N\varepsilon)^{-1}$ is at least of constant order, while the $L_1$ distance between two densities is at most two.

\begin{theorem}[Minimax separation radius under central DP]
\label{thm:main-rate}
For every fixed $\gamma,\beta\in(0,1)$ with $\gamma+\beta<1$ and $N\varepsilon\ge 1$,
\begin{equation}
r^*_{\mathrm{2samp}}(N,\varepsilon)
\;\asymp\;
N^{-2s/(4s+d)}
\vee
(N\sqrt\varepsilon)^{-2s/(2s+d)}
\vee
(N^{3/2}\varepsilon)^{-2s/(4s+d)}
\vee
(N\varepsilon)^{-1}.
\label{eq:rate-combined}
\end{equation}
\end{theorem}

The upper and lower bounds agree up to constants for each of the four terms, and every term in \eqref{eq:rate-combined} is necessary for minimax optimality. The first term is the classical smooth-testing rate. The remaining three terms arise from privacy constraints and correspond to distinct interactions between histogram approximation, discrete closeness testing, and privacy noise. Depending on the privacy regime and on the smoothness-to-dimension ratio, different terms determine the minimax rate. As \Cref{cor:phase-transition} shows, the $(N\sqrt\varepsilon)^{-2s/(2s+d)}$ term appears on the phase diagram only when the class is sufficiently rough, namely $s<d/4$.

\subsection{Sample Complexity and Privacy Phase Transition}

The next two corollaries restate the rate in two different ways. First, inverting the rate gives the optimal sample complexity needed to detect a prescribed separation $r$. Second, setting $\varepsilon=N^{-\alpha}$ turns the four terms into exponent functions, whose lower envelope identifies the active phase. The first consequence follows by directly inverting the rate in \Cref{thm:main-rate}.

\begin{corollary}[Sample complexity]\label{cor:sample-complexity}
For a target separation $r\in(0,2]$, the minimax-optimal sample size $N^*(r,\varepsilon):=\inf\{N:r^*_{\mathrm{2samp}}(N,\varepsilon)\le r\}$ satisfies
\begin{equation*}
N^*(r,\varepsilon)
\;\asymp\;
r^{-(4s+d)/(2s)}
\;\vee\;
\varepsilon^{-1/2}\,r^{-(2s+d)/(2s)}
\;\vee\;
\varepsilon^{-2/3}\,r^{-(4s+d)/(3s)}
\;\vee\;
(\varepsilon r)^{-1}.
\end{equation*}
\end{corollary}

The four terms in $N^*$ reflect the same four rate-determining effects as \Cref{thm:main-rate}. Their relative dominance depends on the privacy level, and the following reparameterization makes the phase boundaries explicit. The only purpose of the notation below is to track which of the four terms in \eqref{eq:rate-combined} is largest when the privacy budget is written as a power of $N$.

\begin{corollary}[Privacy phase transition]\label{cor:phase-transition}
Let $\varepsilon=N^{-\alpha}$ with $\alpha\in[0,1]$.  Define
\begin{align*}
\rho_1 := \frac{2s}{4s+d},\qquad
\rho_2(\alpha) := \frac{2s}{2s+d}\!\left(1-\frac{\alpha}{2}\right),\qquad
\rho_3(\alpha) := \frac{2s}{4s+d}\!\left(\frac{3}{2}-\alpha\right),\qquad
\rho_4(\alpha) := 1-\alpha.
\end{align*}
Then, uniformly for $\alpha\in[0,1]$,
\begin{align*}
r^*_{\mathrm{2samp}}(N,N^{-\alpha})\asymp N^{-\rho(\alpha)},
\qquad
\rho(\alpha):=\min\{\rho_1,\rho_2(\alpha),\rho_3(\alpha),\rho_4(\alpha)\},
\end{align*}
where $\rho$ is the lower envelope of the four exponent functions. Thus the active regime is determined by the exponent attaining the lower envelope. Equivalently:

\begin{enumerate}[(a)]
\item If $s<d/4$, then $ 0<\alpha_{12}<\alpha_{23}<\alpha_{34}<1, $ where $\alpha_{12}=4s/(4s+d)$, $\alpha_{23}=(d-2s)/d$, and $\alpha_{34}=(s+d)/(2s+d)$. In this case,
\begin{align*}
\rho(\alpha)=
\begin{cases}
\rho_1,
& 0\le \alpha\le \alpha_{12},\\[2pt]
\rho_2(\alpha),
& \alpha_{12}\le \alpha\le \alpha_{23},\\[2pt]
\rho_3(\alpha),
& \alpha_{23}\le \alpha\le \alpha_{34},\\[2pt]
\rho_4(\alpha),
& \alpha_{34}\le \alpha\le 1.
\end{cases}
\end{align*}
\item If $s\ge d/4$, then $\alpha_{34}=(s+d)/(2s+d)$, and
\begin{align*}
\rho(\alpha)=
\begin{cases}
\rho_1,
& 0\le \alpha\le \tfrac12,\\[2pt]
\rho_3(\alpha),
& \tfrac12\le \alpha\le \alpha_{34},\\[2pt]
\rho_4(\alpha),
& \alpha_{34}\le \alpha\le 1.
\end{cases}
\end{align*}
\end{enumerate}

See \Cref{fig:phase-envelope} for quantitative lower-envelope plots.
\end{corollary}

\begin{proof}
The proof is in \Cref{app:proof:cor:phase-transition}.
\end{proof}

\begin{figure}[t]
\centering
\begin{tikzpicture}[>=Latex, font=\small]
\tikzset{
  phaseIbg/.style={
    fill=cI!12,
    postaction={pattern=north east lines, pattern color=black!16}
  },
  phaseIIbg/.style={
    fill=cII!12,
    postaction={pattern=north west lines, pattern color=black!16}
  },
  phaseIIIbg/.style={
    fill=cIII!12,
    postaction={pattern=dots, pattern color=black!18}
  },
  phaseIVbg/.style={
    fill=cIV!12,
    postaction={pattern=crosshatch, pattern color=black!13}
  },
  rhoOne/.style={cI, dashed, line width=0.85pt},
  rhoTwo/.style={cII, densely dashed, line width=0.85pt},
  rhoThree/.style={cIII, dash pattern=on 4pt off 1.6pt on 0.8pt off 1.6pt, line width=0.85pt},
  rhoFour/.style={cIV, dotted, line width=1.0pt}
}
\fill[phaseIbg]   (0,0) rectangle (1.667,3.8);
\fill[phaseIIbg]  (1.667,0) rectangle (3.75,3.8);
\fill[phaseIIIbg] (3.75,0) rectangle (4.5,3.8);
\fill[phaseIVbg]  (4.5,0) rectangle (5,3.8);
\begin{scope}
  \clip (0,0) rectangle (5,3.8);
  \draw[black!25, thin] (1.667,0) -- (1.667,3.8);
  \draw[black!25, thin] (3.75,0)  -- (3.75,3.8);
  \draw[black!25, thin] (4.5,0)   -- (4.5,3.8);
  \draw[cI,   line width=2.2pt] (0,3.333)     -- (1.667,3.333);
  \draw[cII,  line width=2.2pt] (1.667,3.333) -- (3.75,2.5);
  \draw[cIII, line width=2.2pt] (3.75,2.5)    -- (4.5,2.0);
  \draw[cIV,  line width=2.2pt] (4.5,2.0)     -- (5,0);
  \draw[rhoOne]   (0,3.333) -- (5,3.333);
  \draw[rhoTwo]   (0,4.0)   -- (5,2.0);
  \draw[rhoThree] (0,5.0)   -- (5,1.667);
  \draw[rhoFour]  (0,20)    -- (5,0);
  \filldraw[fill=white, draw=black!70, line width=0.8pt] (1.667,3.333) circle (2pt);
  \filldraw[fill=white, draw=black!70, line width=0.8pt] (3.75,2.5)    circle (2pt);
  \filldraw[fill=white, draw=black!70, line width=0.8pt] (4.5,2.0)     circle (2pt);
\end{scope}
\draw[black!35] (0,0) rectangle (5,3.8);
\draw[->, black!75, line width=0.6pt] (0,0) -- (5.5,0)
  node[right, font=\scriptsize] {$\alpha$};
\draw[->, black!75, line width=0.6pt] (0,0) -- (0,4.1)
  node[above, font=\scriptsize, xshift=4pt] {$\rho(\alpha)$};
\foreach \x/\lbl in {
    0/{$0$}, 1.667/{$\tfrac{1}{3}$}, 3.75/{$\tfrac{3}{4}$},
    4.5/{$\tfrac{9}{10}$}, 5/{$1$}}{
  \draw[black!70] (\x,-0.08) -- (\x,0.08);
  \node[below=3pt, font=\scriptsize] at (\x,0) {\lbl};
}
\foreach \y/\lbl in {2.5/{$\tfrac{1}{8}$}, 3.333/{$\tfrac{1}{6}$}}{
  \draw[black!70] (-0.08,\y) -- (0.08,\y);
  \node[left=3pt, font=\scriptsize] at (0,\y) {\lbl};
}
\node[font=\scriptsize\bfseries, cI!80!black, fill=white, fill opacity=0.78,
      text opacity=1, inner sep=1pt] at (0.833,0.3) {I};
\node[font=\scriptsize\bfseries, cII!80!black, fill=white, fill opacity=0.78,
      text opacity=1, inner sep=1pt] at (2.708,0.3) {II};
\node[font=\scriptsize\bfseries, cIII!80!black, fill=white, fill opacity=0.78,
      text opacity=1, inner sep=1pt] at (4.125,0.3) {III};
\node[font=\scriptsize\bfseries, cIV!80!black, fill=white, fill opacity=0.78,
      text opacity=1, inner sep=1pt] at (4.70,0.3) {IV};
\node[right=2pt, font=\scriptsize, cI]   at (5, 3.333) {$\rho_1$};
\node[right=2pt, font=\scriptsize, cII]  at (5, 2.10)  {$\rho_2$};
\node[right=2pt, font=\scriptsize, cIII] at (5, 1.55)  {$\rho_3$};
\node[right=2pt, font=\scriptsize, cIV]  at (5, 0.25)  {$\rho_4$};
\node[above=3pt, font=\small\bfseries] at (2.5,3.8)
  {$s=1,\ d=8$\ \ {\normalfont\small($s<d/4$)}};
\begin{scope}[xshift=7.8cm]
\fill[phaseIbg]   (0,0)   rectangle (2.5,3.8);
\fill[phaseIIIbg] (2.5,0) rectangle (3.75,3.8);
\fill[phaseIVbg]  (3.75,0) rectangle (5,3.8);
\begin{scope}
  \clip (0,0) rectangle (5,3.8);
  \draw[black!25, thin] (2.5,0)  -- (2.5,3.8);
  \draw[black!25, thin] (3.75,0) -- (3.75,3.8);
  \draw[cI,   line width=2.2pt] (0,3.333)   -- (2.5,3.333);
  \draw[cIII, line width=2.2pt] (2.5,3.333) -- (3.75,2.5);
  \draw[cIV,  line width=2.2pt] (3.75,2.5)  -- (5,0);
  \draw[rhoOne] (0,3.333) -- (5,3.333);
  \draw[rhoTwo] (2.4,3.8) -- (5,2.5);
  \draw[rhoThree] (1.8,3.8) -- (5,1.667);
  \draw[rhoFour]  (0,10)    -- (5,0);
  \filldraw[fill=white, draw=black!70, line width=0.8pt] (2.5,3.333) circle (2pt);
  \filldraw[fill=white, draw=black!70, line width=0.8pt] (3.75,2.5)  circle (2pt);
\end{scope}
\draw[black!35] (0,0) rectangle (5,3.8);
\draw[->, black!75, line width=0.6pt] (0,0) -- (5.5,0)
  node[right, font=\scriptsize] {$\alpha$};
\draw[->, black!75, line width=0.6pt] (0,0) -- (0,4.1)
  node[above, font=\scriptsize, xshift=4pt] {$\rho(\alpha)$};
\foreach \x/\lbl in {0/{$0$}, 2.5/{$\tfrac{1}{2}$}, 3.75/{$\tfrac{3}{4}$}, 5/{$1$}}{
  \draw[black!70] (\x,-0.08) -- (\x,0.08);
  \node[below=3pt, font=\scriptsize] at (\x,0) {\lbl};
}
\foreach \y/\lbl in {1.667/{$\tfrac{1}{6}$}, 2.5/{$\tfrac{1}{4}$}, 3.333/{$\tfrac{1}{3}$}}{
  \draw[black!70] (-0.08,\y) -- (0.08,\y);
  \node[left=3pt, font=\scriptsize] at (0,\y) {\lbl};
}
\node[font=\scriptsize\bfseries, cI!80!black, fill=white, fill opacity=0.78,
      text opacity=1, inner sep=1pt] at (1.25,0.3) {I};
\node[font=\scriptsize\bfseries, cIII!80!black, fill=white, fill opacity=0.78,
      text opacity=1, inner sep=1pt] at (3.125,0.3) {III};
\node[font=\scriptsize\bfseries, cIV!80!black, fill=white, fill opacity=0.78,
      text opacity=1, inner sep=1pt] at (4.375,0.3) {IV};
\node[right=2pt, font=\scriptsize, cI]       at (5, 3.333) {$\rho_1$};
\node[right=2pt, font=\scriptsize, cII]      at (5, 2.62)  {$\rho_2$};
\node[right=2pt, font=\scriptsize, cIII]     at (5, 1.55)  {$\rho_3$};
\node[right=2pt, font=\scriptsize, cIV]      at (5, 0.17)  {$\rho_4$};
\node[above=3pt, font=\small\bfseries] at (2.5,3.8)
  {$s=2,\ d=4$\ \ {\normalfont\small($s\ge d/4$)}};
\end{scope}
\end{tikzpicture}
\caption{Phase structure of the minimax rate $r^*_{\mathrm{2samp}}(N,\varepsilon)\asymp N^{-\rho(\alpha)}$ under $\varepsilon=N^{-\alpha}$. The bold curve is the lower envelope $\rho(\alpha)=\min_i\rho_i(\alpha)$.  Thin dashed, dash-dotted, or dotted curves show the individual exponent functions, and shaded patterned regions indicate the active regimes. Left ($s=1,d=8$, $s<d/4$): all four exponents attain the envelope, producing four phases with transitions at $\alpha=\tfrac13,\tfrac34,\tfrac9{10}$. Right ($s=2,d=4$, $s\ge d/4$): $\rho_2$ is inactive, so Phase~II disappears and the remaining transitions occur at $\alpha=\tfrac12$ and $\alpha=\tfrac34$.}
\label{fig:phase-envelope}
\end{figure}
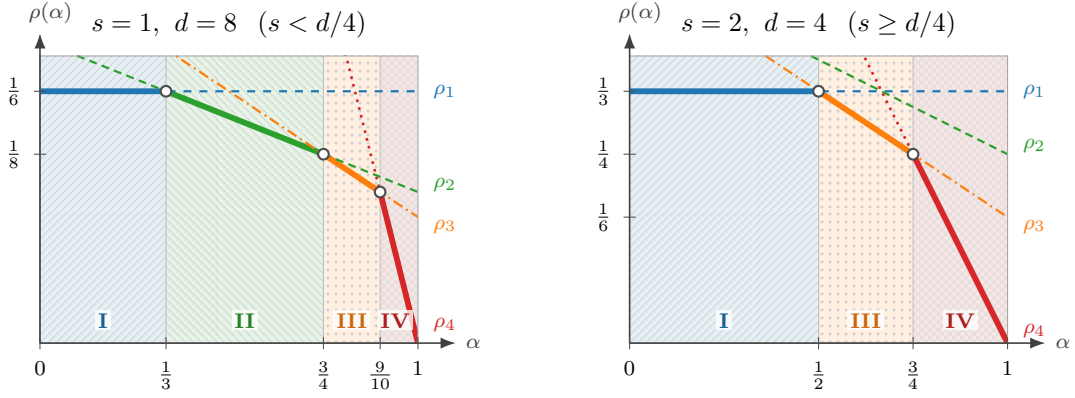

\section{Upper Bounds}
\label{sec:ub}
We prove the upper bound in \Cref{thm:main-rate} by binning the samples and applying a private discrete two-sample test to the resulting binned data.  The proof separates the discrete testing problem from the approximation step.  On the discrete side, we exploit the fact that bounded densities induce histograms with uniformly bounded cell masses, which allows a private tester with a sharper dependence than in the unrestricted discrete problem.  On the continuous side, a binning lemma transfers separation between smooth densities to separation between their histograms, up to an approximation loss. Balancing the binning bias against the sampling and privacy costs of the discrete test gives the rate in \Cref{thm:main-rate}.

We first establish the discrete testing result, then transfer it to smooth densities by binning, and finally give a Rao--Blackwellized implementation. Throughout this section, we keep the notation $N=N_X\wedge N_Y$. In the balanced split construction below, $n$ denotes the size of each artificial half; thus $2n\le N$ observations from each sample are used before splitting. We write $k$ for the discrete domain size. In the continuous binned problem, $k=\kappa^d$, where $\kappa$ is the side resolution of the regular grid.

\subsection{Private Two-Sample Testing for Discrete Distributions}
\label{sec:ub-discrete}

We begin with a balanced multinomial model, which captures the main technical ingredients for the upper bound.  For $p,q\in\Delta_k$, suppose that, for notational simplicity, we observe $2n$ labels from each distribution:
\begin{align*}
S^X_1,\ldots,S^X_{2n}\stackrel{\iid}{\sim} p,
\qquad
S^Y_1,\ldots,S^Y_{2n}\stackrel{\iid}{\sim} q.
\end{align*}
For unequal original sample sizes, we may choose $2n\le N_X\wedge N_Y$ and discard the remaining observations; this affects only constant factors in the final upper bounds.  A more refined version that uses all observations efficiently is discussed in \Cref{sec:rb-implementation}.

The upper bound uses discrete closeness testing after binning.  For reference, in the unrestricted discrete problem, the optimal central-DP sample complexity, up to constant factors, is \citep[Theorem~7]{Zhang2021}
\begin{align*}
\underbrace{\frac{k^{1/2}}{\tau^2}
+\frac{k^{2/3}}{\tau^{4/3}}}_{\text{nonprivate}}
+\underbrace{
\frac{\sqrt{k}}{\tau\sqrt{\varepsilon}}
+\frac{k^{1/3}}{\tau^{4/3}\varepsilon^{2/3}}
+\frac{1}{\tau\varepsilon}
}_{\text{privacy-driven}},
\end{align*}
where $\tau$ is the discrete $\ell_1$ separation parameter and $k$ is the number of bins. Crucially, the binned distributions arising in the upper-bound argument belong to a structured subclass of $\Delta_k$. If the density satisfies $f\le M$ and the domain is partitioned into $k$ equal-volume cells, then every induced cell probability is at most $M/k$.  Thus binning maps bounded densities into the flat discrete subclass
\begin{align*}
\mathcal{P}_{k,M}
:=
\{v\in\Delta_k:\|v\|_\infty\le M/k\}.
\end{align*}
On this subclass, the heavy-element contribution $k^{2/3}/\tau^{4/3}$ from unrestricted closeness testing is no longer the relevant nonprivate bottleneck.  Uniformly over $p,q\in\mathcal P_{k,M}$, the split statistic has nonprivate cost $\sqrt{k}/\tau^2$, while retaining the same privacy-dependent costs as in the unrestricted problem.  This bounded-histogram refinement is the discrete testing ingredient behind the continuous upper bound developed in \Cref{sec:ub-binning}.

One technical distinction of our analysis is that it remains entirely within the multinomial model, avoiding the Poissonization used in prior nonprivate and private closeness-testing analyses \citep{DiakonikolasGouleakisKanePeeblesPrice2021,Zhang2021}. The closest related multinomial analysis is due to \citet{CanonneSun2022}, whose coordinatewise binomial comparison argument assumes $p,q\in[0,1/4]$. We prove a version valid for all $p,q\in[0,1]$, adapted to our normalization, and then specialize it to the bounded-histogram subclass $\mathcal P_{k,M}$. See \Cref{sec:expectation_gap}.

\paragraph{Split Statistic and Monte Carlo Permutation Test.}
At the discrete level, we use a private version of the split closeness statistic. Given two samples from $p$ and $q$, each with $2n$ observations, we randomly split each sample into two halves and let $C^{X,1},C^{X,2},C^{Y,1},C^{Y,2}\in\mathbb N^k$ be the four resulting histograms. Each histogram is based on $n$ observations. We then compute
\begin{align*}
\begin{aligned}
Z_n
:=
\sum_{i=1}^{k}
\bigl\{
|C_i^{X,1}-C_i^{Y,1}|
+
|C_i^{X,2}-C_i^{Y,2}|
-
|C_i^{X,1}-C_i^{X,2}|
-
|C_i^{Y,1}-C_i^{Y,2}|
\bigr\},
\end{aligned}
\end{align*}
the split statistic introduced by \citet{DiakonikolasGouleakisKanePeeblesPrice2021} for nonprivate two-sample testing. It compares cross-sample discrepancies with the within-sample fluctuations induced by the random split. The statistic tends to be large when $p$ and $q$ differ, while under $p=q$ its expectation is zero. Moreover, it is stable in the sense that changing one record can change its value by at most four in the worst case. Consequently, only a modest amount of noise is required to privatize the statistic under differential privacy. The corresponding constant-sensitivity property is also the key observation underlying the central-DP multinomial test of \citet{Zhang2021}.  However, because the required calibration constants in that work are not made explicit, the result does not immediately translate into an implementable finite-sample test.  We address this issue by using the Monte Carlo permutation approach.

We now introduce the notation used to implement the random split and its permutation calibration.  For integers $a,b\ge1$, let
\begin{align*}
\mathfrak L_{a,b}
\coloneqq
\left\{
\ell\in\{X,Y\}^{a+b}:
|\ell^{-1}(X)|=a,\ |\ell^{-1}(Y)|=b
\right\}
\end{align*}
be the set of two-sample labelings with $a$ labels $X$ and $b$ labels $Y$, where $\ell^{-1}(G):=\{r:\ell(r)=G\}$.  In the balanced split construction, $\mathfrak L_{2n,2n}$ is the set of balanced two-sample labelings of the $4n$ observations in the pooled sample.  For $\ell\in\mathfrak L_{2n,2n}$, let $\Lambda_n(\ell)$ be the set of four-block refinements $\lambda=(A_{X,1},A_{X,2},A_{Y,1},A_{Y,2})$ such that, for each group indicator $G\in\{X,Y\}$, the sets $A_{G,1}$ and $A_{G,2}$ are disjoint subsets of $\ell^{-1}(G)$ and both have size $n$.  Given the pooled discrete sample
\begin{align*}
W:=(W_1,\ldots,W_{4n}) \coloneqq (S_1^X,\ldots,S_{2n}^X,S_1^Y,\ldots,S_{2n}^Y)\in[k]^{4n},
\end{align*}
and $\lambda\in\Lambda_n(\ell)$, define
\begin{align*}
C_i^{G,a}(W,\lambda)
:=
\sum_{j=1}^{4n}
\mathbf 1\{W_j=i,\ j\in A_{G,a}\},
\qquad i\in[k].
\end{align*}
The corresponding split statistic is
\begin{align*}
\begin{aligned}
Z_n(W,\lambda)
&:=
\sum_{i=1}^{k}
\Bigl\{
|C_i^{X,1}(W,\lambda)-C_i^{Y,1}(W,\lambda)|
+|C_i^{X,2}(W,\lambda)-C_i^{Y,2}(W,\lambda)|
\\
&\qquad
-|C_i^{X,1}(W,\lambda)-C_i^{X,2}(W,\lambda)|
-|C_i^{Y,1}(W,\lambda)-C_i^{Y,2}(W,\lambda)|
\Bigr\}.
\end{aligned}
\end{align*}
Let $\ell^{(0)}\in\mathfrak L_{2n,2n}$ denote the observed two-sample labeling induced by the pooled sample $W$. Thus the statistic $Z_n$ introduced earlier is $Z_n(W,\lambda^{(0)})$ when $\lambda^{(0)}\sim\Unif(\Lambda_n(\ell^{(0)}))$ is the random split of the observed two-sample labeling.

With this notation in place, we apply the pure-DP permutation calibration of \citet{KimSchrab2026}, a private extension of the classical permutation test for statistics with bounded sensitivity.  The test adds independent Laplace noise, with a common scale, to the observed statistic and to the statistics computed from randomly permuted data, and then uses the resulting permutation p-value. The complete procedure is summarized in \Cref{alg:mcperm}.

\begin{algorithm}[t]
\caption{Discrete DP Monte Carlo permutation test
}
\label{alg:mcperm}
 \KwIn{Two samples of discrete labels $S^X_1,\ldots,S^X_{2n}$ and
$S^Y_1,\ldots,S^Y_{2n}$ in $[k]$. Privacy budget
$\varepsilon$, level $\gamma$, number of Monte Carlo relabelings
$B_{\mathrm{perm}}$.}
Let $W=(W_1,\ldots,W_{4n})$ be the pooled discrete labels obtained by
concatenating the two input samples, and let
$\ell^{(0)}\in\mathfrak L_{2n,2n}$ be the induced two-sample labeling\;
Draw $\lambda^{(0)}\sim\Unif(\Lambda_n(\ell^{(0)}))$\;
Compute $T_0:=Z_n(W,\lambda^{(0)})$\;
Draw $\zeta_0,\zeta_1,\dots,\zeta_{B_{\mathrm{perm}}}
\stackrel{\iid}{\sim}\Lap(1)$ and set
$M_0:=T_0+(8/\varepsilon)\zeta_0$\;
\For{$b=1,\dots,B_{\mathrm{perm}}$}{
  Draw $\ell^{(b)}\sim\Unif(\mathfrak L_{2n,2n})$, and then draw
  $\lambda^{(b)}\sim\Unif(\Lambda_n(\ell^{(b)}))$\;
  Compute $T_b:=Z_n(W,\lambda^{(b)})$\;
  Set $M_b:=T_b+(8/\varepsilon)\zeta_b$\;
}
Set
$\widehat p_{\mathrm{dpperm}}=
\{1+\sum_{b=1}^{B_{\mathrm{perm}}}\mathbf{1}\{M_b\ge M_0\}\}/(B_{\mathrm{perm}}+1)$\;
Reject $H_0$ if  $\widehat p_{\mathrm{dpperm}}\le\gamma$\;
\end{algorithm}

Let $\Psi_{\mathrm{dpperm},\gamma}$ denote the binary decision returned by the DP permutation test in \Cref{alg:mcperm}.  The next proposition records its finite-sample validity and privacy guarantees.

\begin{proposition}[Validity and privacy of the DP permutation test]
\label{prop:dpperm-validity-privacy}
For every $\varepsilon\in(0,\infty]$, $\gamma\in[0,1]$, and every $B_{\mathrm{perm}}\ge 1$, the test $\Psi_{\mathrm{dpperm},\gamma}$ is $\varepsilon$-differentially private and satisfies
\begin{align*}
\sup_{p\in\Delta_k}
\Pbb_{p,p}\bigl(\Psi_{\mathrm{dpperm},\gamma}=1\bigr)
\le \gamma.
\end{align*}
\end{proposition}

\begin{proof}
The proof is in \Cref{app:proof:prop:dpperm-validity-privacy}.
\end{proof}
This proposition is a direct application of \citet[Theorems~1--2]{KimSchrab2026} to $Z_n$ via \Cref{lem:private-permutation-cited}; the sensitivity bound $\GS(Z_n)\le4$ is established in \Cref{lem:mult-bdhist-sensitivity}. We next record the corresponding power guarantee, which will serve as the discrete testing building block for the continuous upper bound: over the bounded-histogram class $\mathcal P_{k,M}$, the DP permutation calibration attains the private closeness bound obtained from the unrestricted rate after removing the nonprivate heavy-element term $k^{2/3}/\tau^{4/3}$.

\begin{proposition}[Bounded-histogram closeness testing by DP permutation]
\label{prop:private-bounded-hist-perm}
Fix $M\ge 1$, $\varepsilon\in(0,\infty]$, and set $ B_{\mathrm{perm}} = \bigl\lceil 6\gamma^{-1}\log(2\beta^{-1})\bigr\rceil.$ There exists $C=C(M,\gamma,\beta)>0$ such that the following holds. For every $\tau\in(0,2]$, suppose that
\begin{align*}
n \ge C
\left[
\frac{\sqrt{k}}{\tau^2}
+
\frac{\sqrt{k}}{\tau\sqrt{\varepsilon}}
+
\frac{k^{1/3}}{\tau^{4/3}\varepsilon^{2/3}}
+
\frac{1}{\tau\varepsilon}
\right].
\end{align*}
Then the test $\Psi_{\mathrm{dpperm},\gamma}$ has type~II error at most $\beta$, uniformly over all $p,q\in\mathcal P_{k,M}$ with $\|p-q\|_1\ge\tau$.
\end{proposition}

\begin{proof}
The proof is in \Cref{app:proofs-private-bounded-hist-perm}.
\end{proof}

We now instantiate \Cref{alg:mcperm} for continuous samples.  Binning loses signal through approximation bias but converts the problem into the bounded histogram setting just described.  Thus, after the discrete tester is established, the continuous upper bound reduces to choosing a resolution $\kappa$ that balances approximation bias against discrete testing cost.

\subsection{Binned Test and Upper Bounds}
\label{sec:ub-binning}

Fix a resolution $\kappa$, and let $\mathrm{cell}_\kappa$ be the regular partition map from $[0,1]^d$ to $[\kappa^d]$, obtained by dividing the cube into $\kappa^d$ equal-volume cells.  We convert the continuous observations into discrete labels by setting
\begin{align*}
\widetilde X_j^{(\kappa)}=\mathrm{cell}_\kappa(X_j),
\qquad
\widetilde Y_j^{(\kappa)}=\mathrm{cell}_\kappa(Y_j),
\end{align*}
so that $\widetilde X_j^{(\kappa)},\widetilde Y_j^{(\kappa)}\in[\kappa^d]$. Using $2n$ observations from each sample and binning them at resolution $\kappa$ gives two label samples in $[\kappa^d]$.  The random split into four blocks of size $n$ is then performed internally by \Cref{alg:mcperm}. Thus the continuous test at resolution $\kappa$ is the discrete DP permutation test with $k=\kappa^d$, preceded by a deterministic binning step.

It remains to check that this binning step places the induced discrete problem within the bounded-histogram setting of the previous subsection, and that it does not lose too much $L_1$ separation. To keep track of the uniform bound entering the discrete tester, we introduce the bounded smooth density class
\begin{align*}
\cF_s^d(L,M)
=
\left\{
f:[0,1]^d\to[0,\infty):
\int_{[0,1]^d} f(x)\,dx=1,\;
f\in\cH_s^d(L),\;
\|f\|_\infty\le M
\right\}.
\end{align*}
For $f\in\cF_s^d(L,M)$, let $p_f^{(\kappa)}\in\Delta_{\kappa^d}$ denote the cell-probability vector induced by binning at resolution $\kappa$. Since every cell has volume $\kappa^{-d}$ and $\|f\|_\infty\le M$, we immediately have $ p_f^{(\kappa)}\in\mathcal P_{\kappa^d,M}. $ The remaining ingredient is to control the approximation error introduced by binning. By the uniform binning approximation lemma (\Cref{lem:binning-approx}), with $C_{\mathrm{bin},s}:=C_{\mathrm{bin}}(d,s,s,L)$,
\begin{align*}
\|p_f^{(\kappa)}-p_g^{(\kappa)}\|_1
\ge
\frac12\|f-g\|_1-C_{\mathrm{bin},s}\kappa^{-s}.
\end{align*}
Hence, whenever $\|f-g\|_1\ge r$, the corresponding binned distributions satisfy
\begin{align*}
\|p_f^{(\kappa)}-p_g^{(\kappa)}\|_1
\ge
\tau_\kappa(r)
:=
\frac r2-C_{\mathrm{bin},s}\kappa^{-s}.
\end{align*}
Applying \Cref{prop:private-bounded-hist-perm} with $k=\kappa^d$ and $\tau=\bar\tau_\kappa(r):=\min\{1,\tau_\kappa(r)\}$ therefore yields the following resolution-dependent guarantee whenever $\bar\tau_\kappa(r)>0$.

\begin{theorem}[Private upper bound over $\cF_s^d(L,M)$]
\label{thm:holder-l1-private-upper-d}
Fix $d\in\mathbb{N}$, $s>0$, $L>0$, and $M\ge 1$, and let $N=N_X\wedge N_Y$ and $C_{\mathrm{bin},s}:=C_{\mathrm{bin}}(d,s,s,L)$. There exists $C=C(d,s,L,M,\gamma,\beta)>0$ such that the following holds. For every $r\in(0,2]$, $\varepsilon>0$, and $\kappa\in\mathbb N$, define
\begin{align*}
\bar\tau_\kappa(r)
:=
\min\left\{1,\,\frac r2-C_{\mathrm{bin},s}\kappa^{-s}\right\}.
\end{align*}
If $\bar\tau_\kappa(r)>0$ and
\begin{equation}
N \ge C
\left(
\frac{\kappa^{d/2}}{\bar\tau_\kappa(r)^2}
+
\frac{\kappa^{d/2}}{\bar\tau_\kappa(r)\sqrt{\varepsilon}}
+
\frac{\kappa^{d/3}}{\bar\tau_\kappa(r)^{4/3}\varepsilon^{2/3}}
+
\frac{1}{\bar\tau_\kappa(r)\varepsilon}
\right),
\label{eq:holder-private-sample-cond-d}
\end{equation}
then the binned version of \Cref{alg:mcperm} with resolution $\kappa$ is an $\varepsilon$-DP test of $H_0:f=g$ against $H_1:\|f-g\|_1\ge r$ over $\cF_s^d(L,M)$, with type~I error at most $\gamma$ and type~II error at most $\beta$. Moreover, for another constant $C'=C'(d,s,L,M,\gamma,\beta)>0$, there exists a choice of resolution $\kappa$ for which the same guarantee holds whenever
\begin{equation}
r
\ge
C'
\inf_{\kappa\in\mathbb{N}}
\left\{
\kappa^{-s}
+
\frac{\kappa^{d/4}}{\sqrt{N}}
+
\frac{\kappa^{d/2}}{N\sqrt{\varepsilon}}
+
\frac{\kappa^{d/4}}{N^{3/4}\sqrt{\varepsilon}}
+
\frac{1}{N\varepsilon}
\right\}.
\label{eq:holder-private-separation-d}
\end{equation}
\end{theorem}

\begin{proof}
The proof is in \Cref{app:proofs-continuous-upper}.
\end{proof}

The five terms in \eqref{eq:holder-private-separation-d} have simple interpretations. Binning at resolution $\kappa$ introduces the binning bias $\kappa^{-s}$. After binning, the problem has $k=\kappa^d$ cells, and the usual nonprivate sampling fluctuation contributes $\kappa^{d/4}/\sqrt N$. Privacy adds the next two $\kappa$-dependent terms, which come from the interaction of the number of cells, sampling noise, and the Laplace noise used for private calibration. The last term, $(N\varepsilon)^{-1}$, does not depend on $\kappa$ and is the irreducible linear privacy cost. Thus choosing $\kappa$ amounts to balancing the binning bias against the three $\kappa$-dependent testing costs.

\paragraph{Bandwidth Optimization and Explicit Rate}

Optimizing the bound in \Cref{thm:holder-l1-private-upper-d} yields the rate in \Cref{thm:main-rate}. Balancing the bias $\kappa^{-s}$ with the three $\kappa$-dependent stochastic and privacy terms leads to the resolutions
\begin{align*}
\kappa_1\asymp N^{2/(4s+d)},\qquad
\kappa_2\asymp (N\sqrt\varepsilon)^{2/(2s+d)},\qquad
\kappa_3\asymp (N^{3/2}\varepsilon)^{2/(4s+d)},
\end{align*}
which produce the corresponding separation scales
\begin{align*}
N^{-2s/(4s+d)},\qquad
(N\sqrt{\varepsilon})^{-2s/(2s+d)},\qquad
(N^{3/2}\varepsilon)^{-2s/(4s+d)}.
\end{align*}
The remaining term $(N\varepsilon)^{-1}$ is independent of $\kappa$. We summarize the resulting rate in the following corollary.

\begin{corollary}[Explicit binned private upper rate]
\label{cor:holder-l1-private-explicit-rate}
Under the assumptions of \Cref{thm:holder-l1-private-upper-d}, there exists $C=C(d,s,L,M,\gamma,\beta)>0$ such that the same guarantee holds whenever
\begin{equation}
r
\ge
C\left[
N^{-2s/(4s+d)}
\vee
(N\sqrt{\varepsilon})^{-2s/(2s+d)}
\vee
(N^{3/2}\varepsilon)^{-2s/(4s+d)}
\vee
(N\varepsilon)^{-1}
\right].
\label{eq:explicit-holder-private-rate-max}
\end{equation}
Consequently, since $\cD_s^d(L)\subset\cF_s^d(L,L)$, the same rate is achieved over $\cD_s^d(L)$.
\end{corollary}

\begin{proof}
The proof is in \Cref{app:proof:cor-holder-private-explicit-rate}.
\end{proof}

The preceding corollary completes the analysis of the split-sample test. This construction, however, uses only a balanced subsample of size $2n\le N_X\wedge N_Y$ from each sample and relies on an artificial random split, which can degrade finite-sample performance. The next subsection removes both limitations by Rao--Blackwellization. The resulting statistic uses all binned observations, handles unequal sample sizes directly, and admits an exact finite-sum representation that can be evaluated in $O(N_X+N_Y)$ unit-cost arithmetic operations once the contingency table is available. This refinement requires a more involved construction but preserves the same upper-bound guarantees.

\subsection{Rao--Blackwellized Statistic and Guarantees}
\label{sec:rb-implementation}

We now replace the split statistic by its Rao--Blackwellized analogue, obtained by averaging over the auxiliary splitting randomness.  This removes the artificial split from the implemented statistic and improves finite-sample performance, while leaving the rate guarantees unchanged. We use the same labeling and refinement notation as in \Cref{sec:ub-discrete}, with the following extension to arbitrary sample sizes. Throughout this subsection the binning resolution is fixed, so we suppress its dependence in the notation. For the pooled binned sample
\begin{align*}
W:=(\widetilde X_1,\ldots,\widetilde X_{N_X},
\widetilde Y_1,\ldots,\widetilde Y_{N_Y})
\in[k]^{N_X+N_Y},
\end{align*}
let $\ell^{(0)}$ denote the observed two-sample labeling, and set $m=\left\lfloor (N_X\wedge N_Y)/2\right\rfloor.$ For $\ell\in\mathfrak L_{N_X,N_Y}$, let $\Lambda_m(\ell)$ denote the corresponding collection of four-block refinements, with $m$ observations in each of the blocks $(X,1),(X,2),(Y,1),(Y,2)$.  For $\lambda\in\Lambda_m(\ell)$, write $Z_m(W,\lambda)$ for the split statistic defined in \Cref{sec:ub-discrete}, with block size $m$.

The Rao--Blackwellized statistic associated with a labeling $\ell$ is
\begin{align*}
\overline Z_m(W,\ell)
:=
\E_{\lambda\sim\Unif(\Lambda_m(\ell))}
\bigl[
Z_m(W,\lambda)
\bigr],
\end{align*}
and the statistic used on the observed data is
\begin{align*}
\overline Z
:=
\overline Z_m(W,\ell^{(0)}).
\end{align*}
When $N_X=N_Y=2m$, $\overline Z$ is the conditional mean, given $(W,\ell^{(0)})$, of the ordinary randomly split statistic used in \Cref{alg:mcperm}.  For general $N_X,N_Y$, the same formula averages over all balanced refinements in $\Lambda_m(\ell)$, so the implemented statistic does not depend on a particular split or on a fixed choice of unused observations.  We then use this averaged statistic in the same private permutation-calibration framework, replacing each single-split value $Z_m(W,\lambda)$ by $\overline Z_m(W,\ell)$.  The resulting procedure is summarized in \Cref{alg:rb-dpperm}.

\begin{algorithm}[!t]
\caption{Rao--Blackwellized DP Monte Carlo permutation test}
\label{alg:rb-dpperm}
\KwIn{Pooled binned labels $W=(W_1,\ldots,W_{N_X+N_Y})$, observed two-sample
labeling $\ell^{(0)}$ with class sizes $N_X,N_Y$. Privacy budget
$\varepsilon$, level $\gamma$, number of Monte Carlo relabelings
$B_{\mathrm{perm}}$.}
Set $m=\lfloor (N_X\wedge N_Y)/2\rfloor$\;
Compute $T_0:=\overline Z_m(W,\ell^{(0)})$, where $\overline Z_m$ is the
Rao--Blackwellized statistic obtained by averaging the split statistic over all
admissible internal splits of the two label classes\;
Draw $\zeta_0,\zeta_1,\dots,\zeta_{B_{\mathrm{perm}}}
\stackrel{\iid}{\sim}\Lap(1)$ and set
$M_0:=T_0+(8/\varepsilon)\zeta_0$\;
\For{$b=1,\dots,B_{\mathrm{perm}}$}{
  Draw $\ell^{(b)}\sim\Unif(\mathfrak L_{N_X,N_Y})$ independently\;
  Compute $T_b:=\overline Z_m(W,\ell^{(b)})$\;
  Set $M_b:=T_b+(8/\varepsilon)\zeta_b$\;
}
Set
$\widehat p_{\mathrm{rb\text{-}dpperm}}=
\{1+\sum_{b=1}^{B_{\mathrm{perm}}}\mathbf{1}\{M_b\ge M_0\}\}/(B_{\mathrm{perm}}+1)$\;
Reject $H_0$ if  $\widehat p_{\mathrm{rb\text{-}dpperm}}\le\gamma$\;
\end{algorithm}

Let $\Psi_{\mathrm{rb\text{-}dpperm},\gamma}$ denote the binary decision returned by the Rao--Blackwellized DP permutation test in \Cref{alg:rb-dpperm}.

Although the definition of $\overline Z$ involves an average over many refinements, it admits an exact linear-time implementation from the nonzero full-sample bin-count pairs.  The closed form and algorithmic derivation are given in \Cref{app:proof:prop:rb-linear-time-main-text}.

\begin{proposition}[Exact linear-time Rao--Blackwellized evaluation]
\label{prop:rb-linear-time-main-text}
Given two full binned samples of sizes $N_X$ and $N_Y$, equivalently the contingency table of complete bin-count pairs represented by its nonzero entries, the statistic $\overline Z$ can be computed exactly using $O(N_X+N_Y)$ operations and $O(N_X+N_Y)$ memory.
\end{proposition}

\begin{proof}
The proof is in \Cref{app:proof:prop:rb-linear-time-main-text}.
\end{proof}

For the guarantees, two facts are needed.  First, Rao--Blackwellization does not affect the privacy or null-calibration argument.  If, conditional on $(W,\ell^{(0)})$, $\lambda^{(0)}\sim\Unif(\Lambda_m(\ell^{(0)}))$, then
\begin{align*}
\overline Z_m(W,\ell^{(0)})
=
\E\!\left[
Z_m(W,\lambda^{(0)})
\mid W,\ell^{(0)}
\right].
\end{align*}
Thus the averaged statistic has the same mean as the single-split statistic, and its variance does not exceed that of the single-split statistic.  Moreover, since $Z_m(W,\lambda)$ has global sensitivity at most four uniformly over $\lambda$, the same bound is inherited by the average. For $w,w'\in[k]^{N_X+N_Y}$, write
\begin{align*}
d_H(w,w'):=\sum_{i=1}^{N_X+N_Y}\mathbf 1\{w_i\ne w_i'\}
\end{align*}
for their Hamming distance. Then
\begin{align*}
\sup_{\substack{w,w'\in[k]^{N_X+N_Y}\\ d_H(w,w')\le1}}
\left|
\overline Z_m(w,\ell)-\overline Z_m(w',\ell)
\right|
\le 4
\end{align*}
for every admissible labeling $\ell$.  Hence the private permutation calibration applies with the same Laplace scale after replacing $Z_m(W,\lambda)$ by $\overline Z_m(W,\ell)$.  These elementary facts are recorded in \Cref{lem:rb-basic-properties}.

The second fact is the corresponding observed-versus-permuted moment comparison. Let $\ell^\pi$ be a uniformly random labeling with $N_X$ labels $X$ and $N_Y$ labels $Y$, independent of the data, and write $\overline Z^{\pi}:=\overline Z_m(W,\ell^\pi)$ and $\overline Z^{(0)}:=\overline Z_m(W,\ell^{(0)})$. With $m=\lfloor (N_X\wedge N_Y)/2\rfloor$, \Cref{lem:rb-perm-moment-bounds} shows that, for arbitrary $p,q\in\Delta_k$, the permuted statistic is conditionally centered, $\E[\overline Z^\pi\mid W]=0$, and, for a universal constant $C>0$, $\Var(\overline Z^{(0)})+\Var(\overline Z^\pi)\le Cm$. Expectation preservation, the coordinatewise comparison in \Cref{lem:binomial-comparison-no-heavy}, and an aggregation of the resulting bounds imply that, whenever $p,q\in\Delta_k$ satisfy $\|p-q\|_1\ge\tau$, the observed mean satisfies the bound below for a universal constant $c>0$:
\begin{align*}
\E_{p,q}\!\left[\overline Z^{(0)}\right]
\ge c\min\!\left\{
m\tau,\frac{m^2\tau^2}{k},\frac{m^{3/2}\tau^2}{\sqrt{k}}
\right\}.
\end{align*}
If, in addition, $p,q\in\mathcal P_{k,M}$, the variance satisfies the improved bound below for a constant $C_M>0$ depending only on $M$:
\begin{align*}
\Var(\overline Z^{(0)})+\Var(\overline Z^\pi)
\le C_M\min\!\left\{\frac{m^2}{k},m\right\}.
\end{align*}
These Rao--Blackwellized moment bounds yield both the optimal unrestricted private closeness rate and its bounded-histogram refinement. In the unrestricted simplex, combining the distribution-free $\sqrt m$ fluctuation scale with the $m^2\tau^2/k$ signal regime yields the classical heavy-element term $k^{2/3}/\tau^{4/3}$. For $p,q\in\mathcal P_{k,M}$, the sharper variance bound eliminates this heavy-element term from the sample-complexity requirement.

\begin{proposition}[Rao--Blackwellized private closeness testing]
\label{prop:private-closeness-rb}
\label{prop:private-bounded-hist-rb}
Fix $\varepsilon\in(0,\infty]$ and $\gamma,\beta\in(0,1)$ with $\gamma+\beta<1$, and set $B_{\mathrm{perm}}=\lceil 6\gamma^{-1}\log(2\beta^{-1})\rceil$. For every $k\ge2$ and sample sizes $N_X,N_Y$, the Rao--Blackwellized permutation test in \Cref{alg:rb-dpperm} is $\varepsilon$-differentially private and satisfies
\begin{align*}
\sup_{p\in\Delta_k}
\Pbb_{p,p}\!\left(\Psi_{\mathrm{rb\text{-}dpperm},\gamma}=1\right)
\le\gamma.
\end{align*}
Let $N=N_X\wedge N_Y$. The following power guarantees hold.

\begin{enumerate}[(i)]
\item \textbf{Unrestricted simplex.} There exists $C=C(\gamma,\beta)>0$ such that, for every $\tau\in(0,2]$, the type~II error is at most $\beta$, uniformly over all $p,q\in\Delta_k$ satisfying $\|p-q\|_1\ge\tau$, whenever
\begin{align}
N \ge C \left[
\frac{\sqrt{k}}{\tau^2}
+ \frac{k^{2/3}}{\tau^{4/3}}
+ \frac{\sqrt{k}}{\tau\sqrt{\varepsilon}}
+ \frac{k^{1/3}}{\tau^{4/3}\varepsilon^{2/3}}
+ \frac{1}{\tau\varepsilon}
\right].
\label{eq:rb-unrestricted-private-rate}
\end{align}

\item \textbf{Bounded-histogram refinement.} For every $M\ge1$, there exists $C_M=C_M(M,\gamma,\beta)>0$ such that, for every $\tau\in(0,2]$, the type~II error is at most $\beta$, uniformly over all $p,q\in\mathcal P_{k,M}$ satisfying $\|p-q\|_1\ge\tau$, whenever
\begin{align}
N \ge C_M \left[
\frac{\sqrt{k}}{\tau^2}
+ \frac{\sqrt{k}}{\tau\sqrt{\varepsilon}}
+ \frac{k^{1/3}}{\tau^{4/3}\varepsilon^{2/3}}
+ \frac{1}{\tau\varepsilon}
\right].
\label{eq:rb-flat-private-rate}
\end{align}
\end{enumerate}
\end{proposition}

\begin{proof}
The proof is in \Cref{app:proof:prop:private-closeness-rb}.
\end{proof}

When $N_X=N_Y$, the sample-size rate in part~{\rm (i)} matches, up to constants, the lower bound for unrestricted central-DP closeness testing in \citet[][Theorem~7]{Zhang2021}, and is therefore minimax optimal over $\Delta_k$. Beyond this rate optimality, the test is directly implementable: its calibration involves no unspecified constants, and its permutation $p$-value provides finite-sample level control. To our knowledge, this is the first central-DP test for discrete $L_1$ two-sample closeness to combine these features with minimax-optimal sample complexity.

Having established part~{\rm (ii)} of \Cref{prop:private-closeness-rb}, the binning and bandwidth-optimization arguments of \Cref{thm:holder-l1-private-upper-d,cor:holder-l1-private-explicit-rate} apply with $N=N_X\wedge N_Y$. Thus the Rao--Blackwellized procedure attains the same minimax upper rate while using all observations and avoiding the auxiliary split. Furthermore, Rao--Blackwellization reduces variance and can improve finite-sample power, as illustrated in \Cref{sec:simulation}.

\section{Lower Bounds}
\label{sec:lb}

The lower bound is proved by reducing to a one-sample goodness-of-fit problem against the uniform density.  This GOF problem already contains the four lower-bound mechanisms underlying \eqref{eq:rate-combined}. This section gives a high-level explanation of these mechanisms; the detailed constructions and scale calculations are deferred to the appendix.

\subsection{Reduction to Uniform Goodness-of-Fit}
\label{sec:lb-overview}

We now formalize the reduction from two-sample testing to uniform-null GOF testing.  Throughout the lower-bound proof, we assume $N\varepsilon\ge1$. As discussed before, when $N\varepsilon<1$, the linear privacy term is already of constant order.

Let $f_{\mathrm{unif}}\equiv 1$ denote the uniform density on $[0,1]^d$, and let $P_0$ denote its law.  For $m\ge1$, write $\E_0^m$ for expectation with respect to $P_0^m$; when the sample size is clear, we abbreviate this expectation by $\E_0$.  The GOF null and alternative are
\begin{align*}
H_0:f\equiv f_{\mathrm{unif}}
\qquad\text{versus}\qquad
H_1:\|f-f_{\mathrm{unif}}\|_1\ge r,\quad f\in\cD_s^d(L),
\end{align*}
and $\Phi_{\gamma,\varepsilon}^{\mathrm{gof}}$ is the class of $\varepsilon$-DP tests $\mathcal M:([0,1]^d)^N\to\{0,1\}$ with type~I error at most $\gamma$ under the uniform null. This is the one-sample analogue of $\Phi_{\gamma,\varepsilon}^{\mathrm{cDP}}$ from \Cref{sec:setup}, with the uniform density as the null.  Define
\begin{align*}
r^*_{\mathrm{gof}}(N,\varepsilon)
:=
\inf\Biggl\{
 r>0:
 \exists\,\mathcal M\in\Phi_{\gamma,\varepsilon}^{\mathrm{gof}}
 \text{ s.t. }
 \sup_{\substack{f\in\cD_s^d(L)\\ \|f-1\|_1\ge r}}
 \Bigl(1-\E_f^N[p_{\mathcal M}]\Bigr)\le\beta
\Biggr\}.
\end{align*}
The following lemma is a central-DP analogue of \citet[][Lemma 1]{AriasCastroPelletierSaligrama2018}. Since the reduction preserves differential privacy, any lower bound for the one-sample GOF problem immediately transfers to the two-sample problem.

\begin{lemma}[GOF-to-two-sample reduction under central DP]
\label{lem:gof-to-twosamp-cdp}
For $N=N_X\wedge N_Y$,
\begin{equation}
\label{eq:twosamp-gof-reduction}
r^*_{\mathrm{2samp}}(N_X,N_Y,\varepsilon)
\ge
r^*_{\mathrm{gof}}(N,\varepsilon).
\end{equation}
\end{lemma}
\begin{proof}
The proof is in \Cref{app:proof:lem:gof-to-twosamp-cdp}.
\end{proof}

We now establish the minimax separation rate for the GOF problem. Combined with \Cref{lem:gof-to-twosamp-cdp}, this immediately gives the lower bound in \Cref{thm:main-rate}.

\begin{theorem}[GOF lower bound]
\label{thm:gof-combined-lower}
Fix $d\ge 1$, $s>0$, $L>1$, and $\gamma,\beta\in(0,1)$ with $\gamma+\beta<1$. There exists $c=c(d,s,L,\gamma,\beta)>0$ such that for all $N\ge 1$ and $\varepsilon>0$ satisfying $N\varepsilon\ge 1$,
\begin{align*}
r^*_{\mathrm{gof}}(N,\varepsilon)
\;\ge\;
c\left[
N^{-2s/(4s+d)}
\vee
(N\sqrt\varepsilon)^{-2s/(2s+d)}
\vee
(N^{3/2}\varepsilon)^{-2s/(4s+d)}
\vee
(N\varepsilon)^{-1}
\right].
\end{align*}
\end{theorem}
\begin{proof}
The proof is in \Cref{app:proof:thm:gof-combined-lower}.
\end{proof}

The remainder of this section explains the four mechanisms responsible for the four terms in the lower bound. The full proofs are deferred to the appendix.

\subsection{Four Lower-Bound Mechanisms}
\label{sec:lb-regimes}

The four terms in \eqref{eq:rate-combined} arise from four distinct lower-bound mechanisms.  The first is the classical nonprivate smooth-testing rate. The other three are privacy effects, consisting of two intermediate losses from within-cell information and private hypercube transport, together with the final linear privacy barrier.

\paragraph{Classical smooth testing.}
The term $N^{-2s/(4s+d)}$ is the classical Ingster-type barrier for smooth nonparametric testing.  The construction uses a random-sign hypercube of localized H\"older bumps, and the classical $\chi^2$ second-moment method controls the resulting mixture \citep{Ingster1987,IngsterSuslina2003,AriasCastroPelletierSaligrama2018}. Thus the mixture remains close to the uniform null, although each hypercube vertex is separated from the null in $L_1$.  This argument is nonprivate and therefore lower bounds arbitrary tests.

\paragraph{Within-cell privacy.}
The term $(N\sqrt\varepsilon)^{-2s/(2s+d)}$ is caused by alternatives that are invisible to any sufficiently coarse histogram.  The construction partitions $[0,1]^d$ into coarse cells and, inside each cell, places sign-symmetric positive and negative H\"older bumps whose integral over that cell is zero. Thus every alternative has exactly the same coarse cell probabilities as the uniform null, so the binned experiment carries no signal; all information is contained in the conditional locations of observations within occupied cells. We then use the DP-coupling method of \citet{AcharyaSunZhang2018}. It couples the null experiment to the random mixture over within-cell perturbations while modifying only a small expected number of sample points. Differential privacy prevents a reliable distinction unless this expected Hamming cost is of order $1/\varepsilon$.  Balancing this privacy constraint with the largest H\"older amplitude allowed inside cells yields the within-cell privacy scale $(N\sqrt\varepsilon)^{-2s/(2s+d)}$.

\paragraph{Private hypercube transport.}
The term $(N^{3/2}\varepsilon)^{-2s/(4s+d)}$ uses the same random-sign hypercube as the classical lower bound, but with a privacy-sensitive testing geometry.  Under central DP, the rejection probability of any test can change only in a controlled way under Hamming perturbations of the dataset.  This stability leads to the following transport-type consequence of \Cref{lem:transport}: if $\mathcal M$ is an $\varepsilon$-DP test with rejection probability $p_{\mathcal M}$, and if $Q$ is any mixture distribution over datasets that is absolutely continuous with respect to the null law $P_0^N$, then
\begin{align*}
\E_Q[p_{\mathcal M}]-\E_0[p_{\mathcal M}]
\;\lesssim\;
\varepsilon\sqrt{N\,\KL(Q\|P_0^N)} .
\end{align*}
Equivalently, an $\varepsilon$-DP test cannot substantially increase its rejection probability on a mixture $Q$ unless the mixture-to-null KL divergence is at least of order $1/(N\varepsilon^2)$.  Applying this inequality to the random-sign hypercube mixture, and then optimizing the bump resolution and amplitude, yields the privacy-transport barrier $(N^{3/2}\varepsilon)^{-2s/(4s+d)}$.

\paragraph{Linear privacy.}
The final term $(N\varepsilon)^{-1}$ is a direct privacy barrier. It already appears when the alternative is obtained by taking a fixed smooth perturbation of the uniform density and varying only its amplitude. A single draw from the null distribution and a single draw from such an alternative can be coupled to differ with probability of order $\|f-1\|_1$. Therefore, under the product coupling for $N$ observations, the expected Hamming distance is of order $N\|f-1\|_1$. The DP-coupling argument of \citet{AcharyaSunZhang2018} then prevents any $\varepsilon$-DP test from substantially changing its rejection probability between the null and the alternative unless $N\varepsilon\|f-1\|_1$ is bounded away from zero. This yields the linear barrier $\|f-1\|_1\gtrsim (N\varepsilon)^{-1}$.

\bigskip
\noindent \Cref{app:lower-bounds} makes these four mechanisms precise and proves the corresponding lower bounds separately.  Together they match the four terms in the upper bound.  We next return to the constructive side and remove prior knowledge of the smoothness index.

\section{Adaptivity to Unknown Smoothness}
\label{sec:ub-adaptive-smoothness}

The fixed-smoothness procedure chooses the resolution $\kappa$ using the smoothness index $s$, which is typically unknown in practice. We remove this dependence without estimating $s$. Instead, we evaluate the same binned Rao--Blackwellized statistic over a dyadic collection of resolutions and aggregate the resulting evidence across scales.

The adaptive test has two Monte Carlo stages. In the first, or centering, stage, each resolution $\kappa$ is assigned its own high conditional permutation quantile, estimated from random relabelings of the pooled sample. This quantile estimates the null scale of the statistic at resolution $\kappa$, and is subtracted from the corresponding score. In the second, or rank-calibration, stage, we maximize the centered scores over all resolutions and apply a single private permutation calibration to this multiscale maximum. The two sets of relabelings play different roles. The centering relabelings are used only to normalize the resolution-specific statistics and are never released. The rank relabelings are used only in the final private permutation comparison of the aggregated score. Thus the procedure releases only one binary decision, while adapting over the entire multiscale grid.

For simplicity, we present the adaptive construction for the balanced even case $N_X=N_Y=N=2n$. Let $\ell^{(0)}\in\mathfrak L_{2n,2n}$ denote the observed labeling of the $4n$ pooled records, with $\ell^{(0)}_i=X$ for $1\le i\le 2n$ and $\ell^{(0)}_i=Y$ for $2n<i\le 4n$. For each resolution $\kappa$, define the binned pooled vector
\begin{align*}
W_\kappa
:=
\bigl(
\widetilde X_1^{(\kappa)},\ldots,\widetilde X_{2n}^{(\kappa)},
\widetilde Y_1^{(\kappa)},\ldots,\widetilde Y_{2n}^{(\kappa)}
\bigr)
\in[\kappa^d]^{4n},
\end{align*}
where $\widetilde X_j^{(\kappa)}=\mathrm{cell}_\kappa(X_j)$ and $\widetilde Y_j^{(\kappa)}=\mathrm{cell}_\kappa(Y_j)$.  For any $\ell\in\mathfrak L_{2n,2n}$, the resolution-specific statistic is $\overline Z_n(W_\kappa,\ell)$, where $\overline Z_n$ is the Rao--Blackwellized statistic from \Cref{sec:rb-implementation}.

To adapt without knowing the oracle resolution, we search over the dyadic grid
\begin{align*}
\mathcal K_N=\{2^0,2^1,\ldots,2^{J_N}\},
\qquad
J_N=\lceil3\log_2N\rceil,
\end{align*}
and write $G_N=|\mathcal K_N|$. The upper endpoint $2^{J_N}\asymp N^3$ is a convenient polynomial envelope large enough to contain, up to constant factors, the optimizing resolutions over the compact smoothness range considered below. Fix an integer $B_{\mathrm{rank}}$ such that $(B_{\mathrm{rank}}+1)^{-1}\le\gamma$, and set
\begin{equation}\label{eq:adaptive-main-tuning}
\begin{aligned}
u_N=\frac{\beta}{256(B_{\mathrm{rank}}+1)G_N},
\qquad
a_N=\log(8/u_N),
\qquad
B_{\mathrm{cen},N}=\left\lceil 16\,\frac{a_N}{u_N}\right\rceil.
\end{aligned}
\end{equation}
The tuning parameters in \eqref{eq:adaptive-main-tuning} control the Monte Carlo error in the centering stage uniformly over the $G_N$ resolutions. The numerical constants are chosen only to simplify the proof. Since $B_{\mathrm{rank}}$ is fixed and $G_N\asymp\log N$, this uniform control requires $a_N\asymp \log\log N$ and $B_{\mathrm{cen},N}\asymp \log N\,\log\log N$ centering relabelings per resolution.

The resulting procedure is summarized in \Cref{alg:adaptive-rb-mc-main}. Operationally, the test first computes the binned pooled data at every resolution in $\mathcal K_N$.  A centering batch then produces offsets $\widehat q_\kappa$, one for each resolution, which are held fixed in the final calibration step.  The rank stage assigns each labeling a single multiscale score by maximizing the centered statistics over $\kappa$, and then applies the private permutation comparison only once to this aggregated score. This two-stage construction is related in spirit to two-batch exchangeability-based aggregation in \citet{SchrabShahGrettonKim2026}, although here the split is used inside a differentially private multiscale test rather than as a general nonprivate aggregation device.

\begin{algorithm}[!t]
\caption{Adaptive Rao--Blackwellized private permutation test}
\label{alg:adaptive-rb-mc-main}
\KwIn{Samples $X_1,\ldots,X_{2n}\stackrel{\iid}{\sim} f$,
$Y_1,\ldots,Y_{2n}\stackrel{\iid}{\sim} g$;
privacy budget $\varepsilon$; level $\gamma$; type~II tolerance $\beta$;
number of rank relabelings $B_{\mathrm{rank}}$.}

Let $N=2n$, and let
$\ell^{(0)}\in\mathfrak L_{2n,2n}$ be the observed assignment with
$\ell_i^{(0)}=X$ for $1\le i\le2n$ and
$\ell_i^{(0)}=Y$ for $2n<i\le4n$\;

Use the grid $\mathcal K_N$, the rank budget $B_{\mathrm{rank}}$, and the
tuning parameters $u_N,a_N,B_{\mathrm{cen},N}$ specified in \eqref{eq:adaptive-main-tuning}\;

\textbf{Centering stage.}\;
Let $\mathcal E$ denote all centering-stage relabelings drawn below; these
relabelings are independent of the rank-stage relabelings\;
\For{$\kappa\in\mathcal K_N$}{
  Form the binned pooled vector $W_\kappa$\;
  Draw independent assignments
  $\ell_{\kappa,1}^{\mathrm{cen}},\ldots,
  \ell_{\kappa,B_{\mathrm{cen},N}}^{\mathrm{cen}}
  \stackrel{\iid}{\sim}\Unif(\mathfrak L_{2n,2n})$\;
  Compute
  $T_{\kappa,r}:=\overline Z_n(W_\kappa,\ell_{\kappa,r}^{\mathrm{cen}})$,
  $r=1,\ldots,B_{\mathrm{cen},N}$\;
  Set $\widehat q_\kappa(X,Y;\mathcal E)$ equal to the empirical
  $(1-u_N)$-quantile of
  $T_{\kappa,1},\ldots,T_{\kappa,B_{\mathrm{cen},N}}$, i.e.,
  \begin{align*}
  \widehat q_\kappa(X,Y;\mathcal E)
  :=
  T_{\kappa,(\lceil(1-u_N)B_{\mathrm{cen},N}\rceil)},
  \end{align*}
  where $T_{\kappa,(1)}\le\cdots\le T_{\kappa,(B_{\mathrm{cen},N})}$ denote the order statistics.
}

\textbf{Private rank stage.}\;
For any $\ell\in\mathfrak L_{2n,2n}$, define
\begin{align*}
\widehat A_{\mathcal E}(X,Y;\ell)
:=
\max_{\kappa\in\mathcal K_N}
\left\{
\overline Z_n(W_\kappa,\ell)-\widehat q_\kappa(X,Y;\mathcal E)
\right\}.
\end{align*}
Set $\ell_0^{\mathrm{rank}}:=\ell^{(0)}$. Draw
$
\ell_1^{\mathrm{rank}},\ldots,\ell_{B_{\mathrm{rank}}}^{\mathrm{rank}}
\stackrel{\iid}{\sim}\Unif(\mathfrak L_{2n,2n})
$, $\zeta_0,\ldots,\zeta_{B_{\mathrm{rank}}}\stackrel{\iid}{\sim}\Lap(1)$
and compute
$\widehat M_b
:=\widehat A_{\mathcal E}(X,Y;\ell_b^{\mathrm{rank}})+(16/\varepsilon)\zeta_b$,
$b=0,\ldots,B_{\mathrm{rank}}$\;
Set
$\widehat p_{\mathrm{rank}}=
\{1+\sum_{b=1}^{B_{\mathrm{rank}}}\mathbf{1}\{\widehat M_b\ge \widehat M_0\}\}/
(B_{\mathrm{rank}}+1)$\;
Reject $H_0$ if  $\widehat p_{\mathrm{rank}}\le\gamma$\;
\end{algorithm}

The centering relabelings in \Cref{alg:adaptive-rb-mc-main} are auxiliary randomness used only to form the offsets $\widehat q_\kappa$.  Conditional on the unlabeled pooled sample and on this centering randomness, these offsets are fixed; under $H_0:f=g$, the observed labeling and the rank-stage relabelings are exchangeable.  Hence the private permutation theorem applies to the aggregated statistic $\widehat A_{\mathcal E}$, giving finite-sample type~I control.  For privacy, each $\overline Z_n(W_\kappa,\ell)$ has sensitivity at most four, and each empirical centering quantile has the same sensitivity.  Thus $\widehat A_{\mathcal E}$ has sensitivity at most eight, so the Laplace scale $16/\varepsilon$ yields $\varepsilon$-differential privacy by \Cref{lem:private-permutation-cited}.  The remaining cost is the uniform estimation of the centering levels over $G_N\asymp\log N$ resolutions, which produces the iterated-logarithmic factor in the separation rate.

Importantly, the procedure requires neither the unknown smoothness parameter $s$ nor the bounds $s_-$ and $s_+$ as input. The latter appear in the theorem below only to specify the range over which the guarantee holds uniformly.

\begin{theorem}[Adaptive upper bound]
\label{thm:adaptive-upper-loglog-rb-mc}
Fix $0<s_-<s_+<\infty$, $d\ge1$, $L>0$, $M\ge1$, and $\gamma,\beta\in(0,1)$ with $\gamma+\beta<1$.  Suppose $s\in[s_-,s_+]$ is unknown to the statistician.  There exists a constant $C=C(d,s_-,s_+,L,M,\gamma,\beta)>0$ such that, for every even $N=2n\ge2$ and $\varepsilon\in(0,1]$, \Cref{alg:adaptive-rb-mc-main}, with the tuning parameters above and with only the final rejection decision released, is $\varepsilon$-DP, has type~I error at most $\gamma$ over $H_0:f=g$, and has type~II error at most $\beta$ over $f,g\in\cF_s^d(L,M)$ whenever
\begin{equation}
\label{eq:adaptive-rb-mc-inf-bound}
\|f-g\|_1
\ge
C
\inf_{\kappa\in\mathcal K_N}
\left\{
\kappa^{-s}
+
\sqrt{\frac{a_N}{N}}
+
a_N^{1/4}\frac{\kappa^{d/4}}{\sqrt N}
+
\frac{\kappa^{d/2}}{N\sqrt{\varepsilon}}
+
\frac{\kappa^{d/4}}{N^{3/4}\sqrt{\varepsilon}}
+
\frac{1}{N\varepsilon}
\right\}.
\end{equation}
Consequently, since $G_N\lesssim\log N$ and $a_N\asymp\log\log N$, the same conclusion holds whenever
\begin{equation}
\label{eq:adaptive-rb-mc-explicit-loglog-rate}
\|f-g\|_1
\ge
C
\left[
a_N^{s/(4s+d)}N^{-2s/(4s+d)}
\vee
(N\sqrt{\varepsilon})^{-2s/(2s+d)}
\vee
(N^{3/2}\varepsilon)^{-2s/(4s+d)}
\vee
(N\varepsilon)^{-1}
\right].
\end{equation}
\end{theorem}

\begin{proof}
See \Cref{app:adaptive-upper}.
\end{proof}

The adaptive theorem shows that the fixed-smoothness minimax rate can still be attained when $s$ is unknown, with only an iterated-logarithmic loss in the classical term. The loss appears only in this term because it comes from searching over many resolutions. Controlling random fluctuations simultaneously at all resolutions requires a slightly higher threshold, which produces the $\log\log N$ factor. By contrast, the test does not release a separate private result for each resolution. It first takes the maximum of the resolution-specific centered scores and then releases only one private decision. Importantly, taking a maximum does not add up the sensitivities across resolutions: if every centered score changes by at most $\Delta$ when one observation is replaced, then their maximum also changes by at most $\Delta$. Thus, the privacy budget does not need to be divided across resolutions, and the three privacy terms remain unchanged. The iterated-logarithmic loss in the classical term is natural in light of the adaptive chi-square testing theory of \citet{Ingster2000AdaptiveChiSquare}, where adaptation is obtained by combining tests across resolutions.  The next theorem shows that this loss is unavoidable even for arbitrary nonprivate randomized tests and therefore for every centrally private test.

\begin{theorem}[Adaptive lower bound]
\label{thm:adaptive-lower-loglog}
Fix $0<s_-<s_+<\infty$, $d\ge1$, $L>1$, $M>1$, and $\gamma,\beta\in(0,1)$ with $\gamma+\beta<1$. There exist constants $c>0$ and $N_0$, depending only on $(d,s_-,s_+,L,M,\gamma,\beta)$, such that the following holds for all $N=N_X\wedge N_Y\ge N_0$. Define the class of adaptive level-$\gamma$ tests by
\begin{align*}
\Phi_{\gamma,N_X,N_Y}^{\mathrm{ad}}
:=
\left\{
\phi:
\sup_{s\in[s_-,s_+]}
\sup_{h\in\cF_s^d(L,M)}
\E_{h,h}[\phi]
\le \gamma
\right\}.
\end{align*}
For $s\in[s_-,s_+]$, set
\begin{align*}
\rho_{N,s}^{\mathrm{ad}}
:=
c\,(\log\log N)^{s/(4s+d)}N^{-2s/(4s+d)}.
\end{align*}
Then
\begin{align*}
\inf_{\phi\in\Phi_{\gamma,N_X,N_Y}^{\mathrm{ad}}}
\;
\sup_{s\in[s_-,s_+]}
\;
\sup_{\substack{
f,g\in\cF_s^d(L,M)\\
\|f-g\|_1\ge \rho_{N,s}^{\mathrm{ad}}
}}
\Bigl(1-\E_{f,g}[\phi]\Bigr)
>
\beta .
\end{align*}
Equivalently, no randomized two-sample test with adaptive type~I error at most $\gamma$ can have worst-case type~II error at most $\beta$, uniformly over $s\in[s_-,s_+]$, at separation level $\rho_{N,s}^{\mathrm{ad}}$.
\end{theorem}

\begin{proof}
The proof is given in \Cref{app:proof:thm:adaptive-lower-loglog}.
\end{proof}

Consequently, since the tuning parameter in \Cref{alg:adaptive-rb-mc-main} satisfies $a_N\asymp\log\log N$, the factor
\begin{align*}
a_N^{s/(4s+d)}
\asymp
(\log\log N)^{s/(4s+d)}
\end{align*}
in \eqref{eq:adaptive-rb-mc-explicit-loglog-rate} cannot be replaced, uniformly over $s\in[s_-,s_+]$, by any smaller-order factor. The remaining privacy-dependent obstructions in \eqref{eq:adaptive-rb-mc-explicit-loglog-rate} are the fixed-smoothness lower bounds already proved in \Cref{sec:lb}.  Since any adaptive procedure is, in particular, a valid procedure at each fixed smoothness level, those lower bounds continue to apply without change.  Thus \Cref{thm:adaptive-lower-loglog} supplies precisely the additional obstruction caused by not knowing $s$.

We close with a numerical illustration of the constructive tests.

\section{Numerical Illustration}
\label{sec:simulation}

This section illustrates the finite-sample behavior of the proposed procedures. The simulations are not intended to estimate exponents in the minimax rates, but rather to illustrate finite-sample calibration, privacy-power tradeoffs, and the effect of Rao--Blackwellization.

For clarity, all experiments are one-dimensional, with sample space $[0,1]$. Unless stated otherwise, we use equal sample sizes $N_X=N_Y=1000$, nominal level $\gamma=0.05$, histogram resolution $\kappa=16$, and $B_{\mathrm{perm}}=199$ Monte Carlo relabelings.  Each reported rejection probability is computed from 5000 independent repetitions.  The nonprivate benchmark uses the same permutation statistic and calibration, but omits the Laplace perturbations; in the plots, this benchmark is denoted by $\varepsilon=\infty$.

Under the null, both samples are drawn from the uniform density $f=g\equiv 1$ on $[0,1]$.  For the power experiments, we consider the smooth one-parameter alternatives
\begin{align*}
 f(x)=1,
 \qquad
 g_r(x)=1+\frac{\pi r}{2}\sin(4\pi x),
 \qquad x\in[0,1].
\end{align*}
Since $\int_0^1 |\sin(4\pi x)|\,dx=2/\pi$, this parametrization gives $\|f-g_r\|_1=r$.  The displayed values of $r$ lie in $[0,0.30]$, so the alternative densities remain uniformly bounded away from zero.

We compare three tests.  The first is the private split-statistic permutation test, based on the four-block balanced relabeling described in \Cref{sec:ub-discrete}.  The second is the private Rao--Blackwellized permutation test, which averages over the artificial split and evaluates the closed-form statistic $\overline Z$ from \Cref{sec:rb-implementation}.  The third is the corresponding nonprivate Rao--Blackwellized benchmark, obtained by removing the Laplace noise.

\begin{figure}[t]
\centering
\includegraphics[width=\textwidth]{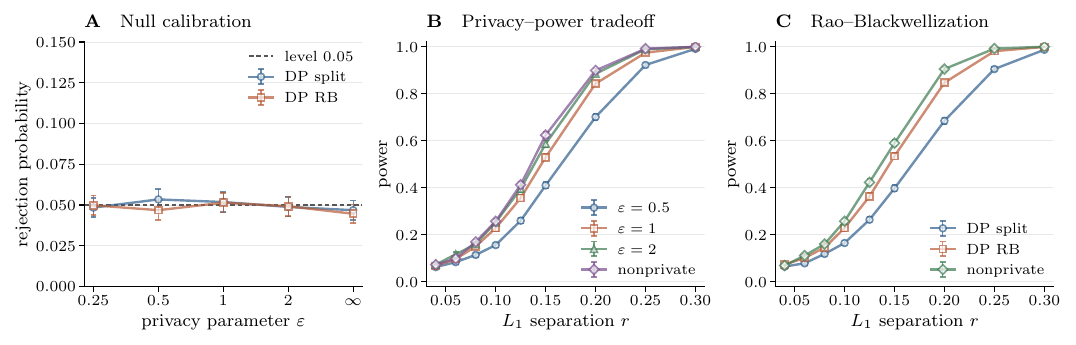}
\caption{Finite-sample behavior of the private permutation tests. Panel A reports empirical type~I error under the uniform null as a function of the privacy parameter; the dashed line marks the nominal level $\gamma=0.05$.  Panel B reports empirical power of the Rao--Blackwellized test against the alternatives $g_r(x)=1+(\pi r/2)\sin(4\pi x)$ for several privacy levels.  Panel C compares the private split-statistic test, the private Rao--Blackwellized test, and the nonprivate Rao--Blackwellized benchmark at $\varepsilon=1$.  Error bars show pointwise 95\% Monte Carlo intervals.}
\label{fig:simulation}
\end{figure}

The results in \Cref{fig:simulation} are consistent with the preceding theoretical guarantees and intuition. The permutation calibration keeps the empirical null rejection probability close to the nominal level, within Monte Carlo fluctuation. Along the one-parameter family $g_r$, power increases with the $L_1$ separation, while privacy has its largest visible effect in the transition region where rejection probabilities are bounded away from both zero and one. Finally, at the same privacy budget, the Rao--Blackwellized statistic has higher power than the raw split statistic, consistent with the variance-reduction property of Rao--Blackwellization discussed in \Cref{sec:rb-implementation}.

We next illustrate the adaptive multiscale procedure from \Cref{sec:ub-adaptive-smoothness}.  The purpose of this experiment is different from that of \Cref{fig:simulation}: rather than varying only the signal strength along a fixed scale, we construct alternatives for which the relevant resolution is unknown.  Let $I=[1/4,3/4]$.  For an even integer $M$, partition $I$ into $M$ equal subintervals and define $\phi_M$ to alternate between $+1$ and $-1$ on consecutive subintervals, with $\phi_M=0$ outside $I$. Since $\int \phi_M=0$ and $\int|\phi_M|=|I|=1/2$, the alternatives
\begin{align*}
 f(x)=1,\qquad
 g_{M,r}(x)=1+2r\,\phi_M(x)
\end{align*}
satisfy $\|f-g_{M,r}\|_1=r$.  Coarse histograms can average out the alternating signal when $M$ is large, while overly fine histograms pay additional variance and privacy-noise costs when the signal is coarse.  We compare fixed Rao--Blackwellized private tests at $\kappa\in\{4,32,128\}$ with the adaptive Rao--Blackwellized test over the simulation grid $\mathcal K_{\mathrm{sim}}=\{4,8,16,32,64,128\}$, with $G_{\mathrm{sim}}:=|\mathcal K_{\mathrm{sim}}|=6$.  Unless otherwise stated, we use $N_X=N_Y=1000$, $\varepsilon=1$, $B_{\mathrm{perm}}=B_{\mathrm{rank}}=B_{\mathrm{cen}}=199$, and 5000 repetitions.  In this finite-sample illustration, the adaptive centering quantile at each of the $G_{\mathrm{sim}}$ resolutions is taken at level $1-\gamma/G_{\mathrm{sim}}$, rather than using the proof-tuned $u_N$ from \Cref{alg:adaptive-rb-mc-main}.

\begin{figure}[t]
\centering
\includegraphics[width=\textwidth]{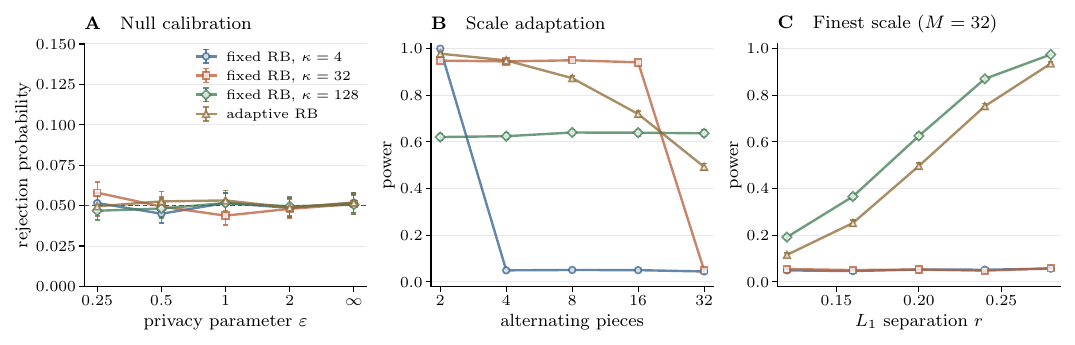}
\caption{Finite-sample behavior of the adaptive multiscale test. Panel A reports empirical type~I error under the uniform null for the fixed resolution tests and the adaptive test; the dashed line marks the nominal level $\gamma=0.05$.  Panel B reports power against the local alternating alternatives $g_{M,r}$ at $r=0.20$, as the number $M$ of alternating pieces varies.  Panel C fixes the finest displayed scale $M=32$ and varies the $L_1$ separation $r$.  Error bars show pointwise 95\% Monte Carlo intervals.}
\label{fig:adaptive-simulation}
\end{figure}

The results in \Cref{fig:adaptive-simulation} show the intended scale-adaptive behavior.  The null rejection probabilities remain close to the nominal level for both fixed and adaptive procedures.  In the power experiment, a fixed coarse resolution performs well for the coarsest alternating signal but loses power once the signs cancel within bins, while the intermediate resolution is strong only at the intermediate scale.  The finest fixed resolution detects the finest alternatives but pays a visible price at coarser scales.  The adaptive test avoids these failures: it tracks the favorable fixed resolution across scales, with the expected moderate loss from combining several resolutions and privately calibrating the resulting maximum.

\section{Discussion}
\label{sec:discussion}

We have characterized the minimax $L_1$ separation radius for two-sample testing over H\"older density classes under central differential privacy in the regime $N\varepsilon\ge1$.  The rate is given by the maximum of the classical nonprivate smooth-testing term and three privacy terms.  This yields a phase diagram in which the dominant term depends on the smoothness, dimension, and privacy level; in particular, one intermediate phase disappears when $s\ge d/4$.  The upper bound is achieved by binning the samples, applying a private closeness test to the resulting histograms, and optimizing the bin resolution. The boundedness of the H\"older class places the binned distributions in a bounded-histogram subclass, which removes the heavy-element term from the worst-case discrete private closeness rate.  The lower bounds match the four terms in the upper bound, showing that each term corresponds to a genuine barrier.  We also obtain an adaptive procedure over compact smoothness ranges and prove that its iterated-logarithmic loss in the nonprivate part of the rate is unavoidable.

Several questions remain open.  A natural first direction is to study the same problem under shuffle privacy~\citep{CheuSmithUllmanZeber2019,ErlingssonFeldmanMironovRaghunathanTalwarThakurta2019}. Since shuffle privacy lies between the local and central models, it is not clear whether the central-DP phase diagram persists, collapses to fewer regimes, or acquires additional terms from the anonymization step.  The histogram-based construction gives a possible starting point, but matching lower bounds would likely require tools that capture both communication through shuffled messages and the smoothness structure of the alternatives. It would also be useful to extend the adaptive theory beyond isotropic H\"older balls to other smoothness scales, such as Sobolev or Besov classes. For anisotropic or spatially inhomogeneous models, this may require choosing different resolutions across directions or regions, with privacy interacting nonuniformly with that choice. More broadly, the techniques developed here may apply to other nonparametric testing problems in which a smooth model is reduced to a structured discrete testing problem. Finally, it would be interesting to extend the theory to unbounded or heavy-tailed classes, where both the tests and the lower-bound arguments would need to account explicitly for tail behavior.

\section*{Acknowledgements}

During this work, the author used ChatGPT 5.5 Pro for language editing, proof brainstorming, and mathematical checks. The author verified all AI-assisted material and assumes full responsibility for the correctness and content of the work.

\bibliographystyle{plainnat}
\bibliography{reference}

\appendix
\setcounter{theorem}{0}
\renewcommand{\thetheorem}{S.\arabic{theorem}}
\renewcommand{\theHtheorem}{supp.\arabic{theorem}}

\section*{Appendix Roadmap}

The appendix is organized into two parts. \Cref{app:auxiliary-lemmas} collects the general tools used throughout the paper, while \Cref{app:proofs} contains the proofs of the main results.

\section{Auxiliary Lemmas}
\label{app:auxiliary-lemmas}

This section collects auxiliary lemmas used repeatedly in the proofs.

\subsection{Probabilistic Tools}
\label{app:sec:probabilistic-tools}

\begin{lemma}[Efron--Stein inequality
{\citealt[Theorem~3.1]{BoucheronLugosiMassart2013}}]
\label{lem:efron-stein}
Let $X_1,\ldots,X_m$ be independent random variables, and let $H=H(X_1,\ldots,X_m)$ be square integrable.  For each $i$, let $X_i'$ be an independent copy of $X_i$, independent of all other variables, and let
\begin{align*}
H^{(i)}
:=
H(X_1,\ldots,X_{i-1},X_i',X_{i+1},\ldots,X_m).
\end{align*}
Then
\begin{align*}
\Var(H)
\le
\frac12\sum_{i=1}^m \E\bigl[(H-H^{(i)})^2\bigr].
\end{align*}
In particular, if replacing the $i$-th coordinate can change $H$ by at most $c_i$, then
\begin{align*}
\Var(H)\le \frac12\sum_{i=1}^m c_i^2.
\end{align*}
\end{lemma}

The next two lemmas are sampling-without-replacement and multislice analogues of bounded differences.

\begin{lemma}[Bounded differences for sampling without replacement]
\label{lem:sampling-without-replacement-bdd-diff}
Let $V=(V_1,\ldots,V_m)$ be the sequence of marks obtained by drawing $m$ tokens without replacement from a finite population of marked tokens, and let $F$ be a real-valued function of $V$. Suppose that $F$ is $c$-Lipschitz in Hamming distance, namely, for any two feasible ordered samples $v,v'$,
\begin{align*}
|F(v)-F(v')|
\le
c\sum_{r=1}^m \mathbf 1\{v_r\neq v'_r\}.
\end{align*}
Then, for every $\lambda\in\mathbb R$,
\begin{align*}
\E\exp\{\lambda(F(V)-\E F(V))\}
\le
\exp\{2c^2m\lambda^2\}.
\end{align*}
Consequently, for every $x>0$,
\begin{align*}
\Pbb\{F(V)-\E F(V)>x\}
\le
\exp\left(
-\frac{x^2}{8c^2m}
\right).
\end{align*}
\end{lemma}

\begin{proof}
Let $\mathcal F_r:=\sigma(V_1,\ldots,V_r)$ and define the Doob martingale
\begin{align*}
M_r:=\E[F(V)\mid \mathcal F_r],
\qquad r=0,1,\ldots,m.
\end{align*}
We claim that $|M_r-M_{r-1}|\le2c$ almost surely. Fix a prefix $v_{1:r-1}$ with positive probability, and consider two possible values $a$ and $b$ for $V_r$. Let $\mathcal U_a$ and $\mathcal U_b$ be the remaining token populations after drawing $a$ or $b$, respectively. These populations differ only by exchanging the token $a$ with the token $b$. Couple the remaining draws by drawing $R$ uniformly without replacement from $\mathcal U_a$, and obtaining $R'$ by replacing the token $b$ by the token $a$ if $b$ appears in $R$. This is a bijective coupling with the conditional law of the remaining draws from $\mathcal U_b$. The resulting full sequences differ at coordinate $r$ and possibly at one future coordinate, so their Hamming distance is at most $2$. Hence the corresponding conditional expectations of $F$ differ by at most $2c$, and the possible values of $M_r$ given $\mathcal F_{r-1}$ lie in an interval of length at most $2c$. This proves $|M_r-M_{r-1}|\le2c$.

Hoeffding's lemma applied to the martingale differences gives
\begin{align*}
\E\exp\{\lambda(F(V)-\E F(V))\}
\le
\exp\left\{
\frac{\lambda^2}{2}\sum_{r=1}^m(2c)^2
\right\}
=
\exp\{2c^2m\lambda^2\}.
\end{align*}
The stated tail bound follows by Chernoff's inequality.
\end{proof}

\begin{lemma}[Concentration on a balanced multislice]
\label{lem:balanced-multislice-conc}
Let $L$ be uniformly distributed on
\begin{align*}
\mathfrak G(n_1,\ldots,n_q)
:=
\left\{
\ell\in[q]^m:
\#\{i:\ell_i=a\}=n_a,\ a=1,\ldots,q
\right\},
\end{align*}
where $\sum_{a=1}^q n_a=m$. Let $I\subseteq[m]$ with $|I|=m_I$, and let $F:\mathfrak G(n_1,\ldots,n_q)\to\mathbb R$ depend only on $(L_i)_{i\in I}$. Suppose that $F$ is $c$-Lipschitz in Hamming distance on the active coordinates:
\begin{align*}
|F(\ell)-F(\ell')|
\le
c\sum_{i\in I}\mathbf 1\{\ell_i\ne \ell_i'\}.
\end{align*}
Then, for every $\lambda\in\mathbb R$,
\begin{align*}
\E\exp\{\lambda(F(L)-\E F(L))\}
\le
\exp\{C c^2m_I\lambda^2\}.
\end{align*}
Consequently, for every $t\ge0$,
\begin{align*}
\Pbb\left\{
F(L)-\E F(L)>C c\sqrt{m_I t}
\right\}
\le e^{-t}.
\end{align*}
\end{lemma}

\begin{proof}
If $m_I=0$, then $F(L)$ is constant and the claim is immediate. Otherwise, order $I=\{i_1,\ldots,i_{m_I}\}$. Under the uniform law on the balanced multislice, $(L_{i_1},\ldots,L_{i_{m_I}})$ has the same distribution as an ordered sample drawn without replacement from a finite population containing $n_a$ tokens marked $a$, $a=1,\ldots,q$. Applying \Cref{lem:sampling-without-replacement-bdd-diff} to the induced function of these active labels gives the exponential-moment bound. The tail bound follows by Chernoff's inequality after increasing the universal constant $C$.
\end{proof}

\begin{lemma}[Rao--Blackwellization preserves conditional sub-Gaussianity]
\label{lem:rb-preserves-subgaussian}
Let $\mathcal G\subseteq\mathcal H$ be sigma-fields, and let $X$ be a real-valued random variable such that
\begin{align*}
\E[X\mid\mathcal G]=0,
\qquad
\E[\exp(\lambda X)\mid\mathcal G]\le \exp(\sigma^2\lambda^2)
\end{align*}
for every $\lambda\in\mathbb R$. Let $Y:=\E[X\mid\mathcal H]$. Then
\begin{align*}
\E[Y\mid\mathcal G]=0,
\qquad
\E[\exp(\lambda Y)\mid\mathcal G]\le \exp(\sigma^2\lambda^2)
\end{align*}
for every $\lambda\in\mathbb R$. Consequently, $Y$ satisfies the same one-sided sub-Gaussian tail bound as $X$, up to universal constants.
\end{lemma}

\begin{proof}
The mean identity follows from the tower property. For the moment bound, Jensen's inequality conditional on $\mathcal H$ gives
\begin{align*}
\exp(\lambda Y)
=
\exp\{\lambda\E[X\mid\mathcal H]\}
\le
\E[\exp(\lambda X)\mid\mathcal H].
\end{align*}
Taking conditional expectation given $\mathcal G$ and using the tower property yields the claim. The tail statement follows by Chernoff's inequality.
\end{proof}

\begin{lemma}[Monte Carlo quantile calibration]
\label{lem:mc-empirical-quantile}
Let $V_0,V_1,\ldots,V_R$ be conditionally i.i.d. real-valued random variables given a sigma-field $\mathcal G$. For $u\in(0,1)$, set $m_u:=\lceil(1-u)R\rceil$ and $\widehat q_u:=V_{(m_u)}$, where $V_{(1)}\le\cdots\le V_{(R)}$ are the order statistics of $V_1,\ldots,V_R$. Then, almost surely, the following assertions hold.
\begin{enumerate}[(i)]
\item If $x$ is $\mathcal G$-measurable and $\Pbb\{V_0>x\mid\mathcal G\}\le u/8$, then
\begin{align*}
\Pbb\{\widehat q_u>x\mid\mathcal G\}\le \exp(-uR).
\end{align*}

\item
\begin{align*}
\Pbb\{V_0>\widehat q_u\mid\mathcal G\}
\le
\frac{R+1-m_u}{R+1}
\le
u+\frac{1}{R+1}.
\end{align*}
\end{enumerate}
\end{lemma}

\begin{proof}
For (i), let $Y:=\sum_{r=1}^R \mathbf 1\{V_r>x\}$. Conditional on $\mathcal G$, $Y$ has a binomial distribution with parameters $(R,p)$, where $p=\Pbb\{V_0>x\mid\mathcal G\}\le u/8$. Since $\widehat q_u>x$ implies $Y\ge R-m_u+1\ge uR$, the binomial Chernoff bound therefore gives
\begin{align*}
\Pbb\{\widehat q_u>x\mid\mathcal G\}
\le
\Pbb\{Y\ge uR\mid\mathcal G\}
\le
\left(\frac e8\right)^{uR}
\le
\exp(-uR).
\end{align*}

For (ii), condition on $\mathcal G$. The variables $V_0,V_1,\ldots,V_R$ are exchangeable, and the event $\{V_0>\widehat q_u\}$ requires at least $m_u$ of the other values to lie strictly below $V_0$. In any deterministic realization of $V_0,\ldots,V_R$, at most $R+1-m_u$ indices can have this property. Thus
\begin{align*}
\Pbb\{V_0>\widehat q_u\mid\mathcal G\}
\le
\frac{R+1-m_u}{R+1}
\le
u+\frac{1}{R+1}.
\end{align*}
\end{proof}

\subsection{Deterministic Approximation}
\label{app:sec:binning-technical}

\begin{lemma}[Uniform binning approximation for H\"{o}lder functions]
\label{lem:binning-approx}
Fix $d\in\mathbb N$, $L>0$, and $0<s_-\le s_+<\infty$. There exists a constant
\begin{align*}
C_{\mathrm{bin}}=C_{\mathrm{bin}}(d,s_-,s_+,L)>0
\end{align*}
such that, for every $s\in[s_-,s_+]$, every $h\in\cH_s^d(2L)$, and every integer $\kappa\ge1$,
\begin{align*}
\|W_{\kappa,d}(h)\|_1
\ge \frac12\|h\|_1-C_{\mathrm{bin}}\kappa^{-s}.
\end{align*}
Here $W_{\kappa,d}$ is the cell-average operator associated with the regular $\kappa^d$-cell partition of $[0,1]^d$:
\begin{align*}
W_{\kappa,d}(h)
:=
\sum_{j\in[\kappa]^d}
\kappa^d
\left(\int_{C_j}h\right)\mathbf 1_{C_j}.
\end{align*}
\end{lemma}

\begin{proof}
We adapt the blockwise-polynomial argument of \citet[Lemma~3]{AriasCastroPelletierSaligrama2018}, which proves the arbitrary-dimensional $L_2$ analogue, but give all details in $L_1$. The proof has three steps.

\paragraph{Step 1: uniform stability of cell averaging on polynomials.}
For an integer $q\ge0$, let $\mathcal P_q^d$ be the finite-dimensional space of real polynomials on $[0,1]^d$ of total degree at most $q$. The set $\{P\in\mathcal P_q^d:\|P\|_1=1\}$ is compact in coefficient space, and the map $P\mapsto\|\nabla P\|_\infty$ is continuous. Therefore
\begin{align*}
A_{d,q}
:=
\sup\bigl\{\|\nabla P\|_\infty:
P\in\mathcal P_q^d,\ \|P\|_1=1\bigr\}
<\infty,
\end{align*}
with $A_{d,0}=0$. Let
\begin{align*}
Q:=\lceil s_+\rceil-1,
\qquad
A_*:=\max_{0\le q\le Q}A_{d,q},
\end{align*}
and choose an integer $r\ge1$ such that
\begin{align*}
\frac{\sqrt d}{r}A_*\le\frac12.
\end{align*}

Consider any rectangular grid partition $\mathcal Q$ of $[0,1]^d$ into $n_j\ge r$ equal subintervals in coordinate $j$, and let $\Pi_{\mathcal Q}$ denote cell averaging over this grid. Every cell has Euclidean diameter at most $\sqrt d/r$. Hence, for $P\in\mathcal P_q^d$ with $q\le Q$,
\begin{align*}
\begin{aligned}
\|P-\Pi_{\mathcal Q}P\|_1
&\le \frac{\sqrt d}{r}\|\nabla P\|_\infty \\
&\le \frac{\sqrt d}{r}A_{d,q}\|P\|_1
\le \frac12\|P\|_1.
\end{aligned}
\end{align*}
The triangle inequality therefore gives the uniform polynomial stability bound
\begin{align*}
\|\Pi_{\mathcal Q}P\|_1\ge\frac12\|P\|_1,
\qquad
P\in\mathcal P_q^d,\quad 0\le q\le Q.
\end{align*}

\paragraph{Step 2: piecewise Taylor approximation on coarse blocks.}
Fix $s\in[s_-,s_+]$ and put
\begin{align*}
q:=\lceil s\rceil-1,
\qquad
\eta:=s-q\in(0,1].
\end{align*}
We first record a Taylor remainder bound with a constant that is uniform in $s\in[s_-,s_+]$. For $x_0,x\in[0,1]^d$, let
\begin{align*}
T_{x_0,q}h(x)
:=
\sum_{|\alpha|\le q}
\frac{D^\alpha h(x_0)}{\alpha!}(x-x_0)^\alpha.
\end{align*}
When $q=0$, the H\"older condition directly gives
\begin{align*}
|h(x)-T_{x_0,0}h(x)|\le 2L\|x-x_0\|_2^s.
\end{align*}
When $q\ge1$, set $v=x-x_0$ and $\varphi(t)=h(x_0+tv)$. Taylor's formula along the line segment from $x_0$ to $x$ gives
\begin{align*}
h(x)-T_{x_0,q}h(x)
=
\frac1{(q-1)!}\int_0^1(1-t)^{q-1}
\{\varphi^{(q)}(t)-\varphi^{(q)}(0)\}\,dt.
\end{align*}
Moreover,
\begin{align*}
\varphi^{(q)}(t)
=
\sum_{|\alpha|=q}\frac{q!}{\alpha!}
v^\alpha D^\alpha h(x_0+tv).
\end{align*}
Since $h\in\cH_s^d(2L)$,
\begin{align*}
\begin{aligned}
|\varphi^{(q)}(t)-\varphi^{(q)}(0)|
&\le 2L\,t^\eta\|v\|_2^\eta
   \sum_{|\alpha|=q}\frac{q!}{\alpha!}|v^\alpha| \\
&=2L\,t^\eta\|v\|_2^\eta\|v\|_1^q \\
&\le 2L\,d^{q/2}t^\eta\|v\|_2^s.
\end{aligned}
\end{align*}
Consequently,
\begin{align*}
|h(x)-T_{x_0,q}h(x)|
\le 2L\frac{d^{q/2}}{q!}\|x-x_0\|_2^s.
\end{align*}
Thus, with
\begin{align*}
c_*:=\max_{0\le j\le Q}\frac{d^{j/2}}{j!},
\end{align*}
where the $j=0$ term equals one, the preceding bound and the $q=0$ case imply
\begin{align*}
|h(x)-T_{x_0,q}h(x)|
\le 2Lc_*\|x-x_0\|_2^s
\end{align*}
uniformly over $s\in[s_-,s_+]$.

Assume first that $\kappa\ge r$. In each coordinate, partition the $\kappa$ fine-grid intervals into consecutive groups, each containing between $r$ and $2r-1$ fine intervals. Such a grouping exists: write $\kappa=mr+t$ with $m=\lfloor\kappa/r\rfloor\ge1$ and $0\le t<r$, and take one group of length $r+t$ and the remaining $m-1$ groups of length $r$. Taking Cartesian products of these one-dimensional groups yields a partition $\mathcal B$ of $[0,1]^d$ into rectangular coarse blocks, each of which is a union of fine cells.

For each $B\in\mathcal B$, let $x_B$ be its center and let $P_B:=T_{x_B,q}h$. Define the piecewise polynomial
\begin{align*}
u(x):=\sum_{B\in\mathcal B}P_B(x)\mathbf 1_B(x).
\end{align*}
Every side length of $B$ is less than $2r/\kappa$, and hence
\begin{align*}
\sup_{x\in B}\|x-x_B\|_2\le \frac{r\sqrt d}{\kappa}.
\end{align*}
By the uniform Taylor bound above, and since $r\sqrt d\ge1$ and $s\le s_+$,
\begin{align*}
\|h-u\|_\infty
\le 2Lc_*(r\sqrt d)^{s_+}\kappa^{-s}
=:E_*\kappa^{-s}.
\end{align*}
The domain has unit volume, so the same bound holds in $L_1$.

Because every coarse block is a union of fine cells, the fine-grid averaging operator acts separately on the blocks. Fix $B\in\mathcal B$ and rescale $B$ affinely onto $[0,1]^d$. If the block contains $n_j(B)$ fine intervals in coordinate $j$, then $n_j(B)\ge r$. Under this rescaling, $P_B$ remains a polynomial of total degree at most $q$, and the fine cells become a rectangular grid with at least $r$ equal cells in every coordinate. Applying the polynomial stability bound above, and then rescaling back, gives
\begin{align*}
\|W_{\kappa,d}(P_B\mathbf 1_B)\|_{L_1(B)}
\ge \frac12\|P_B\|_{L_1(B)}.
\end{align*}
Summing over the disjoint coarse blocks,
\begin{align*}
\|W_{\kappa,d}u\|_1\ge\frac12\|u\|_1.
\end{align*}

\paragraph{Step 3: conclusion.}
The operator $W_{\kappa,d}$ is an $L_1$ contraction, since
\begin{align*}
\|W_{\kappa,d}v\|_1
=
\sum_C\left|\int_Cv\right|
\le
\sum_C\int_C|v|
=
\|v\|_1.
\end{align*}
Combining the approximation, polynomial-stability, and contraction bounds, for $\kappa\ge r$ we obtain
\begin{align*}
\begin{aligned}
\|W_{\kappa,d}h\|_1
&\ge \|W_{\kappa,d}u\|_1
     -\|W_{\kappa,d}(h-u)\|_1 \\
&\ge \frac12\|u\|_1-\|h-u\|_1 \\
&\ge \frac12\|h\|_1-\frac32\|h-u\|_1 \\
&\ge \frac12\|h\|_1-\frac32E_*\kappa^{-s}.
\end{aligned}
\end{align*}

It remains to cover $1\le\kappa<r$. Since $h\in\cH_s^d(2L)$, $\|h\|_1\le2L$. Also, for $\kappa<r$ and $s\le s_+$, $\kappa^{-s}\ge r^{-s_+}$. Therefore
\begin{align*}
0\ge\frac12\|h\|_1-Lr^{s_+}\kappa^{-s}.
\end{align*}
Since $\|W_{\kappa,d}h\|_1\ge0$, the desired result holds for all $\kappa\ge1$ with
\begin{align*}
C_{\mathrm{bin}}
:=
\max\left\{\frac32E_*,\ Lr^{s_+}\right\}.
\end{align*}
All quantities used to define $r$, $E_*$, and $C_{\mathrm{bin}}$ depend only on $(d,s_-,s_+,L)$, proving both the stated inequality and its uniformity over the smoothness interval.
\end{proof}

\subsection{Differential Privacy Tools}
For a statistic $T$, write the global sensitivity of $T$ as
\begin{align*}
\GS(T):=\sup_{D\sim D'}\sup_{\ell}
|T(D,\ell)-T(D',\ell)|,
\end{align*}
where $D\sim D'$ denotes neighboring datasets that differ in one record. Let $T(D,\ell)$ be a generic statistic with global sensitivity $\GS(T)$, uniformly over all admissible relabelings $\ell$. Let $\ell^{(0)}$ denote the observed labeling and let $\ell^{(1)},\ldots,\ell^{(B)}$ denote permutation labelings. For $b=0,1,\ldots,B$, compute $ T_b:=T(D,\ell^{(b)}), $ and draw independent noises $ \eta_0,\ldots,\eta_B \stackrel{\iid}{\sim} \Lap(2\GS(T)/\varepsilon). $ Define the privatized statistics $\widetilde T_b:=T_b+\eta_b$ for $b=0,1,\ldots,B$, and the noisy permutation $p$-value
\begin{align*}
\widetilde p
:=
\frac{
1+\sum_{b=1}^B \mathbf 1\{\widetilde T_b\ge \widetilde T_0\}
}{B+1}.
\end{align*}
The test rejects whenever $\widetilde p\le\gamma$. The following lemma, corresponding to \citet[Theorems~1--2]{KimSchrab2026}, records the privacy and validity guarantees of this procedure.

\begin{lemma}[Private permutation theorem {\citealp[Theorems~1--2]{KimSchrab2026}}]
\label{lem:private-permutation-cited}
Consider the setting above. Then the released decision $\mathbf 1 \{\widetilde p\le\gamma\}$ is $\varepsilon$-differentially private. Moreover, if, under $H_0$, the observed labeling and the permutation labelings are exchangeable conditional on the unlabeled pooled data, then the test has finite-sample type~I error at most $\gamma$.
\end{lemma}

\subsubsection{DP Coupling Lemma}
\label{app:proof:lem:dp-coupling}

The following standard DP-coupling lower bound is used, for example, in \citet{AcharyaSunZhang2018}. We include the short proof to keep constants and notation self-contained. For $x,y\in\mathcal X^n$, write
\begin{align*}
d_H(x,y):=\sum_{i=1}^n \mathbf 1\{x_i\ne y_i\}
\end{align*}
for the Hamming distance between datasets.

\begin{lemma}[DP lower bound via couplings]\label{lem:dp-coupling}
Let $P$ and $Q$ be probability measures on $\mathcal X^n$. If there exists a coupling $(X_1^n,Y_1^n)$ of $(P,Q)$ with $\E[d_H(X_1^n,Y_1^n)]\le D$, then any $\varepsilon$-DP test $\Psi:\mathcal X^n\to\{0,1\}$ with $P(\Psi=1)\le\gamma$ and $Q(\Psi=0)\le\beta$ must satisfy
\begin{align*}
D\ge \frac{1}{\varepsilon}\log\frac{1}{\gamma+\beta}
\ge \frac{1-\gamma-\beta}{\varepsilon}.
\end{align*}
\end{lemma}

\begin{proof}
Let
\begin{align*}
p_\Psi(x):=\Pbb(\Psi(x)=1),\qquad x\in\mathcal X^n.
\end{align*}
Group privacy applied to the non-rejection event gives, for all $x,y\in\mathcal X^n$,
\begin{align*}
1-p_\Psi(y)
\ge
e^{-\varepsilon d_H(x,y)}\{1-p_\Psi(x)\}.
\end{align*}
Evaluating this inequality under the coupling $(X_1^n,Y_1^n)$ and using $Q(\Psi=0)\le\beta$, we obtain
\begin{align*}
\beta
\ge
\E\!\left[
e^{-\varepsilon d_H(X_1^n,Y_1^n)}
\{1-p_\Psi(X_1^n)\}
\right].
\end{align*}
Since $a(1-b)\ge a-b$ for $a,b\in[0,1]$,
\begin{align*}
\beta
\ge
\E\!\left[e^{-\varepsilon d_H(X_1^n,Y_1^n)}\right]
-
\E[p_\Psi(X_1^n)].
\end{align*}
By Jensen's inequality and the coupling assumption,
\begin{align*}
\E\!\left[e^{-\varepsilon d_H(X_1^n,Y_1^n)}\right]
\ge
e^{-\varepsilon \E[d_H(X_1^n,Y_1^n)]}
\ge
e^{-\varepsilon D}.
\end{align*}
Together with $\E[p_\Psi(X_1^n)]=P(\Psi=1)\le\gamma$, this gives
\begin{align*}
\beta\ge e^{-\varepsilon D}-\gamma.
\end{align*}
Thus $e^{-\varepsilon D}\le\gamma+\beta$, and hence
\begin{align*}
D\ge \frac{1}{\varepsilon}\log\frac{1}{\gamma+\beta}.
\end{align*}
If $\gamma+\beta\ge1$, the second displayed lower bound is trivial. If $\gamma+\beta<1$, it follows from $-\log t\ge1-t$ with $t=\gamma+\beta$. This completes the proof.
\end{proof}

\subsubsection{Transport Lemma}
\label{app:proof:lem:transport}

The next lemma provides transport-type control on rejection probabilities for central-DP tests in terms of the mixture-to-null KL divergence and the sample size. To the best of our knowledge, this inequality is new.

\begin{lemma}[Transport inequality for central-DP tests]
\label{lem:transport}
Let $\mathcal M$ be an $\varepsilon$-DP binary mechanism, $p_{\mathcal M}$ its rejection probability function, and let $Q$ be a probability measure absolutely continuous with respect to $P_0^N$, written $Q\ll P_0^N$. Then
\begin{align*}
\E_Q[p_{\mathcal M}]-\E_0[p_{\mathcal M}]
\le
(1-e^{-\varepsilon})\sqrt{\frac{N}{2}\KL(Q\|P_0^N)}
\le
\varepsilon\sqrt{\frac{N}{2}\KL(Q\|P_0^N)}.
\end{align*}
\end{lemma}

\begin{proof}
We first record how differential privacy controls changes in the rejection probability. For a dataset $D$, write $p(D):=p_{\mathcal M}(D)=\Pbb(\mathcal M(D)=1)$.  If $D\sim D'$, then applying differential privacy to the rejection event and to its complement gives
\begin{align*}
p(D)\le e^\varepsilon p(D'),
\qquad
1-p(D)\le e^\varepsilon\{1-p(D')\},
\end{align*}
and the same inequalities also hold with $D$ and $D'$ interchanged.  Suppose first that $p(D)\ge p(D')$.  Using the complement inequality in the direction $D'\to D$,
\begin{align*}
1-p(D')\le e^\varepsilon\{1-p(D)\},
\end{align*}
we get
\begin{align*}
1-p(D)
\ge e^{-\varepsilon}\{1-p(D')\},
\qquad\text{hence}\qquad
p(D)\le 1-e^{-\varepsilon}\{1-p(D')\}.
\end{align*}
Therefore
\begin{align*}
\begin{aligned}
p(D)-p(D')
&\le 1-e^{-\varepsilon}\{1-p(D')\}-p(D') \\
&=(1-e^{-\varepsilon})\{1-p(D')\}
\le 1-e^{-\varepsilon}.
\end{aligned}
\end{align*}
If $p(D')>p(D)$, the same argument with the two datasets interchanged yields $p(D')-p(D)\le 1-e^{-\varepsilon}$.  Thus
\begin{align*}
|p(D)-p(D')|
\le 1-e^{-\varepsilon}
=:c_\varepsilon
\qquad\text{whenever }D\sim D'.
\end{align*}
In particular, under $P_0^N$, the map $(X_1,\ldots,X_N)\mapsto p(X_1, \ldots,X_N)$ satisfies the bounded-differences condition with constants $c_\varepsilon,\ldots,c_\varepsilon$.  For brevity, write $p=p(X_1,\ldots,X_N)$.  McDiarmid's exponential inequality \citep[Theorem~6.2]{BoucheronLugosiMassart2013} therefore implies that, for all $\lambda>0$,
\begin{align*}
\log \E_0
\exp\!\left\{\lambda\bigl(p-\E_0[p]\bigr)\right\}
\le
\frac{\lambda^2}{8}\sum_{i=1}^N c_\varepsilon^2
=
\frac{\lambda^2 N c_\varepsilon^2}{8}.
\end{align*}

We next convert this null concentration estimate into a bound on the change of expectation under $Q$.  The entropy variational formula states that, for probability measures $Q\ll P$ and any measurable $F$ with $\E_P e^F<\infty$, \citep[Corollary~4.15]{BoucheronLugosiMassart2013}
\begin{align*}
\E_Q[F]
\le
\KL(Q\|P)+\log \E_P e^F .
\end{align*}
Apply this with $P=P_0^N$ and
\begin{align*}
F=\lambda\bigl(p-\E_0[p]\bigr),
\qquad \lambda>0.
\end{align*}
Since $0\le p\le1$, the exponential moment is finite.  Moreover,
\begin{align*}
\E_Q[F]
=
\lambda\bigl\{\E_Q[p]-\E_0[p]\bigr\}.
\end{align*}
Consequently,
\begin{align*}
\E_Q[p]-\E_0[p]
\le
\frac{\KL(Q\|P_0^N)}{\lambda}
+ 
\frac{1}{\lambda}
\log \E_0
\exp\!\left\{\lambda\bigl(p-\E_0[p]\bigr)\right\}.
\end{align*}
Combining this display with the McDiarmid bound gives, for every $\lambda>0$,
\begin{align*}
\E_Q[p]-\E_0[p]
\le
\frac{\KL(Q\|P_0^N)}{\lambda}
+ 
\frac{\lambda N c_\varepsilon^2}{8}.
\end{align*}
If $\KL(Q\|P_0^N)=0$, letting $\lambda\downarrow0$ gives $\E_Q[p]\le\E_0[p]$, and the desired bound is immediate.  Otherwise, optimizing the right-hand side yields
\begin{align*}
\E_Q[p]-\E_0[p]
\le
c_\varepsilon\sqrt{\frac{N}{2}\KL(Q\|P_0^N)}.
\end{align*}
Finally, $c_\varepsilon=1-e^{-\varepsilon}\le\varepsilon$, which gives the second inequality.
\end{proof}

\subsection{Discrete Moment and Comparison Tools}
\label{sec:expectation_gap}
This subsection proves the binomial comparison inequality used in the bounded-histogram analysis. Compared with the comparison lemma of \citet{CanonneSun2022}, the formulation below does not require the restriction $p,q\le 1/4$ and applies for all $n\ge2$.

\begin{lemma}[Binomial comparison without a small-mass restriction]
\label{lem:binomial-comparison-no-heavy}
Let $n\ge 2$ and let $p,q\in[0,1]$. Suppose
\begin{align*}
X,X'\stackrel{\iid}{\sim} \mathrm{Bin}(n,p),
\qquad
Y,Y'\stackrel{\iid}{\sim} \mathrm{Bin}(n,q)
\end{align*}
are mutually independent, and define
\begin{align*}
\mathfrak D_n(p,q):=\E\bigl[|X-Y|+|X'-Y'|-|X-X'|-|Y-Y'|\bigr].
\end{align*}
Write $\delta:=|p-q|$, $\mu:=np$, and $\lambda:=nq$. Then
\begin{align*}
\mathfrak D_n(p,q)\ge \frac{1}{16\pi}
\min\left\{n^2\delta^2,\; n\delta,\; \frac{n^2\delta^2}{\sqrt{n(p+q)}}\right\}
= \frac{1}{16\pi}
\min\left\{(\mu-\lambda)^2,\; |\mu-\lambda|,\; \frac{(\mu-\lambda)^2}{\sqrt{\mu+\lambda}}\right\}.
\end{align*}
\end{lemma}

\begin{proof}
If $p=q$, then $\mathfrak D_n(p,q)=0$ and the claim is trivial. Hence assume throughout that $\delta:=|p-q|>0$.

For $s\in[0,1]$ and $t\ge 0$, define
\begin{align*}
u_s(t):=\E[e^{it\mathrm{Bin}(n,s)}]=(1-s+se^{it})^n.
\end{align*}
From \citet[][Equation (8)]{CanonneSun2022}, we obtain
\begin{align*}
\mathfrak D_n(p,q)
&=\frac{2}{\pi}\int_0^{\infty}
\frac{\bigl(|u_p(t)|^2+|u_q(t)|^2-2\Re(u_p(t)\overline{u_q(t)})\bigr)}{t^2}\,dt \\
&=\frac{2}{\pi}\int_0^{\infty}\frac{|u_p(t)-u_q(t)|^2}{t^2}\,dt.
\end{align*}
In particular, the integrand is nonnegative. Therefore, for every $T>0$,
\begin{align*}
\mathfrak D_n(p,q)\ge \frac{2}{\pi}\int_0^T\frac{|u_p(t)-u_q(t)|^2}{t^2}\,dt.
\end{align*}
We work on the interval $0\le t\le 1$. For such $t$, set $z_s(t):=1-s+se^{it}$. Since $z_s(t)$ is a convex combination of $1$ and $e^{it}$, it lies on the line segment joining these two points. For $t\in[0,1]$, we have $\cos t\ge \cos 1>0$, and hence
\begin{align*}
\Re z_s(t)=1-s+s\cos t = 1-s(1-\cos t)\ge \cos t\ge \cos 1>0
\end{align*}
for every $s\in[0,1]$. Thus $z_s(t)$ never vanishes on $[0,1]\times[0,1]$. Consequently the principal argument $\Arg(z_s(t))$ is well defined there, and we may write
\begin{align*}
u_s(t)=\rho_s(t)e^{i\varphi_s(t)},
\end{align*}
where
\begin{align*}
\rho_s(t):=|z_s(t)|^n,\qquad \varphi_s(t):=n\Arg(z_s(t)),
\end{align*}
and $\Arg$ is the principal argument. Then
\begin{equation}\label{eq:uv-lower}
\begin{aligned}
|u_p(t)-u_q(t)|^2
&=|\rho_p(t)e^{i\varphi_p(t)}-\rho_q(t)e^{i\varphi_q(t)}|^2\\
&=\rho_p(t)^2+\rho_q(t)^2-2\rho_p(t)\rho_q(t)\cos(\varphi_p(t)-\varphi_q(t))\\
&=(\rho_p(t)-\rho_q(t))^2+2\rho_p(t)\rho_q(t)\bigl(1-\cos(\varphi_p(t)-\varphi_q(t))\bigr)\\
&\ge 2\rho_p(t)\rho_q(t)\bigl(1-\cos(\varphi_p(t)-\varphi_q(t))\bigr).
\end{aligned}
\end{equation}

We need two estimates.

\medskip\noindent\textbf{Step 1: phase estimate.}
For $0\le t\le 1$ and $s\in[0,1]$,
\begin{align*}
\partial_s\varphi_s(t)
= n\,\Im\!\left(\frac{e^{it}-1}{1-s+se^{it}}\right)
= n\,\frac{\sin t}{|1-s+se^{it}|^2}.
\end{align*}
Moreover,
\begin{align*}
|1-s+se^{it}|^2
=1-2s(1-s)(1-\cos t).
\end{align*}
Since $2s(1-s)\le \frac12$, we obtain
\begin{align*}
|1-s+se^{it}|^2\ge 1-\frac12(1-\cos t)=\frac{1+\cos t}{2}=\cos^2(t/2).
\end{align*}
Hence, for $0\le t\le 1$,
\begin{align*}
n\sin t\le \partial_s\varphi_s(t)
\le \frac{n\sin t}{\cos^2(t/2)}
\le 2nt,
\end{align*}
because $\cos^2(t/2)\ge \cos^2(1/2)>3/4$ and $\sin t\le t$. Also, $\sin t\ge t/2$ on $[0,1]$. Therefore,
\begin{align*}
\frac12 nt\le \partial_s\varphi_s(t)\le 2nt
\qquad (0\le t\le 1,\; 0\le s\le 1).
\end{align*}
Integrating from $s=\min\{p,q\}$ to $s=\max\{p,q\}$ yields
\begin{equation}\label{eq:phase-bound}
\frac12 n\delta t\le |\varphi_p(t)-\varphi_q(t)|\le 2n\delta t
\qquad (0\le t\le 1).
\end{equation}

\medskip\noindent\textbf{Step 2: modulus estimate.}
For $0\le s,t\le 1$,
\begin{align*}
|z_s(t)|^2
=1-2s(1-s)(1-\cos t)
\ge 1-s t^2,
\end{align*}
since $1-\cos t\le t^2/2$ and $1-s\le 1$. As $st^2\in[0,1]$ and $n/2\ge 1$, Bernoulli's inequality gives
\begin{align*}
\rho_s(t)=|z_s(t)|^n\ge (1-st^2)^{n/2}\ge 1-\frac n2 st^2.
\end{align*}
Choose
\begin{align*}
T:=\min\left\{1,\; \frac{1}{2n\delta},\; \frac{1}{\sqrt{n(p+q)}}\right\}.
\end{align*}
Since $\delta>0$, this is well-defined. For every $t\in[0,T]$, we have
\begin{align*}
\frac n2 pt^2 \le \frac{p}{2(p+q)} \le \frac12,
\qquad
\frac n2 qt^2 \le \frac{q}{2(p+q)} \le \frac12,
\end{align*}
which makes the two factors below nonnegative. Therefore,
\begin{equation}\label{eq:modulus-bound}
\rho_p(t)\rho_q(t)
\ge \left(1-\frac n2 pt^2\right)\left(1-\frac n2 qt^2\right)
\ge 1-\frac n2 (p+q)t^2
\ge \frac12
\qquad (0\le t\le T).
\end{equation}
Also, by \eqref{eq:phase-bound},
\begin{align*}
|\varphi_p(t)-\varphi_q(t)|\le 2n\delta t\le 1
\qquad (0\le t\le T).
\end{align*}
For $0\le x\le 1$, we have
\begin{align*}
1-\cos x = \int_0^x \sin y\,dy \ge \int_0^x \frac{y}{2}\,dy = \frac{x^2}{4},
\end{align*}
and therefore, using the lower bound in \eqref{eq:phase-bound},
\begin{align*}
1-\cos(\varphi_p(t)-\varphi_q(t))
\ge \frac14 |\varphi_p(t)-\varphi_q(t)|^2
\ge \frac{1}{16}n^2\delta^2 t^2.
\end{align*}
Combining this with \eqref{eq:uv-lower} and \eqref{eq:modulus-bound}, for every $t\in[0,T]$ we obtain
\begin{align*}
|u_p(t)-u_q(t)|^2
\ge 2\cdot \frac12\cdot \frac{1}{16}n^2\delta^2 t^2
=\frac{1}{16}n^2\delta^2 t^2.
\end{align*}
Substituting into the integral representation of $\mathfrak D_n(p,q)$ yields
\begin{align*}
\mathfrak D_n(p,q)
\ge \frac{2}{\pi}\int_0^T \frac{(1/16)n^2\delta^2 t^2}{t^2}\,dt
=\frac{1}{8\pi}n^2\delta^2 T.
\end{align*}
Hence
\begin{align*}
\mathfrak D_n(p,q)
\ge \frac{1}{8\pi}
\min\left\{n^2\delta^2,\; \frac12 n\delta,\; \frac{n^2\delta^2}{\sqrt{n(p+q)}}\right\}
\ge \frac{1}{16\pi}
\min\left\{n^2\delta^2,\; n\delta,\; \frac{n^2\delta^2}{\sqrt{n(p+q)}}\right\},
\end{align*}
as claimed.
\end{proof}

\section{Proofs of Main Results}
\label{app:proofs}

This section proves the main results stated in the main text.

\subsection{Proof of \CrefInTitle{Theorem}{thm:main-rate}}
\label{app:proof:thm:main-rate}

\begin{proof}
The proof combines the continuous upper bound with the goodness-of-fit lower bound through the two-sample reduction. Let $N:=N_X\wedge N_Y$, and set
\begin{align*}
R_N(\varepsilon)
:=
N^{-2s/(4s+d)}
\vee
(N\sqrt\varepsilon)^{-2s/(2s+d)}
\vee
(N^{3/2}\varepsilon)^{-2s/(4s+d)}
\vee
(N\varepsilon)^{-1}.
\end{align*}

\medskip\noindent\textbf{Upper bound.}
Since $\cD_s^d(L)\subset\cF_s^d(L,L)$ with envelope constant $M=L$, \Cref{cor:holder-l1-private-explicit-rate} implies that, for every $\gamma,\beta\in(0,1)$ with $\gamma+\beta<1$, there exists an $\varepsilon$-DP test $\phi\in\Phi_{\gamma,\varepsilon}^{\mathrm{cDP}}$ whose type~II error is at most $\beta$ whenever
\begin{align*}
\|f-g\|_1\ge C R_N(\varepsilon),
\end{align*}
for a constant $C=C(d,s,L,\gamma,\beta)>0$. Hence
\begin{align*}
r^*_{\mathrm{2samp}}(N_X,N_Y,\varepsilon)\le C R_N(\varepsilon).
\end{align*}

\medskip\noindent\textbf{Lower bound.}
By \Cref{lem:gof-to-twosamp-cdp},
\begin{align*}
r^*_{\mathrm{2samp}}(N_X,N_Y,\varepsilon)
\ge
r^*_{\mathrm{gof}}(N,\varepsilon).
\end{align*}
By \Cref{thm:gof-combined-lower}, whenever $N\varepsilon\ge1$, there exists $c=c(d,s,L,\gamma,\beta)>0$ such that
\begin{align*}
r^*_{\mathrm{gof}}(N,\varepsilon)\ge c R_N(\varepsilon).
\end{align*}
Hence
\begin{align*}
r^*_{\mathrm{2samp}}(N_X,N_Y,\varepsilon)\ge c R_N(\varepsilon).
\end{align*}
Putting everything together, when $N\varepsilon\ge1$,
\begin{align*}
r^*_{\mathrm{2samp}}(N_X,N_Y,\varepsilon)\asymp R_N(\varepsilon).
\end{align*}
\end{proof}

\subsection{Proof of \CrefInTitle{Corollary}{cor:sample-complexity}}
\label{app:proof:cor:sample-complexity}

\begin{proof}
By \Cref{thm:main-rate}, the sample complexity is obtained by solving the four inequalities in the maximum rate term by term. This gives
\begin{align*}
N
\gtrsim
r^{-(4s+d)/(2s)},\qquad
N
\gtrsim
\varepsilon^{-1/2}r^{-(2s+d)/(2s)},
\end{align*}
\begin{align*}
N
\gtrsim
\varepsilon^{-2/3}r^{-(4s+d)/(3s)},\qquad
N
\gtrsim
(\varepsilon r)^{-1}.
\end{align*}
Taking the maximum of these four requirements is therefore sufficient by the upper bound and necessary by the lower bound. When $N\varepsilon<1$, the DP-coupling argument underlying the linear privacy barrier gives $r\gtrsim (N\varepsilon)^{-1}$. Since $r\le 2$ for densities, this regime already corresponds to a constant-order separation requirement. Hence restricting \Cref{thm:main-rate} to $N\varepsilon\ge1$ causes no loss for the converse.
\end{proof}

\subsection{Proof of \CrefInTitle{Corollary}{cor:phase-transition}}
\label{app:proof:cor:phase-transition}

\begin{proof}

Substituting $\varepsilon_N=N^{-\alpha}$ into \eqref{eq:rate-combined} gives $r^*_{\mathrm{2samp}}(N,N^{-\alpha})\asymp N^{-\rho(\alpha)}$, where
\begin{align*}
\rho(\alpha)
=
\min\{\rho_1,\rho_2(\alpha),\rho_3(\alpha),\rho_4(\alpha)\}.
\end{align*}
The constants in $\asymp$ do not depend on $N$ or $\varepsilon$, so this substitution is uniform over $\alpha\in[0,1]$.

The adjacent crossings occur at
\begin{align*}
\begin{array}{c|c}
\text{crossing} & \text{location}\\ \hline
\rho_1=\rho_2 & \alpha_{12}=4s/(4s+d)\\
\rho_2=\rho_3 & \alpha_{23}=(d-2s)/d\\
\rho_3=\rho_4 & \alpha_{34}=(s+d)/(2s+d).
\end{array}
\end{align*}

If $s<d/4$, then $\alpha_{12}<\alpha_{23}<\alpha_{34}$. A direct comparison shows that, as $\alpha$ increases, the minimum is attained successively by $\rho_1,\rho_2,\rho_3,\rho_4$. This yields the four regimes in part~(a).

If $s\ge d/4$, then $\alpha_{23}\le \tfrac12\le \alpha_{12}$. Consequently $\rho_2$ never becomes the minimum: for $\alpha\le\alpha_{12}$ one has $\rho_2\ge\rho_1$, while for $\alpha\ge\alpha_{23}$ one has $\rho_2\ge\rho_3$. Hence the lower envelope is formed by $\rho_1,\rho_3,\rho_4$, with transitions at $\alpha=\tfrac12$ and $\alpha_{34}$, giving part~(b).
\end{proof}

\subsection{Proof of \CrefInTitle{Proposition}{prop:dpperm-validity-privacy}}
\label{app:proof:prop:dpperm-validity-privacy}
We first record the only statistic-specific ingredient needed for the privacy proof, namely the global sensitivity of the split statistic $Z$. For $a,b,c,d\in\mathbb Z_{\ge0}$, write
\begin{align*}
g(a,b,c,d):=|a-c|+|b-d|-|a-b|-|c-d|.
\end{align*}

\begin{lemma}[Sensitivity of $Z$]
\label{lem:mult-bdhist-sensitivity}
Changing any one argument by $\pm1$, whenever the resulting argument remains nonnegative, changes $g$ by at most $2$. Moreover,
\begin{align*}
|g(a\pm1,b,c,d)-g(a,b,c,d)|
\le
2\mathbf 1\{b\neq c\},
\end{align*}
and analogous bounds hold for changes in the other coordinates. Consequently, if
\begin{align*}
Z=\sum_{i=1}^k g(X_i,\widetilde X_i,Y_i,\widetilde Y_i),
\end{align*}
then changing one raw sample in any of the four groups changes $Z$ by at most $4$. Hence $\GS(Z)\le4$.
\end{lemma}

\begin{proof}

Since only two of the four absolute-value terms in $g$ involve any fixed coordinate, and each such term changes by at most $1$ when that coordinate is changed by $\pm 1$, the first bound follows immediately. If $b=c$, then
\begin{align*}
g(a,b,b,d)=|a-b|+|b-d|-|a-b|-|b-d|=0
\end{align*}
for every $a,d$.  Thus changing $a$ or $d$ can matter only when $b\neq c$. Similarly, if $a=d$, then
\begin{align*}
g(a,b,c,a)=|a-c|+|b-a|-|a-b|-|c-a|=0
\end{align*}
for every $b,c$.  Thus changing $b$ or $c$ can matter only when $a\neq d$. Combining these identities with the first bound yields the refined inequalities. For the final claim, changing one raw sample moves one unit mass from one bin to another in exactly one of the count vectors $X,\widetilde X,Y,\widetilde Y$. Thus only two coordinates of the sum defining $Z$ can change, and each corresponding summand changes by at most $2$. Therefore $\GS(Z)\le 4$.
\end{proof}

\begin{proof}[Proof of \Cref{prop:dpperm-validity-privacy}]
By \Cref{lem:mult-bdhist-sensitivity}, the split statistic $Z$ has global sensitivity $\GS(Z)\le4$ under the fixed four-block neighboring relation. Thus the Laplace scale $8/\varepsilon=2\GS(Z)/\varepsilon$ used in \Cref{alg:mcperm} is the scale required by the private permutation theorem \Cref{lem:private-permutation-cited}.

Under $H_0:p=q$, introduce an auxiliary uniform permutation of the pooled record indices and apply it simultaneously to the pooled discrete observations, the observed balanced four-block assignment, and all Monte Carlo balanced reassignments. The joint law is unchanged. Conditional on the resulting unlabeled pooled sample, the symmetrized observed assignment and the Monte Carlo assignments are exchangeable. Hence \Cref{lem:private-permutation-cited} applies, yielding an $\varepsilon$-differentially private test with finite-sample type~I error at most $\gamma$. When $\varepsilon=\infty$, the Laplace noise vanishes and the same exchangeability argument gives the level bound.
\end{proof}

\subsection{Proof of \CrefInTitle{Proposition}{prop:private-bounded-hist-perm}}
\label{app:proofs-private-bounded-hist-perm}
The proof combines moment bounds for the observed split statistic with corresponding bounds for the statistics generated by balanced permutation relabeling. We first collect the split-statistic bounds, then prove the permutation moment bounds, and finally verify the power condition for the private Monte Carlo permutation test.

\subsubsection{Split-statistic moment bounds}
\label{app:aux:split-statistic-moments}
The next lemmas collect the variance, deviation, and mean inputs for the split histogram statistic.

\begin{lemma}[Variance bound in the multinomial/binomial sampling model]
\label{lem:bounded-hist-var-binomial}
Fix $M\ge 1$, $p,q\in\mathcal P_{k,M}:=\{v\in\Delta_k:\|v\|_\infty\le M/k\}$. Let $X,\widetilde X\stackrel{\iid}{\sim} \mathrm{Mult}(n,p)$ and $Y,\widetilde Y\stackrel{\iid}{\sim} \mathrm{Mult}(n,q)$ be mutually independent, with coordinates denoting bin counts. Define
\begin{align*}
Z = \sum_{i=1}^k\bigl\{
|X_i-Y_i|+|\widetilde X_i-\widetilde Y_i|
-|X_i-\widetilde X_i|-|Y_i-\widetilde Y_i|
\bigr\}.
\end{align*}
Then, for all $n\ge 1$, $\operatorname{Var}(Z)\le 64M\,n^2/k$.
\end{lemma}

\begin{proof}

We represent the multinomial counts by underlying samples. Let
\begin{align*}
A_1,\dots,A_n,\ \widetilde A_1,\dots,\widetilde A_n \stackrel{\iid}{\sim} p,
\qquad
B_1,\dots,B_n,\ \widetilde B_1,\dots,\widetilde B_n \stackrel{\iid}{\sim} q,
\end{align*}
all mutually independent, with values in $\{1,\ldots,k\}$, and define
\begin{align*}
X_i=\sum_{m=1}^n \mathbf{1}\{A_m=i\},
\quad
\widetilde X_i=\sum_{m=1}^n \mathbf{1}\{\widetilde A_m=i\},
\quad
Y_i=\sum_{m=1}^n \mathbf{1}\{B_m=i\},
\quad
\widetilde Y_i=\sum_{m=1}^n \mathbf{1}\{\widetilde B_m=i\}.
\end{align*}
For the kernel $g$ defined above,
\begin{align*}
Z=\sum_{i=1}^k g(X_i,\widetilde X_i,Y_i,\widetilde Y_i).
\end{align*}
By \Cref{lem:mult-bdhist-sensitivity}, changing one argument of $g$ by $\pm1$ changes its value by at most $2$, and the refined indicator bounds stated there also hold.

We apply \Cref{lem:efron-stein} to the $4n$ independent raw samples.

\paragraph{Step 1: Replace one \texorpdfstring{$A_m$}{A_m}.}
Fix $m\in[n]$, let $A_m'\sim p$ be an independent copy of $A_m$, and let $Z^{(A,m)}$ be the statistic obtained after replacing $A_m$ by $A_m'$. Write
\begin{align*}
I:=A_m,\qquad J:=A_m'.
\end{align*}
Only the bins $I$ and $J$ can change, and only through the $X$-coordinates. Therefore, conditional on everything else,
\begin{align*}
|Z-Z^{(A,m)}|
\le
2\,\mathbf{1}\{\widetilde X_I\neq Y_I\}
+
2\,\mathbf{1}\{\widetilde X_J\neq Y_J\}.
\end{align*}
Hence
\begin{align*}
\bigl(Z-Z^{(A,m)}\bigr)^2
\le
8\Bigl(
\mathbf{1}\{\widetilde X_I\neq Y_I\}
+
\mathbf{1}\{\widetilde X_J\neq Y_J\}
\Bigr),
\end{align*}
because $(2u+2v)^2\le 8(u+v)$ for $u,v\in\{0,1\}$.

Taking expectations and using that $I,J$ are i.i.d.\ with law $p$, independent of $(\widetilde X,Y)$, we obtain
\begin{align*}
\mathbb{E}\bigl[(Z-Z^{(A,m)})^2\bigr]
\le
16\sum_{i=1}^k p_i\,\mathbb{P}(\widetilde X_i\neq Y_i).
\end{align*}
Moreover, the inclusion
\begin{align*}
\{\widetilde X_i\neq Y_i\}\subseteq \{\widetilde X_i\ge 1\}\cup \{Y_i\ge 1\}
\end{align*}
implies
\begin{align*}
\mathbb{P}(\widetilde X_i\neq Y_i)
\le
\mathbb{P}(\widetilde X_i\ge 1)+\mathbb{P}(Y_i\ge 1)
\le
np_i+nq_i.
\end{align*}
It follows that
\begin{align*}
\mathbb{E}\bigl[(Z-Z^{(A,m)})^2\bigr]
&\le
16n\sum_{i=1}^k p_i(p_i+q_i) \\
&\le
16n\left(\sum_{i=1}^k p_i^2+\sum_{i=1}^k p_iq_i\right) \\
&\le
16n\left(\|p\|_\infty\|p\|_1+\|q\|_\infty\|p\|_1\right) \\
&\le
16n\left(\frac{M}{k}+\frac{M}{k}\right) \\
&=
32M\,\frac{n}{k}.
\end{align*}
\paragraph{Step 2: The other three sample groups.}
Exactly the same argument applies when replacing one of the samples $\widetilde A_m$, $B_m$, or $\widetilde B_m$. Thus every one of the $4n$ independent raw samples satisfies
\begin{align*}
\mathbb{E}\bigl[(Z-Z^{(\ell)})^2\bigr]
\le
32M\,\frac{n}{k},
\end{align*}
where $Z^{(\ell)}$ denotes the statistic after replacing the $\ell$-th raw sample by an independent copy.

\paragraph{Step 3: Efron--Stein.}
\Cref{lem:efron-stein} yields
\begin{align*}
\operatorname{Var}(Z)
\le
\frac12\sum_{\ell=1}^{4n}\mathbb{E}\bigl[(Z-Z^{(\ell)})^2\bigr]
\le
\frac12\cdot 4n \cdot 32M\,\frac{n}{k}
=
64M\,\frac{n^2}{k}.
\end{align*}
This proves the claim.
\end{proof}

\begin{lemma}[Mean lower bound for the split statistic]
\label{lem:split-statistic-mean}
Fix $M\ge1$ and $n\ge2$, let $p,q\in\mathcal P_{k,M}$, and let $X,\widetilde X\stackrel{\iid}{\sim} \mathrm{Mult}(n,p)$ and $Y,\widetilde Y\stackrel{\iid}{\sim} \mathrm{Mult}(n,q)$ be mutually independent. Define
\begin{align*}
Z=
\sum_{i=1}^k
\bigl\{
|X_i-Y_i|+|\widetilde X_i-\widetilde Y_i|
-|X_i-\widetilde X_i|-|Y_i-\widetilde Y_i|
\bigr\}.
\end{align*}
Then there exists a constant $b_M>0$, depending only on $M$, such that
\begin{align}
\E[Z] &= 0,
\qquad \text{if } p=q,
\label{eq:mult-bdhist-null}\\
\E[Z]
&\ge
b_M
\min\left\{
n\tau,\,
\frac{n^2\tau^2}{k},\,
\frac{n^{3/2}\tau^2}{\sqrt{k}}
\right\},
\qquad
\text{if } \|p-q\|_1\ge \tau . \notag
\end{align}
\end{lemma}

\begin{proof}
When $p=q$, the four multinomial count vectors are exchangeable, so the two cross terms and the two within-sample terms have the same expectations. This proves \eqref{eq:mult-bdhist-null}.

For the lower bound, apply \Cref{lem:binomial-comparison-no-heavy} coordinatewise to $(X_i,\widetilde X_i,Y_i,\widetilde Y_i)$. Writing $\delta_i:=|p_i-q_i|$, and ignoring coordinates with $\delta_i=0$, gives
\begin{align*}
\E[Z]\ge
c
\sum_{i=1}^k
\min\left\{
n\delta_i,\,
n^2\delta_i^2,\,
\frac{n^2\delta_i^2}{\sqrt{n(p_i+q_i)}}
\right\}.
\end{align*}
Since $p,q\in\mathcal P_{k,M}$, we have $p_i+q_i\le 2M/k$. Absorbing constants depending only on $M$,
\begin{align*}
\E[Z]\ge
c_M
\sum_{i=1}^k
\min\left\{
n\delta_i,\,
n^2\delta_i^2,\,
n^{3/2}\sqrt{k}\,\delta_i^2
\right\}.
\end{align*}
Let $B:=n^2\wedge n^{3/2}\sqrt{k}$. Then the last display is bounded below by
\begin{align*}
c_M
\sum_{i=1}^k
\min\{n\delta_i,B\delta_i^2\}.
\end{align*}
If $\sum_i\delta_i=\|p-q\|_1\ge\tau$, the set $I:=\{i:\delta_i\ge\tau/(2k)\}$ has $\sum_{i\in I}\delta_i\ge\tau/2$. Therefore
\begin{align*}
\sum_{i=1}^k
\min\{n\delta_i,B\delta_i^2\}
\ge
\sum_{i\in I}
\delta_i\min\left\{n,\frac{B\tau}{2k}\right\}
\ge
c\min\left\{n\tau,\frac{B\tau^2}{k}\right\}.
\end{align*}
Substituting the definition of $B$ gives the displayed lower bound.
\end{proof}

\subsubsection{Permutation moment bounds}
\label{app:sec:perm-technical}

The next lemmas provide the permutation moment bounds used in \Cref{prop:private-bounded-hist-perm}.

\begin{lemma}[Second moment under balanced relabeling]
\label{lem:perm-second-moment}
Let $w=(w_1,\ldots,w_{4n})\in[k]^{4n}$ be a fixed pooled label sequence. Let $\mathfrak B_n$ be the set of assignments $a:[4n]\to\{1,2,3,4\}$ with exactly $n$ indices assigned to each value. For $a\in\mathfrak B_n$, define
\begin{align*}
N_{ij}(a):=\sum_{\ell=1}^{4n}\mathbf 1\{w_\ell=i,\ a_\ell=j\},
\qquad
T_i:=\sum_{j=1}^4 N_{ij}(a)=\sum_{\ell=1}^{4n}\mathbf 1\{w_\ell=i\}.
\end{align*}
Let
\begin{align*}
Z_w(a)
=
\sum_{i=1}^k
\Bigl\{
|N_{i1}(a)-N_{i3}(a)|
+
|N_{i2}(a)-N_{i4}(a)|
-
|N_{i1}(a)-N_{i2}(a)|
-
|N_{i3}(a)-N_{i4}(a)|
\Bigr\}.
\end{align*}
If $A\sim\Unif(\mathfrak B_n)$, then $\E_A[Z_w(A)]=0$ and
\begin{align*}
\E_A[Z_w(A)^2]
\le
C
\left(
n
\wedge
\sum_{i=1}^k T_i(T_i-1)
\right)
\end{align*}
for a universal constant $C>0$.
\end{lemma}

\begin{proof}

The mean is zero by exchangeability of the four labels under the uniform balanced assignment. For each fixed cell $i$, the joint distribution of $(N_{i1}(A),N_{i2}(A),N_{i3}(A),N_{i4}(A))$ is invariant under permutations of the four coordinates. Hence the expectations of the four absolute differences appearing in the $i$-th summand are all equal, and the expectation of that summand is zero. Summing over $i$ gives $\E_A[Z_w(A)]=0$.

We now prove the second-moment bound using only the balanced multislice concentration lemma. Recall the four-count kernel
\begin{align*}
g(x_1,x_2,x_3,x_4)
:=
|x_1-x_3|+|x_2-x_4|-|x_1-x_2|-|x_3-x_4|.
\end{align*}
For any $x,y\in\mathbb Z_{\ge0}^4$,
\begin{align*}
|g(x)-g(y)|
\le
2\|x-y\|_1,
\end{align*}
since each coordinate appears in exactly two of the four absolute-value terms defining $g$. If two assignments $a,a'\in\mathfrak B_n$ differ on $H_i$ observations whose pooled label is $i$, then the corresponding count vectors differ in $\ell_1$-distance at most $2H_i$. Therefore
\begin{align*}
\bigl|
g(N_{i1}(a),\ldots,N_{i4}(a))
-
g(N_{i1}(a'),\ldots,N_{i4}(a'))
\bigr|
\le
4H_i.
\end{align*}
Summing over $i$ gives
\begin{align*}
|Z_w(a)-Z_w(a')|
\le
4\sum_{\ell=1}^{4n}\mathbf 1\{a_\ell\ne a'_\ell\}.
\end{align*}
Thus $Z_w$ is $4$-Lipschitz in Hamming distance on the balanced multislice. Applying \Cref{lem:balanced-multislice-conc} with $I=[4n]$, $m_I=4n$, and $c=4$, we obtain, for all $\lambda\in\mathbb R$,
\begin{align*}
\E_A\exp\{\lambda(Z_w(A)-\E_AZ_w(A))\}
\le
\exp\{C n\lambda^2\}.
\end{align*}
This sub-Gaussian moment bound implies the desired variance bound. Indeed, if $\E e^{\lambda X}\le e^{K\lambda^2}$ for all $\lambda\in\mathbb R$, then $\Var(X)\le 2K$. Therefore
\begin{align*}
\Var_A(Z_w(A))\le Cn.
\end{align*}

It remains to obtain the collision-sensitive bound. If $T_i=1$, then the $i$-th summand of $Z_w$ is identically zero, regardless of the label assigned to the unique observation in that cell. Indeed, a count vector with total mass one has one coordinate equal to $1$ and the other three equal to $0$, and substituting any of the four possibilities into $g$ gives zero. Therefore $Z_w(A)$ depends only on the labels attached to observations in
\begin{align*}
G:=\{\ell\in[4n]:T_{w_\ell}\ge2\}.
\end{align*}
Moreover, the same Lipschitz argument gives, for all feasible assignments $a,a'$,
\begin{align*}
|Z_w(a)-Z_w(a')|
\le
4\sum_{\ell\in G}\mathbf 1\{a_\ell\ne a'_\ell\}.
\end{align*}
Applying \Cref{lem:balanced-multislice-conc} again, now with $I=G$, yields
\begin{align*}
\E_A\exp\{\lambda(Z_w(A)-\E_AZ_w(A))\}
\le
\exp\{C |G|\lambda^2\},
\qquad \lambda\in\mathbb R.
\end{align*}
The same standard consequence of the moment-generating-function bound gives
\begin{align*}
\Var_A(Z_w(A))\le C|G|.
\end{align*}
Finally,
\begin{align*}
|G|
=
\sum_{i:T_i\ge2}T_i
\le
\sum_{i=1}^k T_i(T_i-1).
\end{align*}
Combining the two bounds gives
\begin{align*}
\Var_A(Z_w(A))
\le
C\left(
n\wedge \sum_{i=1}^k T_i(T_i-1)
\right).
\end{align*}
Since $\E_A[Z_w(A)]=0$, this proves
\begin{align*}
\E_A[Z_w(A)^2]
\le
C\left(
n\wedge \sum_{i=1}^k T_i(T_i-1)
\right).
\end{align*}
\end{proof}

\begin{corollary}[Permutation second moment over bounded histograms]
\label{cor:perm-second-moment-bounded-hist}
Suppose the pooled data consist of $2n$ samples from $p$ and $2n$ samples from $q$, where $p,q\in\mathcal P_{k,M}$. Let $\pi$ be an independent uniformly random balanced reassignment of the pooled labels into four blocks of size $n$. Writing $W$ for the pooled discrete labels and $\lambda^\pi$ for the induced four-block assignment, set $Z^\pi:=Z_n(W,\lambda^\pi)$. Then
\begin{align*}
\E[Z^\pi]=0,
\qquad
\E[(Z^\pi)^2]
\le
C_M
\left(
n\wedge\frac{n^2}{k}
\right).
\end{align*}
\end{corollary}

\begin{proof}
Let $T_i$ be the pooled count in cell $i$. By \Cref{lem:perm-second-moment},
\begin{align*}
\E_\pi[(Z^\pi)^2\mid W]
\le
C
\left(
n\wedge \sum_{i=1}^k T_i(T_i-1)
\right),
\qquad
\E_\pi[Z^\pi\mid W]=0.
\end{align*}
Taking expectation gives $\E[Z^\pi]=0$. It remains to control the expected collision count. Since the pooled data contain $2n$ observations from $p$ and $2n$ observations from $q$,
\begin{align*}
\E[T_i(T_i-1)]
=
(2n)(2n-1)(p_i^2+q_i^2)
+8n^2p_iq_i.
\end{align*}
Therefore
\begin{align*}
\sum_{i=1}^k \E[T_i(T_i-1)]
\le
C n^2\sum_{i=1}^k(p_i^2+q_i^2).
\end{align*}
Because $p,q\in\mathcal P_{k,M}$,
\begin{align*}
\sum_{i=1}^k p_i^2\le \frac{M}{k},
\qquad
\sum_{i=1}^k q_i^2\le \frac{M}{k}.
\end{align*}
Thus $\E\sum_iT_i(T_i-1)\le C_M n^2/k$, while the conditional bound is always at most $Cn$. This proves the second-moment bound.
\end{proof}

The next corollary collects, for easy reference, the moment bounds established above for the observed split statistic and its permutation analogues.

\begin{corollary}[Collected moment bounds for $Z$ and $Z^\pi$]
\label{cor:mcperm-moment-bounds}
Fix $M\ge 1$, $p,q\in\mathcal P_{k,M}$, $\pi$ a uniform balanced relabeling independent of the data. Writing $W$ for the pooled discrete labels and $\lambda^\pi$ for the induced four-block assignment, set $Z^\pi:=Z_n(W,\lambda^\pi)$. Then
\begin{align*}
\E[Z^\pi]=0,
\qquad
\Var(Z^\pi)\le C_M\left(n\wedge\frac{n^2}{k}\right).
\end{align*}
Moreover, $\E[Z]=0$ if $p=q$, and if $\|p-q\|_1\ge\tau$ then
\begin{align*}
\E[Z]
\ge
b_M\min\!\left\{
n\tau,\,\frac{n^2\tau^2}{k},\,\frac{n^{3/2}\tau^2}{\sqrt{k}}
\right\}.
\end{align*}
Finally,
\begin{align*}
\Var(Z)\le C_M\left(n\wedge\frac{n^2}{k}\right).
\end{align*}
\end{corollary}

\begin{proof}

The claims for $Z$ follow from \Cref{lem:split-statistic-mean,lem:bounded-hist-var-binomial}; the crude $\Var(Z)\le Cn$ bound follows from \Cref{lem:efron-stein} and $\GS(Z)\le 4$, where changing one raw observation moves one count in one of the four histograms from one bin to another. The claims for $Z^\pi$ follow from \Cref{cor:perm-second-moment-bounded-hist}.
\end{proof}

\subsubsection{Completion of the Proof of \CrefInTitle{Proposition}{prop:private-bounded-hist-perm}}
\label{app:proof:prop:private-bounded-hist-perm}

\begin{proof}

Throughout this proof, $n$ is the block size appearing in \Cref{alg:mcperm}. We split the samples from each distribution into two blocks of equal size $n$. Thus the actual number of samples from each of $p$ and $q$ is $2n$. This changes the sample complexity only by an absolute constant factor, which is absorbed into the constant $C(M,\gamma,\beta)$.

By construction, $B_{\mathrm{perm}} = \bigl\lceil 6\gamma^{-1}\log(2\beta^{-1})\bigr\rceil \ge 6\gamma^{-1}\log(2\beta^{-1}),$ and hence the Monte Carlo sample-size requirement in \citet[Theorem~4]{KimSchrab2026} is satisfied. The level and privacy guarantees for $\Psi_{\mathrm{dpperm},\gamma}$ are established separately in \Cref{prop:dpperm-validity-privacy}; here it remains to prove the asserted type~II bound.

Let $T(W,\lambda):=Z_n(W,\lambda)$.  By \Cref{lem:mult-bdhist-sensitivity}, $\GS(T)=\GS(Z)\le 4$. Hence Theorem~4 of \citet{KimSchrab2026}, specialized to the pure-DP case $\delta=0$, shows that it is enough to verify
\begin{equation}
\E[Z]-\E[Z^\pi]
\ge
C_1
\sqrt{\Var(Z)+\Var(Z^\pi)}
+
C_2\,\varepsilon^{-1},
\label{eq:mcperm-power-cond}
\end{equation}
for constants $C_1,C_2>0$ depending only on $\gamma,\beta$ and $\GS(Z)=4$.

By \Cref{cor:mcperm-moment-bounds}, $\E[Z^\pi]=0$.  We verify \eqref{eq:mcperm-power-cond} separately according to the relative sizes of the block size and the alphabet.

\smallskip\noindent\emph{Sparse regime: $n\le k$.}
Since $\tau\le2$, the lower bound on $\E[Z]$ in \Cref{cor:mcperm-moment-bounds} may be written, after decreasing $b_M$ if necessary, as
\begin{align*}
\E[Z]\ge b_M\frac{n^2\tau^2}{k}.
\end{align*}
The same corollary gives
\begin{align*}
\Var(Z)+\Var(Z^\pi)\le C_M\frac{n^2}{k}.
\end{align*}
Thus the fluctuation term in \eqref{eq:mcperm-power-cond} is absorbed by $\E[Z]$ as soon as
\begin{align*}
n\ge C_M\frac{\sqrt{k}}{\tau^2},
\end{align*}
after enlarging $C_M$.  The privacy term is absorbed provided
\begin{align*}
\frac{n^2\tau^2}{k}\ge \frac{C_M}{\varepsilon},
\qquad\text{equivalently}\qquad
n\ge C_M\frac{\sqrt{k}}{\tau\sqrt{\varepsilon}}.
\end{align*}

\smallskip\noindent\emph{Dense regime: $n>k$.}
Here \Cref{cor:mcperm-moment-bounds} gives
\begin{align*}
\E[Z]\ge
b_M
\min\left\{
n\tau,\,
\frac{n^{3/2}\tau^2}{\sqrt{k}}
\right\},
\qquad
\Var(Z)+\Var(Z^\pi)\le C_M n.
\end{align*}
Under the condition $n\ge C_M\sqrt{k}/\tau^2$, both terms in the minimum are at least a sufficiently large multiple of $\sqrt n$: indeed,
\begin{align*}
n\tau\ge C_M^{1/2} k^{1/4}\sqrt n
\qquad\text{and}\qquad
\frac{n^{3/2}\tau^2}{\sqrt{k}}\ge C_M\sqrt n .
\end{align*}
Since $k\ge1$, increasing $C_M$ makes the fluctuation term in \eqref{eq:mcperm-power-cond} smaller than, say, one half of $\E[Z]$. Similarly, the privacy term is absorbed whenever
\begin{align*}
n\tau \ge \frac{C_M}{\varepsilon}
\qquad\text{and}\qquad
\frac{n^{3/2}\tau^2}{\sqrt{k}}\ge \frac{C_M}{\varepsilon},
\end{align*}
or equivalently
\begin{align*}
n\ge C_M\frac{1}{\tau\varepsilon}
\qquad\text{and}\qquad
n\ge C_M\frac{k^{1/3}}{\tau^{4/3}\varepsilon^{2/3}}.
\end{align*}
Combining the sparse and dense regimes, and increasing the constant in the sample-size condition if needed, proves \eqref{eq:mcperm-power-cond}. Therefore \citet[Theorem~4]{KimSchrab2026} yields
\begin{align*}
\Pbb_{H_1}\bigl(\Psi_{\mathrm{dpperm},\gamma}=0\bigr)\le \beta.
\end{align*}
This proves the proposition.
\end{proof}

\subsection{Proof of \CrefInTitle{Theorem}{thm:holder-l1-private-upper-d}}
\label{app:proofs-continuous-upper}
\begin{proof}
Fix a resolution $\kappa\in\mathbb N$, and write $[0,1]^d=\bigsqcup_{j\in[\kappa]^d}C_j$ for the regular partition into cells of volume $\kappa^{-d}$.  Let $k=\kappa^d$, $N=N_X\wedge N_Y$, and $m=\lfloor N/2\rfloor$, the block size available to the split-histogram statistic after sample splitting.

For $u\in\cF_s^d(L,M)$, define $p_u^{(\kappa)}\in\Delta_k$ by
\begin{align*}
p_u^{(\kappa)}(j)=\int_{C_j}u(x)\,dx,\qquad j\in[\kappa]^d .
\end{align*}
Writing $|C_j|$ for the Lebesgue volume of $C_j$, the uniform bound on $u$ gives
\begin{align*}
p_u^{(\kappa)}(j)\le M |C_j|=\frac{M}{k},
\end{align*}
and therefore $p_u^{(\kappa)}\in\mathcal P_{k,M}$.  For $h=f-g$, the cell-average operator satisfies
\begin{align*}
\|W_{\kappa,d}(h)\|_1
=
\sum_{j\in[\kappa]^d}\left|\int_{C_j}(f-g)\right|
=
\|p_f^{(\kappa)}-p_g^{(\kappa)}\|_1 .
\end{align*}
Since $h\in\cH_s^d(2L)$, \Cref{lem:binning-approx} yields
\begin{align*}
\|p_f^{(\kappa)}-p_g^{(\kappa)}\|_1
\ge
\frac12\|f-g\|_1-C_{\mathrm{bin},s}\kappa^{-s}.
\end{align*}
Consequently, on alternatives satisfying $\|f-g\|_1\ge r$,
\begin{equation}
\|p_f^{(\kappa)}-p_g^{(\kappa)}\|_1
\ge
\frac r2-C_{\mathrm{bin},s}\kappa^{-s}
=
\tau_\kappa(r).
\label{eq:discrete-gap-d}
\end{equation}

Assume $\bar\tau_\kappa(r)=\min\{1,\tau_\kappa(r)\}>0$.  Binning each observation by its cell label is a deterministic preprocessing map: neighboring continuous datasets are sent to either identical or neighboring discrete datasets.  Hence the composition of this binning map with the discrete $\varepsilon$-DP test remains $\varepsilon$-DP.  Applying \Cref{prop:private-bounded-hist-perm} with $k=\kappa^d$ and $\tau=\bar\tau_\kappa(r)$, the binned test has type~I error at most $\gamma$ and type~II error at most $\beta$, provided
\begin{align*}
m \ge C_0
\left(
\frac{\kappa^{d/2}}{\bar\tau_\kappa(r)^2}
+
\frac{\kappa^{d/2}}{\bar\tau_\kappa(r)\sqrt{\varepsilon}}
+
\frac{\kappa^{d/3}}{\bar\tau_\kappa(r)^{4/3}\varepsilon^{2/3}}
+
\frac{1}{\bar\tau_\kappa(r)\varepsilon}
\right),
\end{align*}
where $C_0$ depends only on $(d,s,L,M,\gamma,\beta)$.  Since $m\asymp N$, \eqref{eq:holder-private-sample-cond-d} implies this display after increasing the constant in the theorem.  Under $H_0$, the binned laws coincide; under $H_1$, \eqref{eq:discrete-gap-d} gives the required $\bar\tau_\kappa(r)$-separation.  This proves the fixed-resolution guarantee.

It remains to optimize over the resolution.  For fixed $\kappa\in\mathbb N$, set
\begin{align*}
R_\kappa(N,\varepsilon)
:=
\kappa^{-s}
+
\frac{\kappa^{d/4}}{\sqrt{N}}
+
\frac{\kappa^{d/2}}{N\sqrt{\varepsilon}}
+
\frac{\kappa^{d/4}}{N^{3/4}\sqrt{\varepsilon}}
+
\frac{1}{N\varepsilon}.
\end{align*}
We first verify that a sufficiently large multiple of $R_\kappa$ implies the sample-size condition for this fixed resolution.  Let
\begin{align*}
S_\kappa(r)
:=
\frac{\kappa^{d/2}}{\bar\tau_\kappa(r)^2}
+
\frac{\kappa^{d/2}}{\bar\tau_\kappa(r)\sqrt{\varepsilon}}
+
\frac{\kappa^{d/3}}{\bar\tau_\kappa(r)^{4/3}\varepsilon^{2/3}}
+
\frac{1}{\bar\tau_\kappa(r)\varepsilon}.
\end{align*}
This is the quantity appearing in the fixed-resolution requirement $m\ge C_0S_\kappa(r)$.  We claim that, for a sufficiently large constant $C_1$, this requirement follows from
\begin{align*}
0<r\le2,
\qquad
r\ge C_1R_\kappa(N,\varepsilon).
\end{align*}
If $C_1R_\kappa(N,\varepsilon)>2$, this sufficient condition is vacuous. Otherwise, take $C_1\ge 4C_{\mathrm{bin},s}$. Since $R_\kappa(N,\varepsilon)\ge \kappa^{-s}$, we have
\begin{align*}
C_{\mathrm{bin},s}\kappa^{-s}\le \frac r4,
\qquad
\tau_\kappa(r)=\frac r2-C_{\mathrm{bin},s}\kappa^{-s}\ge \frac r4.
\end{align*}
As $r\le2$, it follows that
\begin{align*}
\bar\tau_\kappa(r)^{-1}\le 4r^{-1}.
\end{align*}
Therefore
\begin{align*}
S_\kappa(r)
\le
16\frac{\kappa^{d/2}}{r^2}
+
4\frac{\kappa^{d/2}}{r\sqrt{\varepsilon}}
+
4^{4/3}\frac{\kappa^{d/3}}{r^{4/3}\varepsilon^{2/3}}
+
4\frac{1}{r\varepsilon}.
\end{align*}
The four terms in $r\ge C_1R_\kappa(N,\varepsilon)$ imply, respectively,
\begin{align*}
\frac{\kappa^{d/2}}{r^2}\le C_1^{-2}N,\qquad
\frac{\kappa^{d/2}}{r\sqrt{\varepsilon}}\le C_1^{-1}N,
\end{align*}
and
\begin{align*}
\frac{\kappa^{d/3}}{r^{4/3}\varepsilon^{2/3}}\le C_1^{-4/3}N,\qquad
\frac{1}{r\varepsilon}\le C_1^{-1}N.
\end{align*}
Combining the preceding displays gives
\begin{align*}
S_\kappa(r)
\le
\left\{
16C_1^{-2}
+8C_1^{-1}
+4^{4/3}C_1^{-4/3}
\right\}N .
\end{align*}
Taking $C_1$ larger if necessary makes the displayed constant at most $(3C_0)^{-1}$.  Hence
\begin{align*}
S_\kappa(r)\le \frac{N}{3C_0}.
\end{align*}
Since $m=\lfloor N/2\rfloor\ge N/3$ for $N\ge2$, this yields
\begin{align*}
m\ge C_0S_\kappa(r).
\end{align*}
The finitely many cases $N<2$ are absorbed by increasing the constant.

We have proved that, for every fixed $\kappa$, the binned test at resolution $\kappa$ succeeds whenever $r\ge C_1R_\kappa(N,\varepsilon)$.  Now let
\begin{align*}
R_*:=\inf_{\kappa\in\mathbb N}R_\kappa(N,\varepsilon),
\end{align*}
and choose $\kappa$ such that
\begin{align*}
R_\kappa(N,\varepsilon)\le 2R_*.
\end{align*}
Then $r\ge 2C_1R_*$ implies $r\ge C_1R_\kappa(N,\varepsilon)$, so the same test has the desired type~I and type~II guarantees.  Renaming $2C_1$ as $C'$ gives \eqref{eq:holder-private-separation-d}.
\end{proof}

\subsection{Proof of \CrefInTitle{Corollary}{cor:holder-l1-private-explicit-rate}}
\label{app:proof:cor-holder-private-explicit-rate}

\begin{proof}

By \Cref{thm:holder-l1-private-upper-d}, the separation radius is bounded above by
\begin{align*}
\inf_{\kappa\in\mathbb N}
\left\{
\kappa^{-s}
+
\frac{\kappa^{d/4}}{\sqrt N}
+
\frac{\kappa^{d/2}}{N\sqrt\varepsilon}
+
\frac{\kappa^{d/4}}{N^{3/4}\sqrt\varepsilon}
+
\frac1{N\varepsilon}
\right\}.
\end{align*}
Define
\begin{align*}
a_N:=N^{-1/2}+N^{-3/4}\varepsilon^{-1/2},
\qquad
b_N:=(N\sqrt\varepsilon)^{-1}.
\end{align*}
Grouping the middle three terms, it suffices to bound
\begin{align*}
\inf_{\kappa\in\mathbb N}
\left\{
\kappa^{-s}
+
a_N\kappa^{d/4}
+
b_N\kappa^{d/2}
\right\}
+
(N\varepsilon)^{-1}.
\end{align*}
\paragraph{Reduction to a one-dimensional problem.}
Since optimizing over $\kappa\in\mathbb{N}$ versus over $\kappa>0$ changes the infimum by at most a constant factor, substitute $x:=\kappa^{d/4}$ and set $\alpha:=4s/d$, reducing the problem to bounding
\begin{align*}
\inf_{x\ge 1} F(x),
\qquad
F(x):=x^{-\alpha}+a_Nx+b_Nx^2.
\end{align*}
\paragraph{Reduction to small parameters.}
If $a_N>1$ or $b_N>1$, then $a_N^{\alpha/(\alpha+1)}\vee b_N^{\alpha/(\alpha+2)}\gtrsim 1$, while $\|f-g\|_1\le 2$.  The claimed bound then holds after enlarging $C$. We therefore restrict to $a_N\le 1$ and $b_N\le 1$ for the remainder.

\paragraph{Evaluating the infimum.}
Let
\begin{align*}
x_1:=a_N^{-1/(\alpha+1)},
\qquad
x_2:=b_N^{-1/(\alpha+2)},
\qquad
x_\star:=x_1\wedge x_2.
\end{align*}
Since $a_N,b_N\le 1$, we have $x_\star\ge 1$. Moreover,
\begin{align*}
x_\star^{-\alpha}
=
x_1^{-\alpha}\vee x_2^{-\alpha}
=
a_N^{\alpha/(\alpha+1)}\vee b_N^{\alpha/(\alpha+2)}.
\end{align*}
The inequalities $x_\star\le x_1$ and $x_\star\le x_2$ give
\begin{align*}
a_Nx_\star\le a_Nx_1=a_N^{\alpha/(\alpha+1)},
\qquad
b_Nx_\star^2\le b_Nx_2^2=b_N^{\alpha/(\alpha+2)}.
\end{align*}
Summing all three contributions,
\begin{align*}
F(x_\star)
\le
3\Bigl(a_N^{\alpha/(\alpha+1)}\vee b_N^{\alpha/(\alpha+2)}\Bigr),
\end{align*}
which implies $\inf_{x\ge 1}F(x)\lesssim a_N^{\alpha/(\alpha+1)}\vee b_N^{\alpha/(\alpha+2)}$.

\paragraph{Translating back to the original parameters.}
The identities $\alpha/(\alpha+1)=4s/(4s+d)$ and $\alpha/(\alpha+2)=2s/(2s+d)$ yield
\begin{align*}
\inf_{\kappa\in\mathbb N}
\left\{
\kappa^{-s}
+
a_N\kappa^{d/4}
+
b_N\kappa^{d/2}
\right\}
\lesssim
a_N^{4s/(4s+d)}
\vee
b_N^{2s/(2s+d)}.
\end{align*}
Applying the subadditivity $(u+v)^t\lesssim u^t+v^t$ (valid for all $t>0$),
\begin{align*}
a_N^{4s/(4s+d)}
=
\bigl(N^{-1/2}+N^{-3/4}\varepsilon^{-1/2}\bigr)^{4s/(4s+d)}
\lesssim
N^{-2s/(4s+d)}
+
N^{-3s/(4s+d)}\varepsilon^{-2s/(4s+d)},
\end{align*}
while $b_N^{2s/(2s+d)}=(N\sqrt\varepsilon)^{-2s/(2s+d)}$. Substituting into \Cref{thm:holder-l1-private-upper-d} and taking the maximum of the four nonneg\-ative terms (their sum is within a factor of~$4$ of the maximum) yields \eqref{eq:explicit-holder-private-rate-max}.
\end{proof}

\subsection{Proof of \CrefInTitle{Proposition}{prop:rb-linear-time-main-text}}
\label{app:proof:prop:rb-linear-time-main-text}
This proof derives a closed form for the Rao--Blackwellized statistic and its linear-time evaluation. We first obtain a cellwise representation of $\overline Z$, and then prove \Cref{prop:rb-linear-time-main-text} by showing that the resulting expression can be evaluated in $O(N_X+N_Y)$ time once the complete bin counts are computed.

Let $B_r^X,B_r^Y\in[k]$ denote the binned observations from the two samples. For each cell $i\in[k]$, let
\begin{align*}
T_i^X:=\sum_{r=1}^{N_X}\mathbf 1\{B_r^X=i\},
\qquad
T_i^Y:=\sum_{r=1}^{N_Y}\mathbf 1\{B_r^Y=i\}
\end{align*}
be the complete bin counts in the two samples. Throughout this subsection, $h_{N,u,\ell}$ denotes the probability mass function of $\mathrm{Hypergeom}(N,u,\ell)$, where $N$ is the population size, $u$ is the number of successes, and $\ell$ is the number of draws:
\begin{align*}
h_{N,u,\ell}(a)
=
\frac{\binom{u}{a}\binom{N-u}{\ell-a}}{\binom{N}{\ell}},
\qquad
a=\max\{0,u+\ell-N\},\dots,\min\{u,\ell\}.
\end{align*}

\begin{lemma}[Closed-form Rao--Blackwellized statistic]
\label{lem:rb-closed-form-main}
Fix an integer $m$ such that $2m\le N_X\wedge N_Y$, and write
\begin{align*}
\overline Z:=\overline Z_m(W,\ell^{(0)}).
\end{align*}
Let $n_{u,v}:=\#\{i\in[k]:(T_i^X,T_i^Y)=(u,v)\}$, with marginal multiplicities $r_u:=\sum_v n_{u,v}$ and $c_v:=\sum_u n_{u,v}$. For each pair $(u,v)$, let
\begin{align*}
A_u\sim \mathrm{Hypergeom}(N_X,u,m),
\qquad
B_v\sim \mathrm{Hypergeom}(N_Y,v,m),
\end{align*}
independently, and define
\begin{align*}
\Gamma_{u,v}:=\E\bigl[|A_u-B_v|\bigr].
\end{align*}
For $a=0,\dots,2m$, define
\begin{align*}
d_m(a):=\E\,\bigl|2V-a\bigr|,
\qquad
V\sim \mathrm{Hypergeom}(2m,a,m).
\end{align*}
For $N\ge 2m$ and $u\in\{0,\dots,N\}$, set
\begin{align*}
D_{N,u}
:=
\sum_{a=0}^{2m} h_{N,u,2m}(a)\,d_m(a),
\qquad
H_{N,u}(t)
:=
\Pbb\{\mathrm{Hypergeom}(N,u,m)\ge t\}.
\end{align*}
The dependence of $D_{N,u}$ and $H_{N,u}$ on the fixed block size $m$ is suppressed. Then
\begin{align*}
\overline Z
=
\sum_{u,v} n_{u,v}\,\bar z(u,v),
\qquad\text{where}\quad
\bar z(u,v)
=
2\Gamma_{u,v}-D_{N_X,u}-D_{N_Y,v},
\end{align*}
and hence
\begin{align*}
\overline Z
=
2\sum_{u,v} n_{u,v}\,\Gamma_{u,v}
-\sum_u r_u\,D_{N_X,u}
-\sum_v c_v\,D_{N_Y,v}.
\end{align*}
Moreover,
\begin{align*}
\Gamma_{u,v}
=
\frac{mu}{N_X}+\frac{mv}{N_Y}
-2\sum_{t=1}^m H_{N_X,u}(t)\,H_{N_Y,v}(t).
\end{align*}
\end{lemma}

\begin{proof}
By definition, for any refinement $\lambda\in\Lambda_m(\ell^{(0)})$, writing $C_i^{G,a}=C_i^{G,a}(W,\lambda)$ within this display,
\begin{align*}
Z_m(W,\lambda)
=
\sum_{i=1}^k
\Bigl(
|C_i^{X,1}-C_i^{Y,1}|
+
|C_i^{X,2}-C_i^{Y,2}|
-
|C_i^{X,1}-C_i^{X,2}|
-
|C_i^{Y,1}-C_i^{Y,2}|
\Bigr).
\end{align*}
Hence it suffices to evaluate the conditional expectation of the $i$-th summand.

Fix $i\in[k]$, and condition on $T_i^X=u$ and $T_i^Y=v$. By exchangeability of the samples within each population, the random split depends on the data only through the counts $u$ and $v$.

First, for each block $b=1,2$, conditional on $T_i^X=u$,
\begin{align*}
C_i^{X,b}(W,\lambda) \mid T_i^X=u \sim \mathrm{Hypergeom}(N_X,u,m),
\end{align*}
and similarly
\begin{align*}
C_i^{Y,b}(W,\lambda) \mid T_i^Y=v \sim \mathrm{Hypergeom}(N_Y,v,m).
\end{align*}
Although $C_i^{X,1}$ and $C_i^{X,2}$ are not independent, this is irrelevant for the expectation of the two cross terms. The $X$- and $Y$-refinements are independent, so each cross term has expectation $\Gamma_{u,v}$. Hence
\begin{align*}
\E\!\left[
|C_i^{X,1}(W,\lambda)-C_i^{Y,1}(W,\lambda)|
+
|C_i^{X,2}(W,\lambda)-C_i^{Y,2}(W,\lambda)|
\,\middle|\,
T_i^X=u,\ T_i^Y=v
\right]
=
2\Gamma_{u,v}.
\end{align*}
Next, for the within-$X$ term, let
\begin{align*}
S_i^{X,\lambda}
:=
C_i^{X,1}(W,\lambda)+C_i^{X,2}(W,\lambda).
\end{align*}
Then
\begin{align*}
S_i^{X,\lambda} \mid T_i^X=u \sim \mathrm{Hypergeom}(N_X,u,2m).
\end{align*}
Conditional on $S_i^{X,\lambda}=a$, the $a$ selected category-$i$ observations are allocated between the two blocks of size $m$. Thus $C_i^{X,1}(W,\lambda)\mid(S_i^{X,\lambda}=a,T_i^X=u) \sim \mathrm{Hypergeom}(2m,a,m)$ and $C_i^{X,2}(W,\lambda)=a-C_i^{X,1}(W,\lambda)$, so the inner conditional expectation is $d_m(a)$. Taking expectation with respect to $S_i^{X,\lambda}\mid T_i^X=u$ yields
\begin{align*}
\E\!\left[
|C_i^{X,1}(W,\lambda)-C_i^{X,2}(W,\lambda)|
\,\middle|\,
T_i^X=u
\right]
&=
\sum_{a=0}^{2m} h_{N_X,u,2m}(a)\,d_m(a) \\
&=
D_{N_X,u}.
\end{align*}
The same argument yields
\begin{align*}
\E\!\left[
|C_i^{Y,1}(W,\lambda)-C_i^{Y,2}(W,\lambda)|
\,\middle|\,
T_i^Y=v
\right]
=
D_{N_Y,v}.
\end{align*}
Combining the preceding identities gives
\begin{align*}
\E\!\left[
\begin{aligned}
&|C_i^{X,1}(W,\lambda)-C_i^{Y,1}(W,\lambda)| \\
&+
|C_i^{X,2}(W,\lambda)-C_i^{Y,2}(W,\lambda)| \\
&-
|C_i^{X,1}(W,\lambda)-C_i^{X,2}(W,\lambda)| \\
&-
|C_i^{Y,1}(W,\lambda)-C_i^{Y,2}(W,\lambda)|
\end{aligned}
\,\middle|\,
T_i^X=u,\ T_i^Y=v
\right]
=
2\Gamma_{u,v}-D_{N_X,u}-D_{N_Y,v}.
\end{align*}
Summing over $i=1,\dots,k$ proves the stated formula for $\overline Z$.

It remains to identify the alternative representation of $\Gamma_{u,v}$. Let $A_u$ and $B_v$ denote the independent hypergeometric random variables from the lemma statement, so that $\Gamma_{u,v}=\E[|A_u-B_v|]$. Using
\begin{align*}
|A_u-B_v|=A_u+B_v-2\min\{A_u,B_v\},
\end{align*}
we obtain
\begin{align*}
\Gamma_{u,v}
=
\E[A_u]+\E[B_v]-2\E[\min\{A_u,B_v\}].
\end{align*}
Since $A_u,B_v\in\{0,\dots,m\}$,
\begin{align*}
\E[\min\{A_u,B_v\}]
=
\sum_{t=1}^m \Pbb(A_u\ge t,\ B_v\ge t).
\end{align*}
By independence,
\begin{align*}
\Pbb(A_u\ge t,\ B_v\ge t)=H_{N_X,u}(t)H_{N_Y,v}(t),
\end{align*}
while
\begin{align*}
\E[A_u]=\frac{mu}{N_X},
\qquad
\E[B_v]=\frac{mv}{N_Y}.
\end{align*}
Substituting these identities yields the claimed expression for $\Gamma_{u,v}$.
\end{proof}

We next prove that the closed-form expression for $\overline Z$ can be evaluated in linear time once the complete bin counts are computed and the implementation details are given in \Cref{alg:rb-linear-time}.

\paragraph{Proof of \Cref{prop:rb-linear-time-main-text}.}
\begin{proof}
We write the complete bin counts as $(T_i^X,T_i^Y)$, and let
\begin{align*}
n_{u,v}:=\#\{i:(T_i^X,T_i^Y)=(u,v)\}.
\end{align*}
Only cells with $u+v>0$ need to be stored. Define
\begin{align*}
\mathcal P:=\{(u,v):n_{u,v}>0,\ u+v>0\},
\\
\mathcal A_X:=\{u:\exists\,v,\ (u,v)\in\mathcal P\},
\\
\mathcal A_Y:=\{v:\exists\,u,\ (u,v)\in\mathcal P\},
\end{align*}
and let
\begin{align*}
r_u:=\sum_v n_{u,v},
\qquad
c_v:=\sum_u n_{u,v}.
\end{align*}
Since $\bar z(0,0)=0$, the empty cells do not contribute. Hence, by \Cref{lem:rb-closed-form-main},
\begin{align*}
\overline Z
=
2\sum_{(u,v)\in\mathcal P} n_{u,v}\Gamma_{u,v}
-
\sum_{u\in\mathcal A_X} r_uD_{N_X,u}
-
\sum_{v\in\mathcal A_Y} c_vD_{N_Y,v}.
\end{align*}
The table $\mathcal P$, together with the multiplicities $n_{u,v},r_u,c_v$, is obtained in one pass over the nonempty cells. Since the number of nonempty cells is at most $N_X+N_Y$, this preprocessing is linear.

We next record the elementary cost bound used throughout. For $H\sim\mathrm{Hypergeom}(N,a,\ell)$, its support is the interval
\begin{align*}
\{\max(0,a+\ell-N),\dots,\min(a,\ell)\},
\end{align*}
whose length is
\begin{align*}
w_{N,\ell}(a):=\min\{a,\ell,N-a,N-\ell\}+1\le a+1.
\end{align*}
Using the standard one-step recursion for consecutive hypergeometric masses, the whole row $h_{N,a,\ell}$ can be generated and normalized in $O(w_{N,\ell}(a))$ arithmetic operations. Explicitly, for consecutive support points,
\begin{align*}
\frac{h_{N,a,\ell}(t+1)}{h_{N,a,\ell}(t)}
=
\frac{\binom{a}{t+1}\binom{N-a}{\ell-t-1}}
{\binom{a}{t}\binom{N-a}{\ell-t}}
=
\frac{(a-t)(\ell-t)}
{(t+1)(N-a-\ell+t+1)}.
\end{align*}
Thus one may start with an arbitrary positive unnormalized value at the left endpoint of the support, generate all consecutive unnormalized masses by this ratio, and finally divide by their sum. This takes time proportional to $w_{N,\ell}(a)$. In the same pass one obtains the prefix sums
\begin{align*}
F_a(x):=\sum_{r\le x} h_{N,a,m}(r),
\qquad
M_a(x):=\sum_{r\le x} r\,h_{N,a,m}(r),
\end{align*}
with the usual boundary conventions outside the support. The quantities
\begin{align*}
D_{N,a}=\sum_{j=0}^{2m} h_{N,a,2m}(j)d_m(j)
\end{align*}
are then computed in $O(w_{N,2m}(a))$ time once the array $d_m(0),\dots,d_m(2m)$ is available.

The array $d_m$ itself is computed in $O(m)$ time. Indeed, let $(\xi_1,\dots,\xi_{2m})$ be a uniformly random permutation of $m$ copies of $+1$ and $m$ copies of $-1$, and set $S_s=\xi_1+\cdots+\xi_s$. Then $d_m(s)=\E|S_s|$, and
\begin{align*}
d_m(0)=0,
\qquad
d_m(s+1)
=
\left(1-\frac{1}{2m-s}\right)d_m(s)+\rho_m(s),
\end{align*}
where $\rho_m(s)=\Pbb(S_s=0)$. Moreover $\rho_m(s)=0$ for odd $s$, while for $s=2j$,
\begin{align*}
\rho_m(2j)
=
\frac{\binom{m}{j}^2}{\binom{2m}{2j}},
\qquad
\rho_m(0)=1,
\qquad
\rho_m(2j+2)
=
\rho_m(2j)
\frac{(m-j)(2j+1)}{(j+1)(2m-2j-1)}.
\end{align*}
To justify the recurrence, condition on $S_s=z$. Given $S_s=z$, the numbers of $+1$'s and $-1$'s already used are $(s+z)/2$ and $(s-z)/2$, respectively. Hence, for $0\le s\le 2m-1$,
\begin{align*}
\Pbb(\xi_{s+1}=+1\mid S_s=z)
=
\frac{2m-s-z}{2(2m-s)},
\qquad
\Pbb(\xi_{s+1}=-1\mid S_s=z)
=
\frac{2m-s+z}{2(2m-s)}.
\end{align*}
A direct calculation gives, for $z\ne0$,
\begin{align*}
\E\bigl[|z+\xi_{s+1}|\,\bigm|\,S_s=z\bigr]
=
\left(1-\frac1{2m-s}\right)|z|,
\end{align*}
whereas for $z=0$ the conditional expectation equals $1$. Therefore
\begin{align*}
\E\bigl[|S_{s+1}|\,\bigm|\,S_s\bigr]
=
\left(1-\frac1{2m-s}\right)|S_s|
+\mathbf 1\{S_s=0\},
\end{align*}
and taking expectations gives the displayed update for $d_m$. The formula for $\rho_m$ follows because $S_s=0$ is impossible for odd $s$, while for $s=2j$ it requires exactly $j$ of the first $2j$ positions to be occupied by the $m$ copies of $+1$ and $j$ by the $m$ copies of $-1$. The ratio update is obtained from
\begin{align*}
\frac{\rho_m(2j+2)}{\rho_m(2j)}
=
\left(\frac{m-j}{j+1}\right)^2
\frac{(2j+2)(2j+1)}{(2m-2j)(2m-2j-1)}
=
\frac{(m-j)(2j+1)}{(j+1)(2m-2j-1)}.
\end{align*}
Hence all $d_m(s)$'s are obtained in linear time and memory.

For each active $u\in\mathcal A_X$, we tabulate the rows $h_{N_X,u,m}$, $h_{N_X,u,2m}$, their prefix arrays, and $D_{N_X,u}$. The total cost is
\begin{align*}
O\!\left(
\sum_{u\in\mathcal A_X}
\{w_{N_X,m}(u)+w_{N_X,2m}(u)\}
\right).
\end{align*}
The analogous construction is performed for all $v\in\mathcal A_Y$.

It remains to compute $\Gamma_{u,v}$ for $(u,v)\in\mathcal P$. Suppose we sum over the support of $h_{N_X,u,m}$. If $Y\sim h_{N_Y,v,m}$, then for any integer $x$,
\begin{align*}
\E|x-Y|
=
\frac{mv}{N_Y}
-2M_v^{(Y)}(x)
+x\bigl(2F_v^{(Y)}(x)-1\bigr).
\end{align*}
Therefore
\begin{align*}
\Gamma_{u,v}
=
\sum_x h_{N_X,u,m}(x)
\left[
\frac{mv}{N_Y}
-2M_v^{(Y)}(x)
+x\bigl(2F_v^{(Y)}(x)-1\bigr)
\right].
\end{align*}
The identity for $\E|x-Y|$ is simply the split of the absolute value at $Y=x$. Writing $\mu_Y=\E Y=mv/N_Y$,
\begin{align*}
\begin{aligned}
\E|x-Y|
&=
\E[(x-Y)\mathbf 1\{Y\le x\}]
+
\E[(Y-x)\mathbf 1\{Y>x\}]\\
&=
xF_v^{(Y)}(x)-M_v^{(Y)}(x)
+
\mu_Y-M_v^{(Y)}(x)-x\{1-F_v^{(Y)}(x)\}\\
&=
\frac{mv}{N_Y}
-2M_v^{(Y)}(x)
+x\bigl(2F_v^{(Y)}(x)-1\bigr).
\end{aligned}
\end{align*}
Conditioning on $X\sim h_{N_X,u,m}$ gives the displayed formula for $\Gamma_{u,v}=\E|X-Y|$. This costs $O(w_{N_X,m}(u))$. If $w_{N_Y,m}(v)<w_{N_X,m}(u)$, we use the symmetric expression instead. Hence the total cost of the pairwise pass is
\begin{align*}
O\!\left(
\sum_{(u,v)\in\mathcal P}
\min\{w_{N_X,m}(u),w_{N_Y,m}(v)\}
\right).
\end{align*}

We now bound these sums. Since $w_{N,\ell}(a)\le a+1$,
\begin{align*}
\sum_{u\in\mathcal A_X}
\{w_{N_X,m}(u)+w_{N_X,2m}(u)\}
\le
2\sum_{u\in\mathcal A_X}(u+1).
\end{align*}
Also,
\begin{align*}
\sum_{u\in\mathcal A_X} u
\le
\sum_{(u,v)\in\mathcal P} u\,n_{u,v}
=
N_X,
\qquad
|\mathcal A_X|\le N_X+1,
\end{align*}
so the total $X$-row cost is $O(N_X)$. Similarly, the total $Y$-row cost is $O(N_Y)$.

For the pairwise term,
\begin{align*}
\min\{w_{N_X,m}(u),w_{N_Y,m}(v)\}\le \min\{u,v\}+1.
\end{align*}
Consequently,
\begin{align*}
\sum_{(u,v)\in\mathcal P}\min\{u,v\}
\le
\sum_{(u,v)\in\mathcal P} u
\le
\sum_{(u,v)\in\mathcal P} u\,n_{u,v}
=
N_X.
\end{align*}
Moreover, since every $(u,v)\in\mathcal P$ satisfies $u+v\ge1$ and $n_{u,v}\ge1$,
\begin{align*}
|\mathcal P|
\le
\sum_{(u,v)\in\mathcal P}(u+v)n_{u,v}
=
N_X+N_Y.
\end{align*}
Hence the pairwise pass costs $O(N_X+N_Y)$.

Finally, computing $d_m$ costs $O(m)$, and $m\le (N_X\wedge N_Y)/2$. Combining the preprocessing, row tabulation, and pairwise evaluation costs, the total arithmetic complexity is $O(N_X+N_Y)$.

The memory bound follows from the same estimates. The table $d_m$ costs $O(m)$, the stored hypergeometric rows and prefix arrays cost
\begin{align*}
O\!\left(
\sum_{u\in\mathcal A_X}
\{w_{N_X,m}(u)+w_{N_X,2m}(u)\}
+
\sum_{v\in\mathcal A_Y}
\{w_{N_Y,m}(v)+w_{N_Y,2m}(v)\}
\right)
=
O(N_X+N_Y).
\end{align*}
The sparse pair table has size $|\mathcal P|=O(N_X+N_Y)$. Thus the total memory usage is also $O(N_X+N_Y)$.
\end{proof}

\subsection{Proof of \CrefInTitle{Lemma}{lem:rb-basic-properties}}
\label{app:proof:lem-rb-basic-properties}

\begin{lemma}[Basic Rao--Blackwellization properties]
\label{lem:rb-basic-properties}
Let
\begin{align*}
\lambda^{(0)}\mid (W,\ell^{(0)})
\sim
\Unif(\Lambda_m(\ell^{(0)})),
\end{align*}
independently of all other randomness conditional on $(W,\ell^{(0)})$. Then
\begin{align*}
\overline Z_m(W,\ell^{(0)})
=
\E\!\left[
Z_m(W,\lambda^{(0)})
\mid W,\ell^{(0)}
\right].
\end{align*}
Consequently,
\begin{align*}
\E[\overline Z_m(W,\ell^{(0)})]
=
\E[Z_m(W,\lambda^{(0)})],
\qquad
\Var(\overline Z_m(W,\ell^{(0)}))
\le
\Var(Z_m(W,\lambda^{(0)})).
\end{align*}
Moreover, for every labeling $\ell\in\mathfrak L_{N_X,N_Y}$ with $N_X$ labels $X$ and $N_Y$ labels $Y$,
\begin{align*}
\sup_{\substack{w,w'\in[k]^{N_X+N_Y}\\ d_H(w,w')\le1}}
\left|
\overline Z_m(w,\ell)-\overline Z_m(w',\ell)
\right|
\le 4 .
\end{align*}
\end{lemma}

\begin{proof}

Let $Z^\lambda:=Z_m(W,\lambda^{(0)})$, where $\lambda^{(0)}\mid (W,\ell^{(0)})\sim\Unif(\Lambda_m(\ell^{(0)}))$. By definition,
\begin{align*}
\overline Z_m(W,\ell^{(0)})
=
\E[Z^\lambda\mid W,\ell^{(0)}].
\end{align*}
The identity $\E[\overline Z_m(W,\ell^{(0)})]=\E[Z^\lambda]$ is the tower property. Also, by the law of total variance,
\begin{align*}
\Var(Z^\lambda)
=
\E\!\left[\Var(Z^\lambda\mid W,\ell^{(0)})\right]
+ 
\Var\!\left(\E[Z^\lambda\mid W,\ell^{(0)}]\right),
\end{align*}
and therefore
\begin{equation*}
\Var(\overline Z_m(W,\ell^{(0)}))
=
\Var\!\left(\E[Z^\lambda\mid W,\ell^{(0)}]\right)
\le \Var(Z^\lambda).
\end{equation*}

Next fix $\ell\in\mathfrak L_{N_X,N_Y}$ and $w,w'\in[k]^{N_X+N_Y}$ with $d_H(w,w')\le1$. Couple both evaluations using the same random refinement $\lambda\sim\Unif(\Lambda_m(\ell))$. For each fixed refinement, changing one coordinate of the pooled binned label vector affects at most one of the four count vectors
\begin{align*}
X^\lambda,\ \widetilde X^\lambda,\ Y^\lambda,\ \widetilde Y^\lambda
\end{align*}
Only that count vector can change. If it does, one unit mass moves from one coordinate to another. Hence, by \Cref{lem:mult-bdhist-sensitivity}, for every fixed refinement,
\begin{align*}
|Z_m(w,\lambda)-Z_m(w',\lambda)|\le 4.
\end{align*}
Therefore,
\begin{align*}
\left|
\overline Z_m(w,\ell)-\overline Z_m(w',\ell)
\right|
\le
\E_{\lambda\sim\Unif(\Lambda_m(\ell))}
\!\left[
\left|Z_m(w,\lambda)-Z_m(w',\lambda)\right|
\right]
\le 4.
\end{align*}
Taking the supremum over $\ell,w,w'$ gives
\begin{equation}\label{eq:rb-sensitivity}
\sup_{\ell\in\mathfrak L_{N_X,N_Y}}
\sup_{\substack{w,w'\in[k]^{N_X+N_Y}\\ d_H(w,w')\le1}}
\left|
\overline Z_m(w,\ell)-\overline Z_m(w',\ell)
\right|
\le 4.
\end{equation}
\end{proof}

\subsection{Moment bounds for the Rao--Blackwellized statistic}
\label{app:proof:lem-rb-perm-moment-bounds}

\begin{lemma}[Permutation moment bounds for the Rao--Blackwellized statistic]
\label{lem:rb-perm-moment-bounds}
Set $m=\lfloor (N_X\wedge N_Y)/2\rfloor$. Let $\ell^\pi$ be a uniformly random relabeling of the pooled binned labels with class sizes $N_X,N_Y$, independent of the data, and define
\begin{align*}
\overline Z^\pi:=\overline Z_m(W,\ell^\pi),
\qquad
\overline Z^{(0)}:=\overline Z_m(W,\ell^{(0)}).
\end{align*}
There exists a universal constant $C>0$ such that, for arbitrary $p,q\in\Delta_k$,
\begin{align}
\E\!\left[\overline Z^\pi\mid W\right]&=0,
\label{eq:rb-perm-conditionally-centered}\\
\Var_{p,q}\!\left(\overline Z^{(0)}\right)
+\Var_{p,q}\!\left(\overline Z^\pi\right)
&\le Cm.
\label{eq:rb-unrestricted-variance}
\end{align}
There also exists a universal constant $c>0$ such that, whenever $p,q\in\Delta_k$ satisfy $\|p-q\|_1\ge\tau$,
\begin{align}
\E_{p,q}\!\left[\overline Z^{(0)}\right]
\ge c\min\!\left\{
m\tau,\frac{m^2\tau^2}{k},\frac{m^{3/2}\tau^2}{\sqrt{k}}
\right\}.
\label{eq:rb-unrestricted-mean}
\end{align}
For every $M\ge1$, there exists a constant $C_M>0$, depending only on $M$, such that, if $p,q\in\mathcal P_{k,M}$, then
\begin{align}
\Var_{p,q}\!\left(\overline Z^{(0)}\right)
+\Var_{p,q}\!\left(\overline Z^\pi\right)
\le C_M\min\!\left\{\frac{m^2}{k},m\right\}.
\label{eq:rb-flat-variance}
\end{align}
\end{lemma}

\begin{proof}
Let $\lambda^\pi\mid\ell^\pi\sim\Unif(\Lambda_m(\ell^\pi))$, and write $Z^{\pi,\lambda}:=Z_m(W,\lambda^\pi)$. Then
\begin{align*}
\overline Z^\pi=\E_\lambda[Z^{\pi,\lambda}\mid W,\ell^\pi].
\end{align*}
For fixed $W$, averaging over $(\ell^\pi,\lambda^\pi)$ is equivalent to assigning selected observations to four labeled blocks of size $m$, with the remaining observations unused. Swapping the second and third selected blocks is a bijection of this assignment space and changes the sign of the statistic, since $g(a,c,b,d)=-g(a,b,c,d)$. Hence
\begin{align*}
\E_{\ell^\pi}[\overline Z^\pi\mid W]
=\E_{\ell^\pi,\lambda}[Z^{\pi,\lambda}\mid W]=0,
\end{align*}
which proves \eqref{eq:rb-perm-conditionally-centered}.

By Jensen's inequality and the distribution-free branch of \Cref{lem:perm-second-moment},
\begin{align*}
\E[(\overline Z^\pi)^2]
\le \E[(Z^{\pi,\lambda})^2]
\le Cm.
\end{align*}
For the observed statistic, let $\lambda^{(0)}\mid\ell^{(0)}\sim\Unif(\Lambda_m(\ell^{(0)}))$ and write $Z^{0,\lambda}:=Z_m(W,\lambda^{(0)})$. By \Cref{lem:rb-basic-properties},
\begin{align*}
\E[\overline Z^{(0)}]=\E[Z^{0,\lambda}],
\qquad
\Var(\overline Z^{(0)})\le\Var(Z^{0,\lambda}).
\end{align*}
Changing one of the $4m$ observations changes the sensitivity-four split statistic by at most four. The Efron--Stein inequality therefore gives $\Var(Z^{0,\lambda})\le Cm$ without any restriction on $p$ or $q$. Together with the preceding permuted bound, this proves \eqref{eq:rb-unrestricted-variance}.

For the mean bound, set $\delta_i:=|p_i-q_i|$ and $s_i:=p_i+q_i$, and let $D_i$ denote the expected contribution of coordinate $i$ to $Z^{0,\lambda}$. By \Cref{lem:binomial-comparison-no-heavy},
\begin{align*}
D_i
\ge c\min\!\left\{
m\delta_i,
m^2\delta_i^2,
\frac{m^{3/2}\delta_i^2}{\sqrt{s_i}}
\right\},
\end{align*}
where the last term is interpreted as zero when $s_i=0$. Partition $[k]$ into $I_1,I_2,I_3$ according to which term attains the minimum, breaking ties arbitrarily. If $\sum_i\delta_i\ge\tau$, then some $I_a$ satisfies $\sum_{i\in I_a}\delta_i\ge\tau/3$. For $a=1$, summing the first terms gives a lower bound $cm\tau/3$. For $a=2$, Cauchy--Schwarz gives
\begin{align*}
\sum_{i\in I_2}\delta_i^2
\ge \frac{\tau^2}{9k}.
\end{align*}
For $a=3$, weighted Cauchy--Schwarz and $\sum_i s_i=2$ give
\begin{align*}
\sum_{i\in I_3}\frac{\delta_i^2}{\sqrt{s_i}}
\ge
\frac{(\sum_{i\in I_3}\delta_i)^2}
{\sum_{i\in I_3}\sqrt{s_i}}
\ge \frac{\tau^2}{9\sqrt{2k}}.
\end{align*}
Combining the three cases with expectation preservation under Rao--Blackwellization proves \eqref{eq:rb-unrestricted-mean}.

Now suppose that $p,q\in\mathcal P_{k,M}$. Conditioning on the selected pooled sequence and using the collision-sensitive branch of \Cref{lem:perm-second-moment}, as in \Cref{cor:perm-second-moment-bounded-hist}, gives
\begin{align*}
\Var(\overline Z^\pi)\le C_M\frac{m^2}{k}.
\end{align*}
For every fixed observed refinement, the four count vectors have the joint law
\begin{equation}\label{eq:rb-split-count-law}
X,\widetilde X\stackrel{\iid}{\sim}\mathrm{Mult}(m,p),
\qquad
Y,\widetilde Y\stackrel{\iid}{\sim}\mathrm{Mult}(m,q),
\end{equation}
and are mutually independent. Thus \Cref{lem:bounded-hist-var-binomial} and variance contraction imply
\begin{align*}
\Var(\overline Z^{(0)})\le 64M\frac{m^2}{k}.
\end{align*}
Combining these bounds with the distribution-free $O(m)$ bounds proves \eqref{eq:rb-flat-variance}.
\end{proof}

\subsection{Proof of
\texorpdfstring{\Cref{prop:private-closeness-rb}} {Proposition~\getrefnumber{prop:private-closeness-rb}}}
\label{app:proof:prop:private-closeness-rb}
\label{app:proof:prop:private-bounded-hist-rb}

\begin{proof}
Let
\begin{align*}
N:=N_X\wedge N_Y,
\qquad
m:=\left\lfloor\frac{N}{2}\right\rfloor,
\end{align*}
and write
\begin{align*}
T_0:=\overline Z_m(W,\ell^{(0)}),
\qquad
T^\pi:=\overline Z_m(W,\ell^\pi),
\end{align*}
where $\ell^\pi$ is a uniformly random labeling with class sizes $N_X,N_Y$, independent of the data.

By \Cref{lem:rb-basic-properties}, the statistic $\overline Z_m(W,\ell)$ has global sensitivity at most four, uniformly over all admissible labelings. Hence \Cref{lem:private-permutation-cited}, applied with Laplace scale $8/\varepsilon$, shows that the released decision is $\varepsilon$-differentially private. Under $p=q$, conditional on the unlabeled pooled sample, the observed labeling and the random permutation labelings are exchangeable. The same lemma therefore gives
\begin{align*}
\sup_{p\in\Delta_k}
\Pbb_{p,p}\!\left(\Psi_{\mathrm{rb\text{-}dpperm},\gamma}=1\right)
\le\gamma.
\end{align*}
Neither argument uses a bound on the cell probabilities.

We next establish the power guarantee over the unrestricted simplex. By \Cref{lem:rb-perm-moment-bounds}, there is a universal constant $c_0>0$ such that, whenever $\|p-q\|_1\ge\tau$,
\begin{align}
\mu_m:=\E_{p,q}[T_0]
\ge c_0\min\!\left\{
m\tau,
\frac{m^2\tau^2}{k},
\frac{m^{3/2}\tau^2}{\sqrt{k}}
\right\}.
\label{eq:rb-unrestricted-global-gap}
\end{align}

By \Cref{lem:rb-perm-moment-bounds},
\begin{align}
\E[T^\pi\mid W]&=0,
\label{eq:rb-unrestricted-perm-centered}\\
\Var_{p,q}(T_0)+\Var_{p,q}(T^\pi)&\le C_1m.
\label{eq:rb-unrestricted-total-variance}
\end{align}
For the stated choice of $B_{\mathrm{perm}}$, Theorem~4 of \citet{KimSchrab2026}, specialized to pure DP, implies that the type~II error is at most $\beta$ whenever
\begin{align}
\mu_m
\ge C_{\gamma,\beta}
\left(
\sqrt{\Var_{p,q}(T_0)+\Var_{p,q}(T^\pi)}
+\frac{1}{\varepsilon}
\right).
\label{eq:rb-private-power-condition}
\end{align}
Thus it is enough that
\begin{align}
\min\!\left\{
m\tau,
\frac{m^2\tau^2}{k},
\frac{m^{3/2}\tau^2}{\sqrt{k}}
\right\}
\gtrsim_{\gamma,\beta}\sqrt m+\frac{1}{\varepsilon}.
\label{eq:rb-unrestricted-signal-condition}
\end{align}

Comparing the three signal terms with $\sqrt m$ gives, respectively,
\begin{align*}
m&\gtrsim\frac{1}{\tau^2},
&m&\gtrsim\frac{k^{2/3}}{\tau^{4/3}},
&m&\gtrsim\frac{\sqrt{k}}{\tau^2}.
\end{align*}
The first requirement is implied by the third. Comparing the same terms with $\varepsilon^{-1}$ gives, respectively,
\begin{align*}
m&\gtrsim\frac{1}{\tau\varepsilon},
&m&\gtrsim\frac{\sqrt{k}}{\tau\sqrt{\varepsilon}},
&m&\gtrsim\frac{k^{1/3}}{\tau^{4/3}\varepsilon^{2/3}}.
\end{align*}
Since $m\asymp N$, these inequalities prove part~{\rm (i)}.

Now suppose that $p,q\in\mathcal P_{k,M}$. The sharper variance estimate in \Cref{lem:rb-perm-moment-bounds} gives
\begin{align}
\Var_{p,q}(T_0)+\Var_{p,q}(T^\pi)
\le C_M\min\!\left\{\frac{m^2}{k},m\right\}.
\label{eq:rb-flat-total-variance-proof}
\end{align}
Thus the stochastic fluctuation scale is bounded by
\begin{align*}
v_m:=\frac{m}{\sqrt{k}}\wedge\sqrt m.
\end{align*}
A direct calculation in the two cases $m\le k$ and $m\ge k$ shows that
\begin{align}
m\gtrsim\frac{\sqrt{k}}{\tau^2}
\quad\Longrightarrow\quad
\min\!\left\{
m\tau,
\frac{m^2\tau^2}{k},
\frac{m^{3/2}\tau^2}{\sqrt{k}}
\right\}
\gtrsim v_m.
\label{eq:rb-flat-nonprivate-comparison}
\end{align}
The comparisons with $\varepsilon^{-1}$ are unchanged from the unrestricted case. Hence \eqref{eq:rb-private-power-condition} holds whenever
\begin{align*}
m\gtrsim_{M,\gamma,\beta}
\frac{\sqrt{k}}{\tau^2}
+\frac{\sqrt{k}}{\tau\sqrt{\varepsilon}}
+\frac{k^{1/3}}{\tau^{4/3}\varepsilon^{2/3}}
+\frac{1}{\tau\varepsilon}.
\end{align*}
Using again $m\asymp N$ proves part~{\rm (ii)}.

\end{proof}

\subsection{Proof of \CrefInTitle{Lemma}{lem:gof-to-twosamp-cdp}}
\label{app:proofs-gof-reduction}
\label{app:proof:lem:gof-to-twosamp-cdp}

\begin{proof}
We show that each feasible two-sample separation is also feasible for the GOF problem. Suppose first that $N_X=N\le N_Y$, and let
\begin{align*}
\mathcal M:([0,1]^d)^{N_X+N_Y}\to\{0,1\}
\end{align*}
be an $\varepsilon$-DP two-sample test with level at most $\gamma$ and uniform type~II error at most $\beta$ at separation $r$. Given a GOF sample $X_{1:N}$, draw
\begin{align*}
U_1,\ldots,U_{N_Y}\stackrel{\iid}{\sim} f_{\mathrm{unif}}
\end{align*}
independently of $X_{1:N}$, and define
\begin{align*}
T(X_{1:N})=\mathcal M(X_{1:N},U_{1:N_Y}).
\end{align*}
The auxiliary sample is internal randomness of $T$, not private data.

If $x,x'\in([0,1]^d)^N$ are neighboring GOF datasets, then for every fixed $u$, the two inputs $(x,u)$ and $(x',u)$ are neighboring two-sample datasets. Let $\nu=P_{f_{\mathrm{unif}}}^{\otimes N_Y}$ be the law of the auxiliary sample. For every $A\subseteq\{0,1\}$,
\begin{align*}
\begin{aligned}
\Pbb\{T(x)\in A\}
&=\int \Pbb_{\mathcal M}\{\mathcal M(x,u)\in A\}\,\nu(du) \\
&\le
e^\varepsilon
\int \Pbb_{\mathcal M}\{\mathcal M(x',u)\in A\}\,\nu(du) \\
&=
e^\varepsilon \Pbb\{T(x')\in A\}.
\end{aligned}
\end{align*}
Thus $T$ is $\varepsilon$-DP.

Under the GOF null $f=f_{\mathrm{unif}}$, the pair $(X_{1:N},U_{1:N_Y})$ has the two-sample null law with common density $f_{\mathrm{unif}}$, so $T$ has level at most $\gamma$. If $\|f-f_{\mathrm{unif}}\|_1\ge r$, then $(f,f_{\mathrm{unif}})$ is a valid two-sample alternative at separation at least $r$. Therefore the type~II error of $T$ is bounded by the uniform type~II error of $\mathcal M$.

When $N_Y=N\le N_X$, the same construction applies with the roles of the two samples interchanged, using the two-sample alternative $(f_{\mathrm{unif}},f)$. Hence every feasible two-sample separation is feasible for the GOF problem. Taking infima over feasible $r$'s gives
\begin{align*}
r^*_{\mathrm{gof}}(N,\varepsilon)
\le
r^*_{\mathrm{2samp}}(N_X,N_Y,\varepsilon),
\end{align*}
as claimed.
\end{proof}

\subsection{Proof of \CrefInTitle{Theorem}{thm:gof-combined-lower}}
\label{app:lower-bounds}
\label{app:proof:thm:gof-combined-lower}

The lower bound is proved by four separate mechanisms, one for each term in the claimed maximum.  Regime~I proves the classical nonprivate smooth term $N^{-2s/(4s+d)}$.  Its smooth hypercube construction is also used in Regime~III, which gives the $(N^{3/2}\varepsilon)^{-2s/(4s+d)}$ term through a transport bound.  Regime~II uses a separate cellwise construction to prove the $(N\sqrt\varepsilon)^{-2s/(2s+d)}$ term.  Regime~IV uses a fixed perturbation to prove the linear privacy barrier $(N\varepsilon)^{-1}$.  The final paragraph assembles these four proposition-level lower bounds.

\subsubsection{Regime~I: Nonprivate Lower Bound}
\label{app:proof:prop:nonprivate-lb}

This is the classical nonprivate smooth-testing lower bound.  We include the proof for completeness and to keep the constants and notation aligned with the private lower-bound arguments below.

\begin{proposition}[Nonprivate $L_1$ lower bound]
\label{prop:nonprivate-lb}
Fix $d\ge 1$, $s>0$, $L>1$, $\gamma\in(0,1)$, and $\beta\in(0,1-\gamma)$. There exists $c=c(d,s,L,\gamma,\beta)>0$ such that, with $r_N:=c\,N^{-2s/(4s+d)}$,
\begin{align*}
\inf_{\substack{\phi:([0,1]^d)^N\to[0,1]\\ \E_0[\phi]\le \gamma}}
\sup_{\substack{f\in \cD_s^d(L)\\ \|f-1\|_1\ge r_N}}
\bigl(1-\E_f^N[\phi]\bigr)
\ge
\beta.
\end{align*}
In particular, $r^*_{\mathrm{gof}}(N,\varepsilon)\gtrsim N^{-2s/(4s+d)}$ for every $\varepsilon>0$.
\end{proposition}

The proof first develops the shared hypercube construction and then applies it to the nonprivate testing problem.

\paragraph{Smooth bump.}
We construct the smooth mean-zero bump used for the shared hypercube family. Here $C_c^\infty((0,1)^d)$ denotes the class of infinitely differentiable functions whose support is compactly contained in $(0,1)^d$. For $q=\lceil s\rceil-1$ and $\eta=s-q$, write
\begin{align*}
\|h\|_{\cH_s^d}
:=
\max\left\{
\max_{|\alpha|\le q}\|D^\alpha h\|_\infty,
\max_{|\alpha|=q}[D^\alpha h]_{C^\eta}
\right\}.
\end{align*}
Thus $\cH_s^d(L)=\{h:\|h\|_{\cH_s^d}\le L\}$.

\begin{lemma}[Smooth localized bumps]
\label{lem:bumps}
There exists a nonzero function $\psi\in C_c^\infty((0,1)^d)$ such that
\begin{align*}
\int_{[0,1]^d}\psi(x)\,dx = 0,
\qquad
\|\psi\|_2=1,
\end{align*}
and, for every fixed $s>0$,
\begin{align*}
\|\psi\|_\infty \le C_\psi,
\qquad
\|\psi\|_{\cH_s^d}\le C_{\psi,s}
\end{align*}
for some finite constants $C_\psi,C_{\psi,s}>0$.
\end{lemma}

\begin{proof}
Choose two nonzero functions $\varphi_1,\varphi_2\in C_c^\infty((0,1)^d)$ with disjoint supports and $\int \varphi_2 \neq 0$. Define
\begin{align*}
a_\varphi:=
\frac{\int_{[0,1]^d}\varphi_1(u)\,du}
     {\int_{[0,1]^d}\varphi_2(u)\,du},
\qquad
\widetilde\psi(x):=\varphi_1(x)-a_\varphi\varphi_2(x).
\end{align*}
Then $\widetilde\psi\in C_c^\infty((0,1)^d)$, $\int \widetilde\psi=0$, and $\widetilde\psi\not\equiv 0$. Set
\begin{align*}
\psi=\frac{\widetilde\psi}{\|\widetilde\psi\|_2}.
\end{align*}
The $L_\infty$ and H\"older bounds are immediate because every $C_c^\infty$ function has finite H\"older norm of any order.
\end{proof}

\paragraph{Properties of the hypercube family.}
The next proposition is the shared technical input for Regimes~I and~III.  It records the validity of the alternatives, their $L_1$ separation from the null, and the mixture KL bound.

Fix a function $\psi\in C_c^\infty((0,1)^d)$ such that $\int\psi=0$ and $\|\psi\|_2=1$; the existence of such a function is proved in \Cref{lem:bumps}.  We view $\psi$ as its zero extension to $\R^d$. For $\kappa\in\mathbb N$, partition $[0,1]^d$ into $K=\kappa^d$ axis-aligned cubes of side length $\kappa^{-1}$, assigning boundary points to one adjacent cube in any fixed deterministic way. Enumerate the cubes as $Q_1,\ldots,Q_K$, and let $a_j$ denote the lower-left corner of $Q_j$.  Define
\begin{align*}
\psi_j(x):=\kappa^{d/2}\psi(\kappa(x-a_j)).
\end{align*}
Since the support of $\psi$ is compactly contained in $(0,1)^d$, the supports of the functions $\psi_j$ are pairwise disjoint and each is contained in $Q_j$. Also $\|\psi_j\|_2=1$. For $\varrho>0$, define
\begin{align*}
\delta := \varrho\,\kappa^{-s-d/2},
\qquad
h_\theta(x):=\delta\sum_{j=1}^K \theta_j\psi_j(x),
\qquad
f_\theta(x):=1+h_\theta(x).
\end{align*}
Let $P_\theta$ denote the law with density $f_\theta$, $P_0$ the null (density $\equiv 1$), and
\begin{align*}
\bar P_\kappa
:=
2^{-K}\sum_{\theta\in\{-1,+1\}^K} P_\theta^N.
\end{align*}

\begin{proposition}[Hypercube bump family]
\label{prop:hypercube-construction}
There exists $c_0=c_0(d,s,\psi)>0$ such that if $0<\varrho\le c_0\min\{1,L-1\}$, then for every $\kappa\ge 1$:
\begin{enumerate}[(i)]
\item each $f_\theta$ is a probability density in $\cH_s^d(L)$ with $f_\theta\ge\tfrac12$.
\item $\|f_\theta-1\|_1 = \varrho\,\|\psi\|_1\,\kappa^{-s}$ for every $\theta\in\{-1,+1\}^K$.
\item there exists $C=C(d,s,\psi)>0$ such that $\KL(\bar P_\kappa\|P_0^N)\le C N^2\varrho^4\kappa^{-(4s+d)}$.
\end{enumerate}
\end{proposition}

\begin{proof}
Let $q=\lceil s\rceil-1$ and $\eta=s-q\in(0,1]$. Since $\operatorname{supp}(\psi)\Subset(0,1)^d$, there is a constant $c_\psi>0$ such that the support of each rescaled bump $\psi_j$ is contained in $Q_j$ and has distance at least $c_\psi\kappa^{-1}$ from $\partial Q_j$. In particular, the supports of $\psi_1,\ldots,\psi_K$ are pairwise disjoint, independently of the convention used to assign boundary points to the cubes.

\paragraph{Step 1: Scaling.}
A change of variables gives, for every $j$,
\begin{align*}
\int \psi_j=0,\qquad
\|\psi_j\|_1=\kappa^{-d/2}\|\psi\|_1,\qquad
\|\psi_j\|_2=1,\qquad
\|\psi_j\|_\infty\le \kappa^{d/2}\|\psi\|_\infty .
\end{align*}
Moreover, for every multi-index $\alpha$ with $|\alpha|\le q$,
\begin{align*}
D^\alpha\psi_j(x)
=
\kappa^{d/2+|\alpha|}
(D^\alpha\psi)(\kappa(x-a_j)).
\end{align*}
Thus
\begin{align*}
\|D^\alpha\psi_j\|_\infty
\le C_\alpha\kappa^{d/2+|\alpha|},
\qquad
[D^\alpha\psi_j]_{C^\eta}
\le C_\alpha\kappa^{d/2+s}
\quad (|\alpha|=q),
\end{align*}
where the second bound follows by applying the preceding identity at two points and using $q+\eta=s$. Since $\psi\in C_c^\infty((0,1)^d)$, the constants depend only on $(d,s,\psi)$.

\paragraph{Step 2: Valid densities.}
At each point at most one bump is active. Hence, with $\delta=\varrho\kappa^{-s-d/2}$,
\begin{align*}
|h_\theta(x)|
\le
\delta\max_j\|\psi_j\|_\infty
\le C\varrho\kappa^{-s}
\le C\varrho.
\end{align*}
Also $\int f_\theta=1$, because $\int\psi_j=0$ for all $j$. It remains to control the H\"older norm of $h_\theta$. The derivative bounds from Step~1 and disjointness give, for every $|\alpha|\le q$,
\begin{align*}
\|D^\alpha h_\theta\|_\infty
\le
\delta\max_j\|D^\alpha\psi_j\|_\infty
\le
C_\alpha\varrho\,\kappa^{|\alpha|-s}
\le
C_\alpha\varrho .
\end{align*}
For the top-order seminorm, fix $|\alpha|=q$ and $x\ne y$. If $x$ and $y$ lie in the same cube $Q_j$, then only the $j$-th bump can contribute on that cube, so
\begin{align*}
\frac{|D^\alpha h_\theta(x)-D^\alpha h_\theta(y)|}{\|x-y\|_2^\eta}
\le
\delta [D^\alpha\psi_j]_{C^\eta}
\le C_\alpha\varrho .
\end{align*}
If $x$ and $y$ lie in different cubes and both values vanish, there is nothing to prove. Otherwise, assume by symmetry that $D^\alpha h_\theta(x)\ne0$. Then $x\in\operatorname{supp}(\psi_j)$ for some $j$, while $y\notin Q_j$; hence $\|x-y\|_2\ge c_\psi\kappa^{-1}$. Using the preceding derivative bound with $|\alpha|=q$,
\begin{align*}
\frac{|D^\alpha h_\theta(x)-D^\alpha h_\theta(y)|}{\|x-y\|_2^\eta}
\le
\frac{2\|D^\alpha h_\theta\|_\infty}{(c_\psi\kappa^{-1})^\eta}
\le
C_\alpha\varrho\,\kappa^{q-s+\eta}
=
C_\alpha\varrho .
\end{align*}
Consequently $[D^\alpha h_\theta]_{C^\eta}\le C_\alpha\varrho$ for $|\alpha|=q$, and therefore $\|h_\theta\|_{\cH_s^d}\le C\varrho$ after enlarging $C=C(d,s,\psi)$.

Choose $c_0>0$, depending only on $(d,s,\psi)$, so small that
\begin{align*}
C c_0\le \frac12
\qquad\text{and}\qquad
c_0^2\le\frac12 .
\end{align*}
If $0<\varrho\le c_0\min\{1,L-1\}$, then $f_\theta\ge1-C\varrho\ge1/2$. The same bound gives $1+C\varrho\le L$, and the positive-order H\"older components of $f_\theta=1+h_\theta$ are bounded by $C\varrho\le L$. Hence $f_\theta\in\cH_s^d(L)$.

\paragraph{Step 3: Separation.}
Using disjoint supports and the $L_1$ scaling from Step~1,
\begin{align*}
\|f_\theta-1\|_1
=
\delta\sum_{j=1}^K\|\psi_j\|_1
=
\varrho\,\|\psi\|_1\,\kappa^{-s}.
\end{align*}

\paragraph{Step 4: Mixture second moment.}
Let
\begin{align*}
L_\theta(x_{1:N})
:=
\frac{dP_\theta^N}{dP_0^N}(x_{1:N})
=
\prod_{i=1}^N(1+h_\theta(x_i)).
\end{align*}
Since $d\bar P_\kappa/dP_0^N=2^{-K}\sum_\theta L_\theta$,
\begin{align*}
1+\chiTwo(\bar P_\kappa\|P_0^N)
=
2^{-2K}\sum_{\theta,\theta'}
\E_0[L_\theta L_{\theta'}].
\end{align*}
Under $P_0$ the observations are i.i.d.; moreover $\int h_\theta=0$ and, by disjointness and $\|\psi_j\|_2=1$,
\begin{align*}
\int h_\theta h_{\theta'}
=
\delta^2\sum_{j=1}^K\theta_j\theta'_j .
\end{align*}
Therefore
\begin{align*}
\E_0[L_\theta L_{\theta'}]
=
\left(1+\delta^2\sum_{j=1}^K\theta_j\theta'_j\right)^N .
\end{align*}
If $\Xi_1,\ldots,\Xi_K$ are i.i.d. Rademacher variables, then the uniform average over $(\theta,\theta')$ has the same law as replacing $(\theta_j\theta'_j)_{j=1}^K$ by $(\Xi_j)_{j=1}^K$. Hence
\begin{align*}
1+\chiTwo(\bar P_\kappa\|P_0^N)
=
\E\left[\left(1+\delta^2\sum_{j=1}^K\Xi_j\right)^N\right].
\end{align*}
Since $\delta^2K=\varrho^2\kappa^{-2s}\le\varrho^2\le1/2$, the base in the last display is positive. Using $1+t\le e^t$ and $\cosh u\le e^{u^2/2}$,
\begin{align*}
1+\chiTwo(\bar P_\kappa\|P_0^N)
\le
\E\exp\left(N\delta^2\sum_{j=1}^K\Xi_j\right)
=
\{\cosh(N\delta^2)\}^K
\le
\exp\left(\frac12 KN^2\delta^4\right).
\end{align*}
Finally $K\delta^4=\varrho^4\kappa^{-(4s+d)}$, and therefore
\begin{align*}
\KL(\bar P_\kappa\|P_0^N)
\le
\log\{1+\chiTwo(\bar P_\kappa\|P_0^N)\}
\le
C N^2\varrho^4\kappa^{-(4s+d)} .
\end{align*}
\end{proof}

\begin{proof}[Proof of \Cref{prop:nonprivate-lb}]

Fix $\varrho>0$ small enough so that \Cref{prop:hypercube-construction} applies. Let $\kappa\in\mathbb N$ be chosen later, and consider the family $\{f_\theta:\theta\in\{-1,+1\}^K\}$ with $K=\kappa^d$. By \Cref{prop:hypercube-construction}(i), (ii),
\begin{align*}
f_\theta\in \cH_s^d(L),
\qquad
\|f_\theta-1\|_1=\varrho\,\|\psi\|_1\,\kappa^{-s}
\qquad\text{for all }\theta.
\end{align*}
Since $f_\theta$ is a probability density and $f_\theta\in\cH_s^d(L)$, we have $f_\theta\in\cD_s^d(L)$. Hence the whole hypercube family lies inside the alternative provided
\begin{align*}
r\le \varrho\,\|\psi\|_1\,\kappa^{-s}.
\end{align*}
Let $\phi:([0,1]^d)^N\to[0,1]$ be any measurable test with $\E_0[\phi]\le \gamma$. By the variational characterization of total variation,
\begin{align*}
\E_{\bar P_\kappa}[\phi]-\E_0[\phi]
\le
\TV(\bar P_\kappa,P_0^N).
\end{align*}
Therefore, by Pinsker's inequality and \Cref{prop:hypercube-construction}(iii),
\begin{align*}
\E_{\bar P_\kappa}[\phi]
\le
\gamma+\sqrt{\frac12 \KL(\bar P_\kappa\|P_0^N)}
\le
\gamma + C N\varrho^2 \kappa^{-(2s+d/2)}.
\end{align*}
Since
\begin{align*}
\E_{\bar P_\kappa}[\phi]
=
2^{-K}\sum_{\theta\in\{-1,+1\}^K}\E_{P_\theta^N}[\phi],
\end{align*}
there exists at least one $\theta^\star$ such that
\begin{align*}
\E_{P_{\theta^\star}^N}[\phi]
\le
\gamma + C N\varrho^2 \kappa^{-(2s+d/2)}.
\end{align*}
Hence
\begin{align*}
1-\E_{P_{\theta^\star}^N}[\phi]
\ge
1-\gamma - C N\varrho^2 \kappa^{-(2s+d/2)}.
\end{align*}
Choose
\begin{align*}
\kappa
=
\left\lceil
A N^{2/(4s+d)}
\right\rceil
\end{align*}
with $A>0$ large enough so that
\begin{align*}
C N\varrho^2 \kappa^{-(2s+d/2)}
\le
1-\gamma-\beta.
\end{align*}
This is possible because
\begin{align*}
N\kappa^{-(2s+d/2)}
\lesssim
A^{-(2s+d/2)}.
\end{align*}
Therefore
\begin{align*}
1-\E_{P_{\theta^\star}^N}[\phi]\ge \beta.
\end{align*}
Finally, for this choice of $\kappa$, since $N\ge1$, the ceiling operation gives $\kappa\le(A+1)N^{2/(4s+d)}$, and hence $\kappa^{-s}\ge(A+1)^{-s}N^{-2s/(4s+d)}$. Therefore
\begin{align*}
\|f_{\theta^\star}-1\|_1
=
\varrho\,\|\psi\|_1\,\kappa^{-s}
\ge
c\,N^{-2s/(4s+d)}
\end{align*}
for some constant $c=c(d,s,L,\gamma,\beta)>0$. Since $\phi$ was arbitrary,
\begin{align*}
\inf_{\substack{\phi:([0,1]^d)^N\to[0,1]\\ \E_0[\phi]\le \gamma}}
\sup_{\substack{f\in \cD_s^d(L)\\ \|f-1\|_1\ge c N^{-2s/(4s+d)}}}
\bigl(1-\E_f^N[\phi]\bigr)
\ge
\beta.
\end{align*}
Because every $\mathcal M\in\Phi_{\gamma,\varepsilon}^{\mathrm{gof}}$ satisfies $\E_0^N[p_{\mathcal M}]\le\gamma$, we have
\begin{align*}
\bigl\{p_{\mathcal M} : \mathcal M\in\Phi_{\gamma,\varepsilon}^{\mathrm{gof}}\bigr\}
\subseteq
\bigl\{
\phi:([0,1]^d)^N\to[0,1]:\ \E_0[\phi]\le \gamma
\bigr\},
\end{align*}
the same lower bound transfers immediately to the central-DP problem:
\begin{align*}
r^*_{\mathrm{gof}}(N,\varepsilon)\gtrsim N^{-2s/(4s+d)}.
\end{align*}
\end{proof}

\subsubsection{Regime~II: Intermediate-DP Lower Bound}
\label{app:sec:regime2-technical}

The Regime~II argument uses a different, cellwise construction: the bins have equal mass under every alternative, so the lower bound comes from the privacy cost of distinguishing biased signs within occupied cells.  The proof first establishes the resulting finite-sample necessary condition and then inverts it to obtain the separation rate used in the theorem.

\begin{proposition}[Regime~II lower bound]
\label{prop:regime-ii-lb}
Let $d\ge 1$, $s>0$, $L>1$, and $\gamma,\beta>0$ with $\gamma+\beta<1$. There exists $c>0$, depending only on $(d,s,L,\gamma,\beta)$, such that, with
\begin{align*}
r_N:=c\,(N\sqrt{\varepsilon})^{-2s/(2s+d)},
\end{align*}
for all $N\ge1$ and $\varepsilon\in(0,1]$ with $N\varepsilon\ge1$,
\begin{align*}
\inf_{\mathcal M\in\Phi_{\gamma,\varepsilon}^{\mathrm{gof}}}
  \sup_{\substack{f\in\cD_s^d(L)\\\|f-1\|_1\ge r_N}}
\bigl(1-\E_f^N[p_{\mathcal M}]\bigr)
\ge
\beta.
\end{align*}
Consequently, whenever $N\varepsilon\ge1$,
\begin{align*}
r^*_{\mathrm{gof}}(N,\varepsilon)
\gtrsim
(N\sqrt{\varepsilon})^{-2s/(2s+d)}.
\end{align*}
\end{proposition}

The proof uses three ingredients: localized alternatives with equal cell probabilities, a latent-sign representation within each cell, and a Bernoulli-mixture coupling for those signs.  The next three lemmas record these ingredients.

\begin{lemma}[Localized bump alternatives for Regime II]
\label{lem:lb-local-bumps}
Let $d\ge 1$, $s>0$, $L>1$, and $\phi\in C_c^\infty((0,1)^d)$ with $\int_{[0,1]^d}\phi\,dx=0$ and $\|\phi\|_\infty\le 1$. Write $b_1:=\|\phi\|_1>0$. There exists $c_\star=c_\star(d,s,\phi)>0$ such that the following holds. For any $\kappa\in\mathbb N$, partition $[0,1]^d=\bigsqcup_{j=1}^K Q_j$ into $K=\kappa^d$ cubes, and let $a_j$ denote the lower-left corner of $Q_j$.  Define
\begin{align*}
\phi_j(x):=\phi(\kappa(x-a_j)).
\end{align*}
For $\theta\in\{-1,+1\}^K$, define
\begin{align*}
f_\theta(x):=
1+\lambda\sum_{j=1}^K \theta_j\phi_j(x).
\end{align*}
If $0<\lambda\le c_\star\min\{1,L-1\}$ and $\lambda\kappa^s\le c_\star L$, then the supports of $\phi_1,\dots,\phi_K$ are pairwise disjoint, and for every $\theta\in\{-1,+1\}^K$,
\begin{align*}
f_\theta\ge \frac12,\qquad
\int f_\theta=1,\qquad
f_\theta\in \cD_s^d(L),
\end{align*}
\begin{align*}
\|f_\theta-1\|_1=\lambda b_1,
\qquad
\int_{Q_j} f_\theta\,dx=\kappa^{-d}\quad\text{for every }j\in[K].
\end{align*}
\end{lemma}

\begin{proof}
We use the same localization estimates as in the proof of \Cref{prop:hypercube-construction}, now with the $L_2$-normalizing factor $\kappa^{d/2}$ removed.  We record the resulting bounds, since the scaling is slightly different in the present construction. Write
\begin{align*}
q_0:=\lceil s\rceil-1,
\qquad
\eta:=s-q_0\in(0,1].
\end{align*}

\paragraph{Step 1: Localization and scaling identities.}
View $\phi$ as its zero extension to $\R^d$.  Since $\phi$ is compactly supported in $(0,1)^d$, its support is separated from the boundary of $(0,1)^d$ by a positive distance; denote this distance by
\begin{align*}
c_\phi:=\operatorname{dist}\bigl(\operatorname{supp}(\phi),\partial(0,1)^d\bigr)>0.
\end{align*}
Then, after scaling and translation, $\operatorname{supp}(\phi_j)\subset Q_j$, and this support is separated from the boundary $\partial Q_j$ of the cube $Q_j$ by at least $c_\phi\kappa^{-1}$. Hence the supports of $\phi_1,\ldots,\phi_K$ are pairwise disjoint, independently of the convention used to assign boundary points to partition cells. By the change of variables $u=\kappa(x-a_j)$,
\begin{align*}
\int_{Q_j}\phi_j(x)\,dx
=
\kappa^{-d}\int_{[0,1]^d}\phi(u)\,du
=
0,
\qquad
\|\phi_j\|_1
=
\kappa^{-d}\|\phi\|_1
=
\kappa^{-d}b_1 .
\end{align*}

\paragraph{Step 2: Mass and separation identities.}
Consequently,
\begin{align*}
\int f_\theta
=
1+\lambda\sum_{j=1}^K\theta_j\int\phi_j
=
1,
\end{align*}
and, for every $j\in[K]$,
\begin{align*}
\int_{Q_j} f_\theta(x)\,dx
=
\kappa^{-d}
+
\lambda\theta_j\int_{Q_j}\phi_j(x)\,dx
=
\kappa^{-d}.
\end{align*}
The same disjointness gives
\begin{align*}
\|f_\theta-1\|_1
=
\lambda
\left\|\sum_{j=1}^K\theta_j\phi_j\right\|_1
=
\lambda\sum_{j=1}^K\|\phi_j\|_1
=
\lambda K\kappa^{-d}b_1
=
\lambda b_1.
\end{align*}

\paragraph{Step 3: Positivity and boundedness.}
Since $\|\phi\|_\infty\le1$ and at most one bump is active at any point,
\begin{align*}
\left|\sum_{j=1}^K\theta_j\phi_j(x)\right|\le1.
\end{align*}
Thus
\begin{align*}
f_\theta(x)\ge1-\lambda,
\qquad
\|f_\theta\|_\infty\le1+\lambda.
\end{align*}
In particular, if $c_\star\le1/2$ and $\lambda\le c_\star\min\{1,L-1\}$, then
\begin{align*}
f_\theta\ge \frac12,
\qquad
\|f_\theta\|_\infty\le L.
\end{align*}

\paragraph{Step 4: H\"older regularity.}
It remains to verify the H\"older regularity.  Let
\begin{align*}
S_\theta(x):=\sum_{j=1}^K\theta_j\phi_j(x).
\end{align*}
For every multi-index $\alpha$ with $|\alpha|\le q_0$, the chain rule gives
\begin{align*}
D^\alpha\phi_j(x)
=
\kappa^{|\alpha|}
(D^\alpha\phi)(\kappa(x-a_j)).
\end{align*}
Since the supports are disjoint, at most one term contributes at each point. Hence, for $\kappa\ge1$,
\begin{align*}
\|D^\alpha S_\theta\|_\infty
\le
\kappa^{|\alpha|}\|D^\alpha\phi\|_\infty
\le
C_\alpha\kappa^s .
\end{align*}
We next claim that, for every $|\alpha|=q_0$,
\begin{align*}
[D^\alpha S_\theta]_{C^\eta}
\le
C_\alpha\kappa^s .
\end{align*}
Indeed, if $x,y$ belong to the same cell $Q_j$, then only the $j$-th bump can contribute on that cell, and the chain rule gives
\begin{align*}
\frac{
|D^\alpha\phi_j(x)-D^\alpha\phi_j(y)|
}{
\|x-y\|_2^\eta
}
\le
\kappa^{q_0+\eta}[D^\alpha\phi]_{C^\eta}
=
\kappa^s[D^\alpha\phi]_{C^\eta}.
\end{align*}
If $x$ and $y$ belong to different cells, then either both values $D^\alpha S_\theta(x)$ and $D^\alpha S_\theta(y)$ vanish, or one of the two points lies in the support of some $\phi_j$.  In the latter case, say this point is $x\in\operatorname{supp}(\phi_j)$.  Since $y\notin Q_j$ and $\operatorname{supp}(\phi_j)$ is separated from $\partial Q_j$ by $c_\phi\kappa^{-1}$, we have
\begin{align*}
\|x-y\|_2\ge c_\phi\kappa^{-1}.
\end{align*}
Using the preceding derivative bound with $|\alpha|=q_0$,
\begin{align*}
\frac{
|D^\alpha S_\theta(x)-D^\alpha S_\theta(y)|
}{
\|x-y\|_2^\eta
}
\le
\frac{2\|D^\alpha S_\theta\|_\infty}
     {(c_\phi\kappa^{-1})^\eta}
\le
C_\alpha\kappa^{q_0+\eta}
=
C_\alpha\kappa^s .
\end{align*}
This proves the claimed H\"older seminorm bound. Combining the derivative and seminorm estimates, there exists $C=C(d,s,\phi)<\infty$ such that all positive-order H\"older components of $f_\theta-1=\lambda S_\theta$ are bounded by $C\lambda\kappa^s$.

\paragraph{Step 5: Choice of constants.}
Choose $c_\star=c_\star(d,s,\phi)>0$ sufficiently small so that $C c_\star\le1$, and also so that the positivity and $L_\infty$ bounds above hold.  If
\begin{align*}
\lambda\kappa^s\le c_\star L
\qquad\text{and}\qquad
\lambda\le c_\star\min\{1,L-1\},
\end{align*}
then the H\"older bounds are within radius $L$, the density is nonnegative and bounded, and $f_\theta\in\cD_s^d(L)$.  Together with the integral, cell-mass, and $L_1$-separation identities proved above, this gives all the stated properties.
\end{proof}

\begin{lemma}[Latent-sign representation within cells]
\label{lem:lb-cellwise-signs}
Under the setup of \Cref{lem:lb-local-bumps}, define $v_j(x):=K\,\mathbf{1}_{Q_j}(x)$, $r_j^\pm(x):=v_j(x)(1\pm \phi_j(x))$. For $\theta\in\{-1,+1\}^K$, let $P_\theta$ be the law of $N$ i.i.d.\ samples from $f_\theta$, and let $P_0$ be the law of $N$ i.i.d.\ samples from the uniform density $1$.  Define the mixture law
\begin{align*}
\bar P:=2^{-K}\sum_{\theta\in\{-1,+1\}^K}P_\theta .
\end{align*}
Then:
\begin{enumerate}[(i)]
\item Each $r_j^\pm$ is a density on $Q_j$, and $v_j=\tfrac12(r_j^++r_j^-)$.
\item Under each $P_\theta$, under $P_0$, and under $\bar P$, the cell indices $J_1,\dots,J_N$ are i.i.d.\ uniform on $\{1,\dots,K\}$. Conditional on $(J_1,\dots,J_N)$, observations from different cells are independent.
\item The within-cell laws admit the following latent-sign representation. Under $P_0$, each observation in cell $j$ has an independent latent sign $B_{j,m}$ with $\Pbb(B_{j,m}=\pm1)=\tfrac12$.  Under $\bar P$, cell $j$ first draws a uniform latent direction $\Theta_j\in\{-1,+1\}$, and then its conditional signs satisfy
        \begin{align*}
        \Pbb(C_{j,m}=+1\mid\Theta_j)=\frac{1+\lambda\Theta_j}{2}.
        \end{align*}
\end{enumerate}
\end{lemma}

\begin{proof}

Since $\|\phi_j\|_\infty\le 1$ and $\int_{Q_j}\phi_j=0$,
\begin{align*}
r_j^\pm(x)\ge 0,
\qquad
\int_{Q_j} r_j^\pm(x)\,dx
=
K\int_{Q_j}(1\pm \phi_j(x))\,dx
=
K\kappa^{-d}
=
1.
\end{align*}
Thus each $r_j^\pm$ is a probability density on $Q_j$, and $\frac12(r_j^+ + r_j^-)=v_j$. By \Cref{lem:lb-local-bumps},
\begin{align*}
\int_{Q_j} f_\theta(x)\,dx=\kappa^{-d}=\frac1K
\qquad\text{for all }j,\theta.
\end{align*}
Hence under every $P_\theta$, and therefore also under $\bar P$, the cell indices $J_1,\dots,J_N$ are i.i.d.\ uniform on $\{1,\dots,K\}$. The same is true under $P_0$. Fix $\theta$ and $j$. For $x\in Q_j$,
\begin{align*}
f_\theta(x)
=
1+\lambda\theta_j\phi_j(x),
\end{align*}
and therefore the conditional density of a sample given $J_i=j$ is
\begin{align*}
K\,\mathbf{1}_{Q_j}(x)\bigl(1+\lambda\theta_j\phi_j(x)\bigr)
=
\frac{1+\lambda\theta_j}{2}\,r_j^+(x)
+
\frac{1-\lambda\theta_j}{2}\,r_j^-(x).
\end{align*}
Under $P_0$, the conditional density given $J_i=j$ is simply $v_j=\frac12(r_j^+ + r_j^-)$. Finally, under the mixture law $\bar P$, we may equivalently first draw $\Theta=(\Theta_1,\dots,\Theta_K)$ uniformly from $\{-1,+1\}^K$, and then sample according to $P_\Theta$. Conditional on $\Theta$ and on the full sequence $(J_1,\dots,J_N)$, the coordinates remain independent, and within a fixed cell $j$ only $\Theta_j$ matters. Grouping coordinates by cell gives the stated latent-sign representation.
\end{proof}

\begin{lemma}[Bernoulli-mixture coupling]\label{lem:bern-mixture}
Fix $t\in\mathbb N$ and $\mu\in[0,1]$. Let $P_t$ be the law of $B_1,\dots,B_t\in\{-1,+1\}$ i.i.d.\ Rademacher with $\mathbb P(B_i=+1)=\tfrac12$. Let $Q_t$ be the law obtained by drawing $\Theta\in\{-1,+1\}$ uniformly and then, conditionally on $\Theta$, drawing $C_1,\dots,C_t$ i.i.d.\ with $\mathbb P(C_i=+1\mid \Theta)=(1+\mu\Theta)/2$. Then there exists a coupling of $(B_1^t,C_1^t)$ with
\begin{align*}
\E\big[d_H(B_1^t,C_1^t)\big]
\le \mu^2 (t^2-t).
\end{align*}
\end{lemma}

\begin{proof}
Apply \citet[Lemma~12]{AcharyaSunZhang2018} after identifying $\{0,1\}$ with $\{-1,+1\}$ via $x\mapsto 2x-1$. In the notation of \citet{AcharyaSunZhang2018}, the Bernoulli bias shift is $\delta=\mu/2$, hence
\begin{align*}
4(t^2-t)\delta^2
=
4(t^2-t)\left(\frac{\mu}{2}\right)^2
=
\mu^2(t^2-t).
\end{align*}
\end{proof}

\begin{proof}[Proof of \Cref{prop:regime-ii-lb}]
We first record a finite-sample necessary condition.  There exist positive constants $c_{\rm sep}, C_{\rm nec}>0$, depending only on $(d,s,\gamma,\beta)$, with the following property: for every $0<r\le c_{\rm sep}\min\{1,L-1\}$, any $\varepsilon$-DP test $\Psi_N$ based on $N$ i.i.d.\ observations and satisfying
\begin{align*}
\Pbb_0(\Psi_N=1)\le \gamma,
\qquad
\sup_{\substack{f\in \cD_s^d(L)\\ \|f-1\|_1\ge r}}
\Pbb_f(\Psi_N=0)\le \beta
\end{align*}
must obey
\begin{align*}
N \ge
C_{\rm nec}\,
L^{d/(2s)}r^{-1-d/(2s)}\varepsilon^{-1/2}.
\end{align*}
Throughout the proof of this necessary condition, constants implicit in $\lesssim$ and $\asymp$ depend only on $(d,s,\gamma,\beta)$; the dependence on $L$ is displayed explicitly.

Let $\psi$ be the bump from \Cref{lem:bumps}, and set
\begin{align*}
\phi:=\frac{\psi}{\max\{1,\|\psi\|_\infty\}},
\qquad
b_1:=\|\phi\|_1 .
\end{align*}
Let $c_\star$ be the constant in \Cref{lem:lb-local-bumps}.  Fix $c_0\in(0,c_\star^{1/s}]$.  Decreasing $c_{\rm sep}$, if necessary, we may assume that for every $L>1$ and every $0<r\le c_{\rm sep}\min\{1,L-1\}$, with $\lambda:=r/b_1$,
\begin{equation}
\label{eq:regime-ii-lambda-small}
\lambda\le c_\star\min\{1,L-1\}
\qquad\text{and}\qquad
c_0\left(\frac{L}{\lambda}\right)^{1/s}\ge 2 .
\end{equation}
Indeed, it suffices to choose
\begin{align*}
c_{\rm sep}
\le
b_1\min\left\{
c_\star,\left(\frac{c_0}{2}\right)^s
\right\}.
\end{align*}

Fix such an $r$ and a test $\Psi_N$ satisfying the preceding error bounds. Choose an integer $\kappa$ such that
\begin{equation}
\label{eq:regime-ii-kappa-choice}
\frac{c_0}{2}\left(\frac{L}{\lambda}\right)^{1/s}
\le
\kappa
\le
c_0\left(\frac{L}{\lambda}\right)^{1/s},
\end{equation}
which is possible by the second inequality in \eqref{eq:regime-ii-lambda-small}.  Set $K:=\kappa^d$.

\paragraph{Step 1: Construction of alternatives.}
Partition $[0,1]^d$ into $K$ cubes $Q_j=a_j+\kappa^{-1}[0,1]^d$.  For $\theta\in\{-1,+1\}^K$, define
\begin{align*}
f_\theta(x)
:=
1+\lambda\sum_{j=1}^K
\theta_j\phi\bigl(\kappa(x-a_j)\bigr).
\end{align*}
By \eqref{eq:regime-ii-kappa-choice},
\begin{align*}
\lambda\kappa^s
\le
c_0^sL
\le
c_\star L .
\end{align*}
Together with \eqref{eq:regime-ii-lambda-small}, the hypotheses of \Cref{lem:lb-local-bumps} are satisfied.  Hence, for every $\theta\in\{-1,+1\}^K$,
\begin{align*}
f_\theta\in \cD_s^d(L),
\qquad
\|f_\theta-1\|_1=\lambda b_1=r,
\qquad
\int_{Q_j} f_\theta\,dx=\frac1K
\quad\text{for all }j\in[K].
\end{align*}
Thus every $f_\theta$ is an admissible alternative at separation $r$.

\paragraph{Step 2: Reduction to a simple-versus-mixture problem.}
Let $P_0$ denote the law of $N$ i.i.d.\ observations from the uniform density on $[0,1]^d$, and let $P_\theta$ denote the corresponding product law under $f_\theta$.  Define the mixture
\begin{align*}
\bar P:=2^{-K}\sum_{\theta\in\{-1,+1\}^K}P_\theta .
\end{align*}
Since the test has type~II error at most $\beta$ under every $P_\theta$, averaging over $\theta$ gives
\begin{align*}
\bar P(\Psi_N=0)\le \beta,
\qquad
P_0(\Psi_N=1)\le \gamma .
\end{align*}
It therefore suffices to compare $P_0$ and $\bar P$.

\paragraph{Step 3: A cellwise Hamming coupling.}
By \Cref{lem:lb-cellwise-signs}, under both $P_0$ and $\bar P$ the cell indices $J_1,\ldots,J_N$ are i.i.d.\ uniform on $[K]$.  Couple the two experiments so that they share the same realization of these indices, and let
\begin{align*}
T_j:=\#\{i:J_i=j\},\qquad j\in[K].
\end{align*}
Conditional on $(J_1,\ldots,J_N)$, the observations in distinct cells are independent.  Moreover, within a cell $j$, \Cref{lem:lb-cellwise-signs} identifies the conditional laws with the Bernoulli-mixture experiment of \Cref{lem:bern-mixture} with bias parameter $\mu=\lambda$.

Apply \Cref{lem:bern-mixture} independently in each cell.  Thus, for the $m$-th observation falling in cell $j$, we obtain coupled signs $(B_{j,m},C_{j,m})$ such that
\begin{align*}
\E\!\left[
d_H\bigl((B_{j,m})_{m=1}^{T_j},(C_{j,m})_{m=1}^{T_j}\bigr)
\,\middle|\,T_j
\right]
\le
\lambda^2 T_j(T_j-1).
\end{align*}
We lift this sign coupling to the observations as follows.  If $B_{j,m}=C_{j,m}$, the two coupled observations are taken to be the same draw from the density $r_j^{B_{j,m}}$.  If $B_{j,m}\ne C_{j,m}$, the two observations are drawn from their respective conditional densities $r_j^{B_{j,m}}$ and $r_j^{C_{j,m}}$.  This preserves the required marginals, and the observation-level mismatch is bounded by the sign mismatch:
\begin{align*}
\mathbf{1}\{X_i\neq Y_i\}
\le
\mathbf{1}\{B_{j,m}\neq C_{j,m}\}
\end{align*}
where $(j,m)$ is the cell--position index of observation $i$.  Consequently, after averaging over the multinomial cell counts,
\begin{align*}
\begin{aligned}
\E[d_H(X_1^N,Y_1^N)]
&\le
\lambda^2\sum_{j=1}^K \E[T_j(T_j-1)]  \\
&=
\lambda^2 K\cdot \frac{N(N-1)}{K^2}
\le
\lambda^2\frac{N^2}{K}.
\end{aligned}
\end{align*}
Thus $P_0$ and $\bar P$ admit a coupling whose expected Hamming distance is at most $\lambda^2N^2/K$.

\paragraph{Step 4: Differential privacy forces a large sample size.}
By \Cref{lem:dp-coupling}, the preceding coupling and the two error bounds $P_0(\Psi_N=1)\le\gamma$, $\bar P(\Psi_N=0)\le\beta$ imply
\begin{align*}
\lambda^2\frac{N^2}{K}
\ge
\frac1\varepsilon
\log\frac1{\gamma+\beta}.
\end{align*}
Since $\gamma+\beta<1$, the logarithm is a positive constant depending only on $(\gamma,\beta)$.  Moreover, by \eqref{eq:regime-ii-kappa-choice},
\begin{align*}
K=\kappa^d
\asymp
\left(\frac{L}{\lambda}\right)^{d/s}.
\end{align*}
Therefore
\begin{align*}
N
\gtrsim
\frac{\sqrt{K}}{\lambda\sqrt{\varepsilon}}
\asymp
L^{d/(2s)}
\lambda^{-1-\frac{d}{2s}}
\varepsilon^{-1/2}.
\end{align*}
Since $\lambda=r/b_1$, and $b_1$ depends only on $(d,s)$, this gives
\begin{align*}
N
\ge
C_{\rm nec}\,
L^{d/(2s)}
r^{-1-d/(2s)}
\varepsilon^{-1/2},
\end{align*}
after adjusting $C_{\rm nec}$.  This proves the necessary condition.

\medskip
We now invert this condition.  Put
\begin{align*}
a:=1+\frac{d}{2s}=\frac{2s+d}{2s}.
\end{align*}
For all sufficiently small separations $r$, the necessary condition says that any $\varepsilon$-DP test with type~I error at most $\gamma$ and type~II error at most $\beta$ over the GOF alternatives must satisfy
\begin{align*}
N
\ge
C_{\rm nec}\,
L^{d/(2s)}r^{-a}\varepsilon^{-1/2}.
\end{align*}
Let
\begin{align*}
r_N
:=
c\,(N\sqrt\varepsilon)^{-1/a}
=c\,(N\sqrt\varepsilon)^{-2s/(2s+d)}.
\end{align*}
Because $N\varepsilon\ge1$ and $\varepsilon\le1$, we have $N\sqrt\varepsilon\ge1$.  Hence, by choosing $c=c(d,s,L,\gamma,\beta)>0$ sufficiently small, we can ensure both
\begin{align*}
r_N\le c_{\rm sep}\min\{1,L-1\}
\end{align*}
and
\begin{align*}
C_{\rm nec}L^{d/(2s)}c^{-a}>1 .
\end{align*}
If an $\varepsilon$-DP test had type~I error at most $\gamma$ and type~II error at most $\beta$ uniformly over all alternatives with $\|f-1\|_1\ge r_N$, the necessary condition applied with $r=r_N$ would yield
\begin{align*}
N
\ge
C_{\rm nec}L^{d/(2s)} r_N^{-a}\varepsilon^{-1/2}
=
C_{\rm nec}L^{d/(2s)}c^{-a}N,
\end{align*}
contradicting $C_{\rm nec}L^{d/(2s)}c^{-a}>1$.  Thus no such test exists. Equivalently,
\begin{align*}
\inf_{\mathcal M\in\Phi_{\gamma,\varepsilon}^{\mathrm{gof}}}
  \sup_{\substack{f\in\cD_s^d(L)\\ \|f-1\|_1\ge r_N}}
\bigl(1-\E_f^N[p_{\mathcal M}]\bigr)
\ge
\beta,
\end{align*}
with
\begin{align*}
r_N=c\,(N\sqrt\varepsilon)^{-2s/(2s+d)}.
\end{align*}
This is the claimed Regime~II lower bound.
\end{proof}

\subsubsection{Regime~III: Transport-DP Hypercube Lower Bound}
\label{app:proof:prop:cdp-lb}

The proof uses the validity, $L_1$ separation, and mixture KL bounds from \Cref{prop:hypercube-construction}.  The additional privacy cost enters through the transport bound.

\begin{proposition}[Transport-DP $L_1$ lower bound]
\label{prop:cdp-lb}
Fix $d\ge 1$, $s>0$, $L>1$, $\gamma\in(0,1)$, $\beta\in(0,1-\gamma)$, and $\varepsilon\in(0,1]$.  There exists $c=c(d,s,L,\gamma,\beta)>0$ such that, with $r_N:=c\,(N^{3/2}\varepsilon)^{-2s/(4s+d)}$, for all $N\ge 1$ with $N\varepsilon\ge 1$,
\begin{align*}
\inf_{\mathcal M\in\Phi_{\gamma,\varepsilon}^{\mathrm{gof}}}
  \sup_{\substack{f\in\cD_s^d(L)\\\|f-1\|_1\ge r_N}}
\bigl(1-\E_f^N[p_{\mathcal M}]\bigr)
\ge
\beta.
\end{align*}
In particular, $r^*_{\mathrm{gof}}(N,\varepsilon)\gtrsim(N^{3/2}\varepsilon)^{-2s/(4s+d)}$.
\end{proposition}

\begin{proof}

Fix $\varrho>0$ small enough so that \Cref{prop:hypercube-construction} applies. Let $\kappa\in\mathbb{N}$ be chosen later, and consider the family $\{f_\theta:\theta\in\{-1,+1\}^K\}$ with $K=\kappa^d$.

By \Cref{prop:hypercube-construction}(i), every $f_\theta$ is a probability density in $\cH_s^d(L)$, and hence $f_\theta\in\cD_s^d(L)$. By \Cref{prop:hypercube-construction}(ii),
\begin{align*}
\|f_\theta-1\|_1 = \varrho\,\|\psi\|_1\,\kappa^{-s}
\qquad\text{for all }\theta.
\end{align*}
Thus the whole hypercube family lies inside the alternative provided
\begin{align*}
r \le \varrho\,\|\psi\|_1\,\kappa^{-s}.
\end{align*}
Let $\mathcal M\in \Phi_{\gamma,\varepsilon}^{\mathrm{gof}}$ be arbitrary, with rejection probability function $p_{\mathcal M}$. Since $\E_0^N[p_{\mathcal M}]\le \gamma$, applying \Cref{lem:transport} with $Q=\bar P_\kappa$ yields
\begin{align*}
\E_{\bar P_\kappa}[p_{\mathcal M}]
\le
\gamma
+
\varepsilon \sqrt{\frac{N}{2}\KL(\bar P_\kappa\|P_0^N)}.
\end{align*}
Combining this with \Cref{prop:hypercube-construction}(iii),
\begin{align*}
\E_{\bar P_\kappa}[p_{\mathcal M}]
\le
\gamma
+
C \varepsilon \sqrt{N\cdot N^2 \varrho^4 \kappa^{-(4s+d)}}
=
\gamma + C \varepsilon N^{3/2}\varrho^2 \kappa^{-(2s+d/2)}.
\end{align*}
Since
\begin{align*}
\E_{\bar P_\kappa}[p_{\mathcal M}]
=
2^{-K}\sum_{\theta\in\{-1,+1\}^K} \E_{P_\theta^N}[p_{\mathcal M}],
\end{align*}
there exists at least one $\theta^\star$ such that
\begin{align*}
\E_{P_{\theta^\star}^N}[p_{\mathcal M}]
\le
\gamma + C \varepsilon N^{3/2}\varrho^2 \kappa^{-(2s+d/2)}.
\end{align*}
Hence the type~II error of $\mathcal M$ under $f_{\theta^\star}$ satisfies
\begin{align*}
1-\E_{P_{\theta^\star}^N}[p_{\mathcal M}]
\ge
1-\gamma - C \varepsilon N^{3/2}\varrho^2 \kappa^{-(2s+d/2)}.
\end{align*}
Choose
\begin{align*}
\kappa
=
\left\lceil
A (N^{3/2}\varepsilon)^{2/(4s+d)}
\right\rceil
\end{align*}
with $A>0$ large enough so that
\begin{align*}
C \varepsilon N^{3/2}\varrho^2 \kappa^{-(2s+d/2)}
\le
1-\gamma-\beta.
\end{align*}
Indeed, since $\kappa \ge A(N^{3/2}\varepsilon)^{2/(4s+d)}$ and $2s+\tfrac{d}{2}=\tfrac{4s+d}{2}$, we have
\begin{align*}
\kappa^{-(2s+d/2)}
\le
\Bigl[A(N^{3/2}\varepsilon)^{2/(4s+d)}\Bigr]^{-(4s+d)/2}
=
A^{-(2s+d/2)}\,(N^{3/2}\varepsilon)^{-1},
\end{align*}
and therefore the error term satisfies
\begin{align*}
C \varepsilon N^{3/2}\varrho^2 \kappa^{-(2s+d/2)}
\le
C\varrho^2\,A^{-(2s+d/2)}.
\end{align*}
Choosing $A$ large enough so that $C\varrho^2 A^{-(2s+d/2)}\le 1-\gamma-\beta$ (which depends only on $(d,s,\gamma,\beta)$) gives the required bound. Therefore,
\begin{align*}
1-\E_{P_{\theta^\star}^N}[p_{\mathcal M}]\ge \beta.
\end{align*}
Finally, for this choice of $\kappa$, because $N\varepsilon\ge1$, we have $N^{3/2}\varepsilon\ge1$. Thus the ceiling operation gives $\kappa\le(A+1)(N^{3/2}\varepsilon)^{2/(4s+d)}$, and hence $\kappa^{-s}\ge(A+1)^{-s}(N^{3/2}\varepsilon)^{-2s/(4s+d)}$. Therefore
\begin{align*}
\|f_{\theta^\star}-1\|_1
=
\varrho\,\|\psi\|_1\,\kappa^{-s}
\ge
c (N^{3/2}\varepsilon)^{-2s/(4s+d)}
\end{align*}
for some constant $c=c(d,s,L,\gamma,\beta)>0$ by adjusting $A$ and $\varrho$. Since $\mathcal M$ was arbitrary,
\begin{align*}
\inf_{\mathcal M\in \Phi_{\gamma,\varepsilon}^{\mathrm{gof}}}
\sup_{\substack{f\in \cD_s^d(L)\\ \|f-1\|_1\ge c (N^{3/2}\varepsilon)^{-2s/(4s+d)}}}
\bigl(1-\E_f^N[p_{\mathcal M}]\bigr)
\ge
\beta.
\end{align*}
This proves the proposition.
\end{proof}

\subsubsection{Regime~IV: Fixed-Perturbation Linear-DP Lower Bound}
\label{app:proof:prop:linear-lb}

This final regime uses a single fixed smooth perturbation and the DP-coupling lemma directly.

\begin{proposition}[Fixed-perturbation $L_1$ lower bound]
\label{prop:linear-lb}
Fix $d\ge 1$, $s>0$, $L>1$, $\gamma\in(0,1)$, and $\beta\in(0,1-\gamma)$. There exists $c=c(d,s,L,\gamma,\beta)>0$ such that for all $N\ge 1$ and $\varepsilon\in(0,1]$ with $N\varepsilon\ge 1$,
\begin{align*}
r \le c\,(N\varepsilon)^{-1}
\;\Longrightarrow\;
\inf_{\mathcal M \in \Phi_{\gamma,\varepsilon}^{\mathrm{gof}}}
  \sup_{\substack{f \in \cD_s^d(L)\\ \|f-1\|_1 \ge r}}
\bigl(1-\E_f^N[p_{\mathcal M}]\bigr)
\ge \beta.
\end{align*}
In particular, $r^*_{\mathrm{gof}}(N,\varepsilon)\gtrsim (N\varepsilon)^{-1}$.
\end{proposition}

\begin{proof}

\noindent\textbf{Step 1: Fixed perturbation.} Fix a nonzero function $\psi\in C_c^\infty((0,1)^d)$ with $\int\psi=0$. For $r>0$, define
\begin{align*}
\lambda_r := \frac{r}{\|\psi\|_1},
\qquad
f_r := 1 + \lambda_r\psi.
\end{align*}
Since $\int\psi=0$, we have $\int f_r=1$. For $r$ sufficiently small, depending only on $(d,s,L,\psi)$, the following properties hold.

We first choose the perturbation amplitude small enough so that the path $r\mapsto f_r$ stays inside the null-centered H\"older class. Since $\psi$ is fixed, $\|\psi\|_1>0$ and
\begin{align*}
\|f_r-1\|_\infty
=
\lambda_r\|\psi\|_\infty
=
\frac{r\|\psi\|_\infty}{\|\psi\|_1}.
\end{align*}
Thus, if $r\le \|\psi\|_1/(2\|\psi\|_\infty)$, then $f_r(x)\ge 1/2$ for all $x\in[0,1]^d$. Moreover, $\int f_r=1$, because $\int\psi=0$.

Next, since $f_r=1+\lambda_r\psi$,
\begin{align*}
\|f_r\|_{\cH_s^d}
\le
\|1\|_{\cH_s^d}+\lambda_r\|\psi\|_{\cH_s^d}
=
1+\frac{r\|\psi\|_{\cH_s^d}}{\|\psi\|_1},
\end{align*}
where $\|1\|_{\cH_s^d}=1$ under our normalization of the H\"older norm. Hence $f_r\in\cH_s^d(L)$ whenever
\begin{align*}
r
\le
\frac{(L-1)\|\psi\|_1}{\|\psi\|_{\cH_s^d}}.
\end{align*}
Finally, by construction,
\begin{align*}
\|f_r-1\|_1
=
\lambda_r\|\psi\|_1
=
r.
\end{align*}
Consequently, after fixing $\psi$ once and for all, there exists $c_1=c_1(d,s,L)>0$ such that, for every $0<r\le c_1$, the function $f_r$ is a probability density in $\cD_s^d(L)$, satisfies $f_r\ge 1/2$, and has $\|f_r-1\|_1=r$. Since $N\varepsilon\ge1$, the condition $r\le c(N\varepsilon)^{-1}$ with $c\le c_1$ guarantees $r\le c_1$, and hence this fixed perturbation is a valid alternative at separation exactly $r$.

\medskip\noindent\textbf{Step 2: Maximal coupling and expected Hamming distance.}
Let $P_r$ be the law of one observation from $f_r$, and let $P_0$ denote the uniform law on $[0,1]^d$. For each coordinate $i\in\{1,\dots,N\}$, take a maximal coupling $(X_i,Y_i)$ of $(P_r,P_0)$, so that
\begin{align*}
\Pbb(X_i\neq Y_i)=\TV(P_r,P_0).
\end{align*}
Choosing these coordinate couplings independently gives a coupling of $P_r^{\otimes N}$ and $P_0^{\otimes N}$. Under this coupling,
\begin{align*}
\E\bigl[d_H(X_{1:N},Y_{1:N})\bigr]
=
\sum_{i=1}^N \Pbb(X_i\neq Y_i)
=
N\,\TV(P_r,P_0).
\end{align*}
Since
\begin{align*}
\TV(P_r,P_0)
=
\frac12\|f_r-1\|_1
=
\frac r2,
\end{align*}
we have constructed a coupling with
\begin{align*}
\E\bigl[d_H(X_{1:N},Y_{1:N})\bigr]
=
\frac{Nr}{2}.
\end{align*}

\medskip\noindent\textbf{Step 3: Apply the DP coupling lemma.}
Applying \Cref{lem:dp-coupling} with $P=P_0^{\otimes N}$, $Q=P_r^{\otimes N}$, and $D=Nr/2$, any $\varepsilon$-DP test $\Psi$ with type~I error at most $\gamma$ and type~II error at most $\beta$ must satisfy
\begin{align*}
\frac{Nr}{2}
\ge
\frac{1}{\varepsilon}\log\frac{1}{\gamma+\beta}.
\end{align*}
Equivalently,
\begin{align*}
r
\ge
\frac{2}{N\varepsilon}\log\frac{1}{\gamma+\beta}.
\end{align*}
Now set
\begin{align*}
c := \min\!\left\{c_1,\;\log\frac{1}{\gamma+\beta}\right\}.
\end{align*}
If $r\le c(N\varepsilon)^{-1}$, then $r\le c_1$, so Step~1 ensures that $f_r\in\cD_s^d(L)$ and $\|f_r-1\|_1=r$. On the other hand,
\begin{align*}
\frac{Nr}{2}
\le
\frac{1}{2\varepsilon}\log\frac{1}{\gamma+\beta},
\end{align*}
which violates the necessary condition above for any $\varepsilon$-DP test with type~I error at most $\gamma$ and type~II error at most $\beta$. Hence every admissible private test has type~II error at least $\beta$ against this valid alternative, completing the proof.
\end{proof}

\subsubsection{Completion of the Proof of \CrefInTitle{Theorem}{thm:gof-combined-lower}}

We now assemble the four proposition-level lower bounds.

\begin{proof}
Set
\begin{align*}
R_N(\varepsilon)
:=
N^{-2s/(4s+d)}
\vee
(N\sqrt\varepsilon)^{-2s/(2s+d)}
\vee
(N^{3/2}\varepsilon)^{-2s/(4s+d)}
\vee
(N\varepsilon)^{-1}.
\end{align*}
First suppose that $0<\varepsilon\le1$ and $N\varepsilon\ge1$.  The four lower bounds in \Cref{prop:nonprivate-lb,prop:regime-ii-lb,prop:cdp-lb,prop:linear-lb} give
\begin{align*}
r^*_{\mathrm{gof}}(N,\varepsilon)
\gtrsim R_N(\varepsilon).
\end{align*}
It remains to consider $\varepsilon>1$.  In this case the privacy-dependent terms in $R_N(\varepsilon)$ are all dominated by the nonprivate term.  Indeed, since $N\ge1$,
\begin{align*}
(N\sqrt\varepsilon)^{-2s/(2s+d)}
&\le N^{-2s/(2s+d)}
\le N^{-2s/(4s+d)},\\
(N^{3/2}\varepsilon)^{-2s/(4s+d)}
&\le N^{-3s/(4s+d)}
\le N^{-2s/(4s+d)},\\
(N\varepsilon)^{-1}
&\le N^{-1}
\le N^{-2s/(4s+d)}.
\end{align*}
Thus \Cref{prop:nonprivate-lb} alone yields the asserted maximum lower bound. The corresponding two-sample lower bound follows from \eqref{eq:twosamp-gof-reduction}.
\end{proof}

\subsection{Proof of \CrefInTitle{Theorem}{thm:adaptive-upper-loglog-rb-mc}}
\label{app:adaptive-upper}

\paragraph{Roadmap of the proof.}
The proof has three layers. First, conditional on the unlabeled pooled data, we control the Monte Carlo centering quantiles uniformly over the dyadic grid $\mathcal K_N$. Second, building on the fixed-resolution power argument from the non-adaptive test, we show that the same comparison remains valid after empirical centering and maximization over the adaptive grid: if one resolution $\kappa$ has enough binned $L_1$ signal, then the observed calibrated score dominates the final permutation scores and the Laplace noise with high probability. Third, for each unknown smoothness $s\in[s_-,s_+]$, we choose a dyadic resolution within a constant factor of the oracle resolution, converting the fixed-resolution criterion into the adaptive rate.

\subsubsection{Adaptive Test Proof Setup}
\label{app:sec:adaptive-technical}

We analyze the adaptive Rao--Blackwellized Monte Carlo permutation test in \Cref{alg:adaptive-rb-mc-main}. Let $n=\lfloor N/2\rfloor$, and use the first $2n$ observations from each sample. In the proof we write the two sample labels as $X$ and $Y$, and set
\begin{align*}
\ell_i^{(0)}=X\quad(1\le i\le 2n),
\qquad
\ell_i^{(0)}=Y\quad(2n<i\le4n).
\end{align*}
The procedure itself does not randomize the order of the pooled records. For the null-validity argument, one may introduce an auxiliary uniform permutation $\pi$ of the $4n$ record positions and apply it simultaneously to the records and to all two-label assignments used by the algorithm. Under $H_0:f=g$, conditional on the resulting unlabeled pooled sample, $\pi\ell^{(0)}$ is uniform on $\mathfrak L_{2n,2n}$ and is exchangeable with independent draws from $\Unif(\mathfrak L_{2n,2n})$. Since all statistics below are invariant under simultaneous reindexing of records and labels, this symmetrized view leaves the implemented procedure unchanged. Let
\begin{align*}
\mathfrak G_n
:=
\Bigl\{
L\in\{1,2,3,4\}^{4n}:\#\{i:L_i=a\}=n,\ a=1,2,3,4
\Bigr\}
\end{align*}
be the balanced four-slice of all equal four-labelings. For $\ell\in\mathfrak L_{2n,2n}$, let $\mathfrak S(\ell)\subset\mathfrak G_n$ be the set of four-label refinements of $\ell$, namely the $L\in\mathfrak G_n$ satisfying
\begin{align*}
\ell_i=X \Longleftrightarrow L_i\in\{1,2\},
\qquad
\ell_i=Y \Longleftrightarrow L_i\in\{3,4\}.
\end{align*}
For a resolution $\kappa$, recall that $\widetilde X_j^{(\kappa)}=\mathrm{cell}_\kappa(X_j)$ and $\widetilde Y_j^{(\kappa)}=\mathrm{cell}_\kappa(Y_j)$, and let
\begin{align*}
W_\kappa
:=
\bigl(
\widetilde X_1^{(\kappa)},\ldots,\widetilde X_{2n}^{(\kappa)},
\widetilde Y_1^{(\kappa)},\ldots,\widetilde Y_{2n}^{(\kappa)}
\bigr)
\in[\kappa^d]^{4n}.
\end{align*}
The Rao--Blackwellized split-histogram statistic from \Cref{sec:rb-implementation} is $\overline Z_n(W_\kappa,\ell)$. Equivalently, if $T_\kappa(W_\kappa,L):=Z(W_\kappa,L)$, then
\begin{align*}
\overline Z_n(W_\kappa,\ell)
=
\frac{1}{|\mathfrak S(\ell)|}
\sum_{L\in\mathfrak S(\ell)}
T_\kappa(W_\kappa,L).
\end{align*}

The considered dyadic grid is
\begin{align*}
\mathcal K_N=\{2^0,2^1,\ldots,2^{J_N}\},
\qquad
J_N=\lceil 3\log_2 N\rceil,
\qquad
G_N=|\mathcal K_N|.
\end{align*}
The tuning parameters are chosen as in \eqref{eq:adaptive-main-tuning}. In the proof, we use only the consequences
\begin{align*}
u_N=\frac{\beta}{256(B_{\mathrm{rank}}+1)G_N},
\qquad
a_N=\log\frac{8}{u_N},
\qquad
u_N B_{\mathrm{cen},N}\ge 16a_N,
\end{align*}
and hence $a_N\asymp\log\log N.$ The numerical constants in these tuning parameters are chosen for convenience and are used to simplify the union-bound calculations in the proof. Let
\begin{align*}
\mathcal E
:=
\left\{
\ell_{\kappa,r}^{\mathrm{cen}}:
\kappa\in\mathcal K_N,\
r=1,\ldots,B_{\mathrm{cen},N}
\right\}
\end{align*}
denote the centering-stage reassignments. Conditionally on the observed samples, the assignments in $\mathcal E$, the final rank reassignments, and the Laplace noises are mutually independent. Each reassignment is uniform on $\mathfrak L_{2n,2n}$, and the noises satisfy
\begin{align*}
\zeta_0,\zeta_1,\ldots,\zeta_{B_{\mathrm{rank}}}
\stackrel{\iid}{\sim}
\Lap(1).
\end{align*}
We retain the notation of \Cref{alg:adaptive-rb-mc-main}: $\widehat q_\kappa(X,Y;\mathcal E)$ is the empirical $(1-u_N)$-quantile of the centering batch at resolution $\kappa$, and
\begin{align*}
\widehat A_{\mathcal E}(X,Y;\ell)
:=
\max_{\kappa\in\mathcal K_N}
\left\{
\overline Z_n(W_\kappa,\ell)-\widehat q_\kappa(X,Y;\mathcal E)
\right\}
\end{align*}
is the centered multiscale score. With $\ell_0^{\mathrm{rank}}:=\ell^{(0)}$, we abbreviate $\widehat A_b:=\widehat A_{\mathcal E}(X,Y;\ell_b^{\mathrm{rank}})$ and $\widehat M_b:=\widehat A_b+(16/\varepsilon)\zeta_b$ for the privatized rank scores, $b=0,1,\ldots,B_{\mathrm{rank}}$, and recall that the test rejects according to the final rule of \Cref{alg:adaptive-rb-mc-main}.

\subsubsection{Auxiliary Concentration and Quantile Lemmas}
\label{app:sec:adaptive-auxiliary-lemmas}

Throughout this subsection we fix $n\ge1$, a deterministic pooled vector of binned labels $W=(W_1,\ldots,W_{4n})$, and a resolution $\kappa$; we write $k=\kappa^d$ and let $c_i\in[k]$ denote the bin containing $W_i$. For $j\in[k]$ set $M_j:=\sum_{i=1}^{4n}\mathbf 1\{c_i=j\}$, and define the active index set and its cardinality
\begin{align*}
I_\kappa(W):=\{i\in[4n]:M_{c_i}\ge2\},
\qquad
\mathcal A_\kappa(W):=|I_\kappa(W)|=\sum_{j=1}^k M_j\mathbf 1\{M_j\ge2\}.
\end{align*}
Recall the balanced four-slice $\mathfrak G_n$ from the setup above, and let $\ell=\ell(L)\in\mathfrak L_{2n,2n}$ denote the associated two-class coarsening, defined by
\begin{align*}
\ell_i=X \iff L_i\in\{1,2\},
\qquad
\ell_i=Y \iff L_i\in\{3,4\}.
\end{align*}
For $L\in\mathfrak G_n$, $j\in[k]$, and $a\in\{1,2,3,4\}$, set
\begin{align*}
N_{ja}(L):=\sum_{i=1}^{4n}\mathbf 1\{c_i=j,\ L_i=a\}.
\end{align*}
Using the same four-count kernel
\begin{align*}
g(x_1,x_2,x_3,x_4):=|x_1-x_3|+|x_2-x_4|-|x_1-x_2|-|x_3-x_4|,
\end{align*}
the four-way split statistic is
\begin{align*}
T_\kappa(W,L):=\sum_{j=1}^k g\bigl(N_{j1}(L),N_{j2}(L),N_{j3}(L),N_{j4}(L)\bigr),
\end{align*}
whose Rao--Blackwellization is $\overline Z_n(W,\ell)=\E[T_\kappa(W,L)\mid W,\ell]$.

\begin{lemma}[Conditional concentration under balanced four-way relabeling]
\label{lem:balanced-relabeling-exp-conc}
Let $L\sim\Unif(\mathfrak G_n)$. Then $T_\kappa(W,L)$ is conditionally mean-zero and sub-Gaussian with variance proxy $\mathcal A_\kappa(W)$: namely $\E_L[T_\kappa(W,L)\mid W]=0$, and there is a universal constant $C_0<\infty$ such that
\begin{align*}
\log
\E_L\!\left[
\exp\{\lambda T_\kappa(W,L)\}
\,\middle|\,W
\right]
\le
C_0\lambda^2\mathcal A_\kappa(W),
\qquad\lambda\in\mathbb R.
\end{align*}
Consequently, for a universal constant $c_0>0$ and every $x\ge0$,
\begin{align*}
\Pbb_L\!\left(
|T_\kappa(W,L)|\ge x
\,\middle|\,W
\right)
\le
2\exp\!\left\{
-\frac{c_0x^2}{\mathcal A_\kappa(W)}
\right\},
\end{align*}
with the right-hand side read as $0$ when $\mathcal A_\kappa(W)=0$ and $x>0$.
\end{lemma}
\begin{proof}
Only active cells with $M_j\ge2$ can contribute: if $M_j=0$ the cell contribution is zero, and if $M_j=1$ the count vector is one of the four standard basis vectors, for which $g=0$. Hence $T_\kappa(W,L)$ depends only on $(L_i)_{i\in I_\kappa(W)}$.

Changing one active label affects only one cell. If the label in that cell changes from $a$ to $b$, the corresponding count vector $x=(x_1,x_2,x_3,x_4)$ is replaced by $x'=x-e_a+e_b$. For any pair of coordinates $r,s$,
\begin{align*}
\left|
|x_r-x_s|-|x'_r-x'_s|
\right|
\le
|x_r-x'_r|+|x_s-x'_s|.
\end{align*}
The four absolute-value terms in $g$ involve the pairs $(1,3),(2,4),(1,2),(3,4)$, so each coordinate appears exactly twice. Since $x$ and $x'$ differ by one unit in only two coordinates, this gives $|g(x)-g(x')|\le4$. Thus $T_\kappa(W,\cdot)$ is $4$-Lipschitz on the active labels. Applying \Cref{lem:balanced-multislice-conc} with $q=4$, $n_1=\cdots=n_4=n$, $I=I_\kappa(W)$, and $F=T_\kappa(W,\cdot)$ gives the stated exponential-moment bound after centering.

It remains to compute the conditional mean. For a fixed bin $j$, the vector
\begin{align*}
(N_{j1}(L),N_{j2}(L),N_{j3}(L),N_{j4}(L))
\end{align*}
is exchangeable under permutations of the four labels, because the uniform measure on $\mathfrak G_n$ is invariant under relabeling $1,2,3,4$. Hence all pairwise absolute differences have the same expectation:
\begin{align*}
\E_L|N_{j1}-N_{j3}|
=
\E_L|N_{j1}-N_{j2}|
=
\E_L|N_{j2}-N_{j4}|
=
\E_L|N_{j3}-N_{j4}|.
\end{align*}
Therefore $\E_L g(N_{j1},N_{j2},N_{j3},N_{j4})=0$, and summing over $j=1,\ldots,k$ yields $\E_L[T_\kappa(W,L)\mid W]=0$. The moment bound is therefore centered at zero. The tail bound follows by Chernoff's inequality.
\end{proof}

\begin{lemma}[Rao--Blackwellized split-statistic concentration]
\label{lem:rb-balanced-relabeling-exp-conc}
If $\ell\sim\Unif(\mathfrak L_{2n,2n})$, then
\begin{align*}
\E_\ell[\overline Z_n(W,\ell)\mid W]=0,
\end{align*}
and for every $t\ge0$,
\begin{align*}
\Pbb_\ell\!\left(
\overline Z_n(W,\ell)
>
C\sqrt{\mathcal A_\kappa(W)t}
\,\middle|\,W
\right)
\le
e^{-t},
\end{align*}
where $C>0$ is a universal constant.
\end{lemma}

\begin{proof}
Let $L\sim\Unif(\mathfrak G_n)$ be a balanced four-labeling, and let $\ell=\ell(L)$ be its coarsening:
\begin{align*}
\ell_i=X \Longleftrightarrow L_i\in\{1,2\},
\qquad
\ell_i=Y \Longleftrightarrow L_i\in\{3,4\}.
\end{align*}
Then $\ell\sim\Unif(\mathfrak L_{2n,2n})$, and conditionally on $\ell$, the four-labeling $L$ is uniform on $\mathfrak S(\ell)$. Hence
\begin{align*}
\overline Z_n(W,\ell)
=
\E_L[T_\kappa(W,L)\mid W,\ell].
\end{align*}
By \Cref{lem:balanced-relabeling-exp-conc}, conditionally on $W$, $T_\kappa(W,L)$ is centered sub-Gaussian with variance proxy $C\mathcal A_\kappa(W)$. Applying \Cref{lem:rb-preserves-subgaussian} with $\mathcal G=\sigma(W)$, $\mathcal H=\sigma(W,\ell)$, and $X=T_\kappa(W,L)$ gives the same conditional mean-zero and sub-Gaussian moment bounds for $\overline Z_n(W,\ell)$. Chernoff's inequality gives the stated tail bound.
\end{proof}

\subsubsection{Completion of the Proof of \CrefInTitle{Theorem}{thm:adaptive-upper-loglog-rb-mc}}
\label{app:proof:thm:adaptive-upper-loglog-rb-mc}
\begin{proof}
Let $k=\kappa^d$, and let $C,c>0$ denote constants depending only on $(d,s_-,s_+,L,M,\gamma,\beta)$, whose values may change from line to line. Since $n=\lfloor N/2\rfloor$ and $n\asymp N$ for $N\ge2$, we state all fixed-resolution bounds in terms of $N$.

\paragraph{Privacy and type~I error.}
Condition on all centering reassignments. Under the fixed-design neighboring relation, neighboring raw datasets induce, for each $\kappa$, binned vectors $W_\kappa,W_\kappa'$ differing in at most one coordinate. Fix a four-labeling $L$. If the changed observation carries label $a\in\{1,2,3,4\}$ and moves from cell $j$ to cell $j'$, there is nothing to prove when $j=j'$. Otherwise, the count vector in cell $j$ changes from $x$ to $x-e_a$, while the count vector in cell $j'$ changes from $y$ to $y+e_a$. Changing a single coordinate of a cell count vector by one affects only the two absolute-value terms in $g$ in which that coordinate appears, and therefore
\begin{align*}
|g(x\pm e_a)-g(x)|\le2 .
\end{align*}
Thus the two affected cells together change $T_\kappa(W_\kappa,L)$ by at most $4$. Averaging over $L\in\mathfrak S(\ell)$ gives
\begin{align*}
\left|\overline Z_n(W_\kappa,\ell)-\overline Z_n(W_\kappa',\ell)\right|\le 4
\qquad
\text{for all }\kappa,\ell .
\end{align*}
Since the empirical quantile $\widehat q_\kappa(X,Y;\mathcal E)$ is one-Lipschitz in the sup-norm of the centering scores,
\begin{align*}
|\widehat q_\kappa(X,Y;\mathcal E)-\widehat q_\kappa(X',Y';\mathcal E)|\le4,
\qquad\text{and therefore}\qquad
|\widehat A_{\mathcal E}(X,Y;\ell)-\widehat A_{\mathcal E}(X',Y';\ell)|\le8 .
\end{align*}
Thus the adaptive score $\widehat A_{\mathcal E}$ has sensitivity $8$.

Moreover, under $H_0$, use the auxiliary symmetrized view described in the setup. Conditional on the symmetrized unlabeled pooled sample and the centering batch, the symmetrized observed assignment and the rank assignments are i.i.d.\ $\Unif(\mathfrak L_{2n,2n})$, and the same centering quantiles enter every rank score. By simultaneous-reindexing invariance, the symmetrized and implemented procedures have the same scores. The algorithm adds Laplace noise with scale $16/\varepsilon=2\cdot8/\varepsilon$. Therefore \Cref{lem:private-permutation-cited} applies to the centered multiscale statistic $\widehat A_{\mathcal E}$, yielding $\varepsilon$-differential privacy and finite-sample type~I error at most $\gamma$.

\paragraph{Monte Carlo centering bound.}
For the observed samples, write $\mathcal A_\kappa(X,Y):=\mathcal A_\kappa(W_\kappa)$. By \Cref{lem:rb-balanced-relabeling-exp-conc} and \Cref{lem:mc-empirical-quantile}(i), with $a_N=\log(8/u_N)$,
\begin{align*}
\Pbb\!\left(
\widehat q_\kappa(X,Y;\mathcal E)>
C\sqrt{\mathcal A_\kappa(X,Y)a_N}
\,\middle|\,X,Y
\right)
\le \exp(-u_NB_{\mathrm{cen},N}) .
\end{align*}
Since $u_NB_{\mathrm{cen},N}\ge16a_N$, a union bound over $\kappa\in\mathcal K_N$ yields, with probability at least $1-\beta/32$, simultaneously for all $\kappa$,
\begin{equation}
\label{eq:mc-qhat-active-bound}
\widehat q_\kappa(X,Y;\mathcal E)
\le C\sqrt{\mathcal A_\kappa(X,Y)a_N}.
\end{equation}
For the fixed resolution used below, boundedness of $f,g\in\cF_s^d(L,M)$ implies
\begin{align*}
\E\mathcal A_\kappa(X,Y)
\le
C\left(N\wedge\frac{N^2}{k}\right).
\end{align*}
Thus, by Markov's inequality and \eqref{eq:mc-qhat-active-bound}, with probability at least $1-\beta/16$,
\begin{equation}
\label{eq:mc-qhat-final-bound}
\widehat q_\kappa(X,Y;\mathcal E)
\le
C\sqrt{
\left(N\wedge\frac{N^2}{k}\right)a_N
}.
\end{equation}

\paragraph{Fixed-resolution power.}
Let $p^{(\kappa)},q^{(\kappa)}$ be the binned distributions. By \Cref{lem:binning-approx}, with $C_{\mathrm{bin}}^{\mathrm{ad}}:=C_{\mathrm{bin}}(d,s_-,s_+,L)$, uniformly for $s\in[s_-,s_+]$,
\begin{align*}
\|p^{(\kappa)}-q^{(\kappa)}\|_1
\ge
\tau_\kappa(r):=\frac r2-C_{\mathrm{bin}}^{\mathrm{ad}}\kappa^{-s}.
\end{align*}
Using the fixed-resolution moment bounds in \Cref{lem:rb-perm-moment-bounds} and Rao--Blackwellization,
\begin{align*}
\E[\overline Z_n(W_\kappa,\ell^{(0)})]
\ge c\,\mu_\kappa(\tau_\kappa(r)),
\qquad
\Var(\overline Z_n(W_\kappa,\ell^{(0)}))
\le
C\left(N\wedge\frac{N^2}{k}\right),
\end{align*}
where
\begin{align*}
\mu_\kappa(\tau)
:=
\min\left\{
N\tau,\,
\frac{N^2\tau^2}{k},\,
\frac{N^{3/2}\tau^2}{\sqrt k}
\right\}.
\end{align*}
Let
\begin{align*}
r_{\mathrm{rank}}:=\log\frac{4(B_{\mathrm{rank}}+1)}{\beta}.
\end{align*}
Suppose that
\begin{equation}
\label{eq:adaptive-mc-fixed-resolution-condition}
\tau_\kappa(r)
\ge
C
\left[
\sqrt{\frac{a_N}{N}}
+
a_N^{1/4}\frac{k^{1/4}}{\sqrt N}
+
\frac{k^{1/2}}{N\sqrt{\varepsilon}}
+
\frac{k^{1/4}}{N^{3/4}\sqrt{\varepsilon}}
+
\frac{1}{N\varepsilon}
\right].
\end{equation}
By elementary algebra, \eqref{eq:adaptive-mc-fixed-resolution-condition} implies
\begin{align*}
\mu_\kappa(\tau_\kappa(r))
\ge
C
\left[
\sqrt{
\left(N\wedge\frac{N^2}{k}\right)a_N
}
+ 
\frac{r_{\mathrm{rank}}}{\varepsilon}
\right],
\end{align*}
since $B_{\mathrm{rank}}$, $\gamma$, and $\beta$ are fixed constants.  For example,
\begin{align*}
\tau_\kappa(r)
\ge
C\sqrt{\frac{a_N}{N}}
\quad\Longrightarrow\quad
N\tau_\kappa(r)
\ge
C\sqrt{Na_N},
\end{align*}
and
\begin{align*}
\tau_\kappa(r)
\ge
C\,a_N^{1/4}\frac{k^{1/4}}{\sqrt N}
\quad\Longrightarrow\quad
\frac{N^2\tau_\kappa(r)^2}{k}
\ge
C\sqrt{\frac{N^2a_N}{k}}.
\end{align*}
These two implications explain the fluctuation term $\sqrt{(N\wedge N^2/k)a_N}$. The remaining conditions in \eqref{eq:adaptive-mc-fixed-resolution-condition} similarly ensure that each of the three components in the minimum defining $\mu_\kappa(\tau)$ is at least a sufficiently large constant multiple of the required fluctuation and privacy scales. By the preceding lower bound on $\mu_\kappa(\tau_\kappa(r))$, choosing the constant $C$ in \eqref{eq:adaptive-mc-fixed-resolution-condition} sufficiently large yields
\begin{align*}
\E[\overline Z_n(W_\kappa,\ell^{(0)})]
\ge
C_1
\left[
\sqrt{
\left(N\wedge\frac{N^2}{k}\right)a_N
}
+
\frac{r_{\mathrm{rank}}}{\varepsilon}
\right].
\end{align*}
Since
\begin{align*}
\Var(\overline Z_n(W_\kappa,\ell^{(0)}))
\le
C_2\left(N\wedge\frac{N^2}{k}\right),
\end{align*}
Chebyshev's inequality implies that, with probability at least $1-\beta/16$,
\begin{align*}
\overline Z_n(W_\kappa,\ell^{(0)})
\ge
2C
\sqrt{
\left(N\wedge\frac{N^2}{k}\right)a_N
}
+
64\,\frac{r_{\mathrm{rank}}}{\varepsilon},
\end{align*}
after increasing $C$ in \eqref{eq:adaptive-mc-fixed-resolution-condition} if necessary.\footnote{The numerical constants appearing here and below are not intrinsic to the argument. They are chosen only for convenience and may be replaced by any sufficiently large fixed constants.} Combining this with \eqref{eq:mc-qhat-final-bound}, which holds with probability at least $1-\beta/16$, and applying a union bound yields an event $\Omega_0$ of probability at least $1-\beta/8$ on which
\begin{align*}
\widehat A_0
\ge
\overline Z_n(W_\kappa,\ell^{(0)})-\widehat q_\kappa(X,Y;\mathcal E)
\ge
48\,\frac{r_{\mathrm{rank}}}{\varepsilon}.
\end{align*}

\paragraph{Final rank comparison.}
For each competitor $b\ge1$, $\ell_b^{\mathrm{rank}}$ is independent of the centering batch. Hence \Cref{lem:mc-empirical-quantile}(ii), followed by a union bound over $b$ and $\kappa$, gives
\begin{align*}
\Pbb\!\left(
\max_{1\le b\le B_{\mathrm{rank}}}\widehat A_b>0
\right)
\le
B_{\mathrm{rank}}G_N
\left(
u_N+\frac{1}{B_{\mathrm{cen},N}+1}\right),
\end{align*}
which is at most $\beta/32$ by the definitions of $u_N$ and $B_{\mathrm{cen},N}$. Let $\Omega_1$ be the event $\max_{1\le b\le B_{\mathrm{rank}}}\widehat A_b\le0$. Also, for $\zeta_b\stackrel{\iid}{\sim}\Lap(1)$, the union bound and definition of $r_{\mathrm{rank}}$ give
\begin{align*}
\Pbb\!\left(
\max_{0\le b\le B_{\mathrm{rank}}}|\zeta_b|>r_{\mathrm{rank}}
\right)
\le
(B_{\mathrm{rank}}+1)e^{-r_{\mathrm{rank}}}
=
\frac{\beta}{4}.
\end{align*}
Let $\Omega_2$ be the complementary event. On $\Omega_0\cap\Omega_1\cap\Omega_2$,
\begin{align*}
\widehat M_0
\ge
\frac{48r_{\mathrm{rank}}}{\varepsilon}
-
\frac{16r_{\mathrm{rank}}}{\varepsilon}
>
\frac{16r_{\mathrm{rank}}}{\varepsilon}
\ge
\max_{1\le b\le B_{\mathrm{rank}}}\widehat M_b .
\end{align*}
Therefore the rank of the observed noisy score is one, and the test rejects because $(B_{\mathrm{rank}}+1)^{-1}\le\gamma$. Since
\begin{align*}
\Pbb(\Omega_0^c\cup\Omega_1^c\cup\Omega_2^c)
\le
\frac{\beta}{8}+\frac{\beta}{32}+\frac{\beta}{4}
<\beta,
\end{align*}
the type~II error is at most $\beta$ whenever \eqref{eq:adaptive-mc-fixed-resolution-condition} holds for some $\kappa\in\mathcal K_N$.

Because $\tau_\kappa(r)=r/2-C_{\mathrm{bin}}^{\mathrm{ad}}\kappa^{-s}$, this fixed-resolution condition is implied by
\begin{align*}
r
\ge
C
\left[
\kappa^{-s}
+
\sqrt{\frac{a_N}{N}}
+
a_N^{1/4}\frac{\kappa^{d/4}}{\sqrt N}
+
\frac{\kappa^{d/2}}{N\sqrt{\varepsilon}}
+
\frac{\kappa^{d/4}}{N^{3/4}\sqrt{\varepsilon}}
+
\frac{1}{N\varepsilon}
\right].
\end{align*}
Taking the infimum over $\kappa\in\mathcal K_N$ proves \eqref{eq:adaptive-rb-mc-inf-bound}.

\paragraph{Explicit rate.}
Since $G_N\lesssim\log N$, the tuning gives $a_N=\log(8/u_N)\asymp\log\log N$. Define
\begin{align*}
\begin{aligned}
\rho_1&:=a_N^{s/(4s+d)}N^{-2s/(4s+d)},&
\rho_2&:=(N\sqrt{\varepsilon})^{-2s/(2s+d)},\\
\rho_3&:=(N^{3/2}\varepsilon)^{-2s/(4s+d)},&
\rho_4&:=(N\varepsilon)^{-1},
\end{aligned}
\end{align*}
and set
\begin{align*}
\rho_N(s,\varepsilon)
:=\rho_1\vee\rho_2\vee\rho_3\vee\rho_4 .
\end{align*}
In what follows, write $\rho_N=\rho_N(s,\varepsilon)$. It remains to upper-bound the preceding infimum by $C\rho_N(s,\varepsilon)$. If $\rho_N(s,\varepsilon)$ is bounded below by a fixed constant, take $\kappa=1$. Otherwise define the oracle resolution
\begin{align*}
\kappa_\star:=\rho_N(s,\varepsilon)^{-1/s}.
\end{align*}
Since $\rho_N(s,\varepsilon)\ge a_N^{s/(4s+d)}N^{-2s/(4s+d)}$, we have $\kappa_\star\lesssim (N^2/a_N)^{1/(4s+d)} \le (N^2/a_N)^{1/(4s_-+d)}=o(N^3)$, uniformly over $s\in[s_-,s_+]$. Thus the dyadic grid, whose endpoint is of order $N^3$, contains $\kappa\in\mathcal K_N$ such that $\kappa_\star\le\kappa\le2\kappa_\star$. For this choice of $\kappa$, the bias term satisfies $\kappa^{-s}\le\rho_N$. The first resolution-dependent stochastic term is bounded by
\begin{align*}
a_N^{1/4}\frac{\kappa^{d/4}}{\sqrt N}
\le
C a_N^{1/4}\rho_N^{-d/(4s)}N^{-1/2}
\le
C\rho_N,
\end{align*}
where the last step follows from $\rho_N^{1+d/(4s)}\ge\rho_1^{1+d/(4s)} =a_N^{1/4}N^{-1/2}$. The remaining two resolution-dependent terms are handled in the same way: since $\rho_N^{1+d/(2s)}\ge\rho_2^{1+d/(2s)}=(N\sqrt{\varepsilon})^{-1}$ and $\rho_N^{1+d/(4s)}\ge\rho_3^{1+d/(4s)} =(N^{3/4}\sqrt{\varepsilon})^{-1}$, we have
\begin{align*}
\frac{\kappa^{d/2}}{N\sqrt{\varepsilon}}
\le
C\rho_N,
\qquad
\frac{\kappa^{d/4}}{N^{3/4}\sqrt{\varepsilon}}
\le
C\rho_N .
\end{align*}
The privacy floor is immediate from $(N\varepsilon)^{-1}=\rho_4\le\rho_N$. Finally,
\begin{align*}
\sqrt{\frac{a_N}{N}}
\le
C a_N^{s/(4s+d)}N^{-2s/(4s+d)}
\end{align*}
uniformly for $s\in[s_-,s_+]$ and all sufficiently large $N$, because
\begin{align*}
\frac{\sqrt{a_N/N}}{\rho_1}
=
a_N^{1/2-s/(4s+d)}
N^{-d/(2(4s+d))}
\to0
\end{align*}
uniformly over this range of $s$. Hence
\begin{align*}
\inf_{\kappa\in\mathcal K_N}
\left\{
\kappa^{-s}
+
\sqrt{\frac{a_N}{N}}
+
a_N^{1/4}\frac{\kappa^{d/4}}{\sqrt N}
+
\frac{\kappa^{d/2}}{N\sqrt{\varepsilon}}
+
\frac{\kappa^{d/4}}{N^{3/4}\sqrt{\varepsilon}}
+
\frac{1}{N\varepsilon}
\right\}
\le
C\rho_N(s,\varepsilon),
\end{align*}
which proves \eqref{eq:adaptive-rb-mc-explicit-loglog-rate}.
\end{proof}

\subsection{Proof of \CrefInTitle{Theorem}{thm:adaptive-lower-loglog}}
\label{app:proof:thm:adaptive-lower-loglog}

We prove the adaptive lower bound through one-sample testing against the uniform density.  The argument is nonprivate, so the conclusion automatically holds for central-DP tests.  Let $P_0$ denote the uniform distribution on $[0,1]^d$, and write $\E_0$ for expectation under $P_0^{\otimes N}$.  For $N$ sufficiently large, set
\begin{align*}
B_N:=\log\log N,
\qquad
\bar\rho_N^{\mathrm{ad}}(s)
:=B_N^{s/(4s+d)}N^{-2s/(4s+d)}.
\end{align*}

\begin{proposition}[Adaptive one-sample lower bound]
\label{prop:adaptive-gof-lower-loglog}
Under the assumptions of \Cref{thm:adaptive-lower-loglog}, there exist constants $c>0$ and a sufficiently large $N_0$, depending only on $(d,s_-,s_+,L,M,\gamma,\beta)$, such that the following holds for every $N\ge N_0$.  Define the class of level-$\gamma$ one-sample tests by
\begin{align*}
\Phi_{\gamma,N}^{\mathrm{gof,ad}}
:=
\left\{
\psi:([0,1]^d)^N\to[0,1]:
\E_0\psi\le\gamma
\right\}.
\end{align*}
Then
\begin{align*}
\inf_{\psi\in\Phi_{\gamma,N}^{\mathrm{gof,ad}}}
\;
\sup_{s\in[s_-,s_+]}
\;
\sup_{\substack{
f\in\cF_s^d(L,M)\\
\|f-1\|_1\ge c\,\bar\rho_N^{\mathrm{ad}}(s)
}}
\bigl(1-\E_f\psi\bigr)
>
\beta .
\end{align*}
Equivalently, no level-$\gamma$ one-sample goodness-of-fit test can have worst-case type~II error at most $\beta$, uniformly over $s\in[s_-,s_+]$, at separation $c\,\bar\rho_N^{\mathrm{ad}}(s)$.
\end{proposition}

The proof uses a multiscale Ingster mixture \citep{Ingster2000AdaptiveChiSquare}.  Choose an integer $R\ge0$, fixed throughout.  Let $\omega\in C_c^\infty(0,1)$ be nonzero and set
\begin{align*}
\varphi_1(t):=\frac{d^{R+1}}{dt^{R+1}}\omega(t),
\qquad
\varphi(x):=\varphi_1(x_1)\prod_{r=2}^d \omega(x_r),
\qquad x\in[0,1]^d .
\end{align*}
Then $\varphi\in C_c^\infty((0,1)^d)$, $\varphi\not\equiv0$, and integration by parts gives the vanishing moment identities
\begin{equation}
\label{eq:adaptive-lower-vanishing-moments}
\int_{[0,1]^d}x^a\varphi(x)\,dx=0,
\qquad |a|\le R .
\end{equation}
In particular, $\int\varphi=0$.  Write
\begin{align*}
A_1:=\int_{[0,1]^d}|\varphi(x)|\,dx>0,
\qquad
A_2:=\int_{[0,1]^d}\varphi(x)^2\,dx>0.
\end{align*}
For a dyadic integer $m=2^\ell$, let $\mathcal I_m=\{0,1,\ldots,m-1\}^d$.  For $i\in\mathcal I_m$, define
\begin{align*}
\varphi_{m,i}(x):=\varphi(mx-i),
\end{align*}
with the convention that this function is zero outside $m^{-1}(i+[0,1]^d)$.  For $\theta=(\theta_i)_{i\in\mathcal I_m}\in\{-1,1\}^{\mathcal I_m}$, set
\begin{align*}
H_{m,\theta}(x):=\sum_{i\in\mathcal I_m}\theta_i\varphi_{m,i}(x).
\end{align*}
The supports of the $\varphi_{m,i}$'s are disjoint, and therefore
\begin{equation}
\label{eq:adaptive-lower-H-basic}
\int H_{m,\theta}=0,
\qquad
\|H_{m,\theta}\|_1=A_1,
\qquad
\|H_{m,\theta}\|_\infty\le\|\varphi\|_\infty .
\end{equation}
Moreover, uniformly for $s\in[s_-,s_+]$,
\begin{equation}
\label{eq:adaptive-lower-holder-scale}
\|H_{m,\theta}\|_{C^s}\le C_\varphi m^s,
\end{equation}
where $C_\varphi$ depends only on $(d,s_-,s_+,\varphi)$.  This follows by differentiating $\varphi(mx-i)$; derivatives of order $q$ scale as $m^q$, and the corresponding H\"older seminorm scales as $m^s$.

We next record the cross-scale inner-product bounds.  For dyadic $m\le m'$, put
\begin{align*}
B_{m,m'}(i,j):=
\langle \varphi_{m,i},\varphi_{m',j}\rangle_{L_2([0,1]^d)},
\qquad i\in\mathcal I_m,\ j\in\mathcal I_{m'} .
\end{align*}

\begin{lemma}[Cross-scale Gram bounds]
\label{lem:adaptive-lower-gram-bounds}
There is a constant $C_B<\infty$, depending only on $(d,R,\varphi)$, such that for all dyadic $m\le m'$,
\begin{align}
\label{eq:adaptive-lower-gram-frob}
\|B_{m,m'}\|_F^2
&\le C_B\,m'^{-d}\left(\frac m{m'}\right)^{2R+2},
\\
\label{eq:adaptive-lower-gram-op}
\|B_{m,m'}\|_{\mathrm{op}}
&\le C_B\,(mm')^{-d/2}\left(\frac m{m'}\right)^{R+1}.
\end{align}
For $m=m'$, in fact $B_{m,m}=m^{-d}A_2 I$.
\end{lemma}

\begin{proof}
The identity for $m=m'$ follows from disjoint supports and a change of variables.  Suppose $m<m'$, and write $q=m/m'$.  Since $m,m'$ are dyadic, each fine cell of side length $1/m'$ is contained in a unique coarse cell of side length $1/m$.  If the fine cell indexed by $j$ lies inside the coarse cell indexed by $i$, then, after the change of variables $y=m'x-j$,
\begin{align*}
B_{m,m'}(i,j)
=m'^{-d}\int_{[0,1]^d}\varphi(z_{i,j}+qy)\varphi(y)\,dy
\end{align*}
for some $z_{i,j}\in[0,1]^d$.  Taylor expanding $y\mapsto\varphi(z_{i,j}+qy)$ to order $R$, all polynomial terms vanish after integration against $\varphi(y)dy$ by \eqref{eq:adaptive-lower-vanishing-moments}.  The remainder is bounded by $Cq^{R+1}$, uniformly in $(i,j)$.  Hence
\begin{equation*}
|B_{m,m'}(i,j)|\le C m'^{-d}q^{R+1}.
\end{equation*}
Only $O(m'^d)$ pairs $(i,j)$ can overlap, so \eqref{eq:adaptive-lower-gram-frob} follows. Also, each row has at most $Cq^{-d}$ nonzero entries and each column has at most one nonzero entry. Thus the maximum row sum is bounded by $C m^{-d}q^{R+1}$, while the maximum column sum is bounded by $C m'^{-d}q^{R+1}$. Therefore,
\begin{align*}
\|B_{m,m'}\|_{\mathrm{op}}
\le
\sqrt{\|B_{m,m'}\|_{1}\,\|B_{m,m'}\|_{\infty}},
\end{align*}
where
\begin{align*}
\|A\|_{1}:=\max_j\sum_i |A_{ij}|,
\qquad
\|A\|_{\infty}:=\max_i\sum_j |A_{ij}|.
\end{align*}
Hence
\begin{align*}
\|B_{m,m'}\|_{\mathrm{op}}
\le
C(mm')^{-d/2}\left(\frac{m}{m'}\right)^{R+1},
\end{align*}
which is \eqref{eq:adaptive-lower-gram-op}.
\end{proof}

We shall also use a standard exponential bound for decoupled Rademacher chaos.

\begin{lemma}[Rademacher chaos exponential bound]
\label{lem:adaptive-lower-chaos-mgf}
There exist universal constants $c_0,C_0>0$ such that the following holds. Let $\xi\in\{-1,1\}^K$ and $\eta\in\{-1,1\}^L$ have independent i.i.d. Rademacher coordinates, and let $B\in\mathbb R^{K\times L}$.  If $|t|\|B\|_{\mathrm{op}}\le c_0$, then
\begin{align*}
\E\exp\{t\xi^\top B\eta\}
\le
\exp\{C_0t^2\|B\|_F^2\}.
\end{align*}
\end{lemma}

\begin{proof}
Conditioning on $\eta$ and using $\cosh u\le e^{u^2/2}$,
\begin{align*}
\E_\xi e^{t\xi^\top B\eta}
\le
\exp\left\{\frac{t^2}{2}\eta^\top B^\top B\eta\right\}.
\end{align*}
For a positive semidefinite matrix $A$ and $0\le\lambda\le (4\|A\|_{\mathrm{op}})^{-1}$, the standard Rademacher quadratic-form estimate gives
\begin{align*}
\E\exp\{\lambda\eta^\top A\eta\}
\le
\exp\{2\lambda\operatorname{tr}(A)\}.
\end{align*}
Indeed, using the Gaussian integral identity and the sub-Gaussianity of Rademacher sums, for $G\sim N(0,I_L)$,
\begin{align*}
\E_\eta e^{\lambda\eta^\top A\eta}
=
\E_\eta\E_G\exp\{\sqrt{2\lambda}\,G^\top A^{1/2}\eta\}
\le
\E_G\exp\{\lambda G^\top A G\}
=
\det(I-2\lambda A)^{-1/2}.
\end{align*}
If $2\lambda\|A\|_{\mathrm{op}}\le1/2$, then $-\log(1-x)\le2x$ on the relevant eigenvalue range, yielding the displayed bound.  Applying it with $A=B^\top B$ and $\lambda=t^2/2$ proves the lemma after decreasing $c_0$, since $\operatorname{tr}(B^\top B)=\|B\|_F^2$.
\end{proof}

We now build the adaptive mixture.  For large $N$, define
\begin{align*}
\ell_-:=\left\lceil
\frac{2\log_2N-\log_2B_N}{4s_+ +d}
\right\rceil,
\qquad
\ell_+:=\left\lfloor
\frac{2\log_2N-\log_2B_N}{4s_- +d}
\right\rfloor,
\end{align*}
and
\begin{align*}
\mathcal L_N:=\{\ell_-,\ell_-+1,\ldots,\ell_+\},
\qquad
J_N:=|\mathcal L_N|.
\end{align*}
Since $s_-<s_+$, for all sufficiently large $N$,
\begin{equation}
\label{eq:adaptive-lower-J-asymp}
0<c_J\log N\le J_N\le C_J\log N<\infty,
\end{equation}
with constants depending only on $(s_-,s_+,d)$.  For $\ell\in\mathcal L_N$, set $m_\ell=2^\ell$ and define $s_\ell$ by
\begin{equation*}
m_\ell^{-s_\ell}=B_N^{1/4}m_\ell^{d/4}N^{-1/2}.
\end{equation*}
Equivalently,
\begin{align*}
s_\ell=\frac{2\log N-\log B_N}{4\log m_\ell}-\frac d4 .
\end{align*}
The definitions of $\ell_-$ and $\ell_+$ imply
\begin{equation*}
s_\ell\in[s_-,s_+]
\qquad\text{for all }\ell\in\mathcal L_N .
\end{equation*}
Let $\delta>0$, to be chosen sufficiently small depending only on the fixed parameters, and set
\begin{equation}
\label{eq:adaptive-lower-amplitude}
a_\ell:=\delta m_\ell^{-s_\ell}
=
\delta B_N^{1/4}m_\ell^{d/4}N^{-1/2}
=
\delta B_N^{s_\ell/(4s_\ell+d)}N^{-2s_\ell/(4s_\ell+d)}.
\end{equation}
For $\theta\in\{-1,1\}^{\mathcal I_{m_\ell}}$, define
\begin{equation*}
f_{\ell,\theta}:=1+a_\ell H_{m_\ell,\theta}.
\end{equation*}
By \eqref{eq:adaptive-lower-H-basic}, these functions integrate to one.  By choosing $\delta$ sufficiently small and then $N$ sufficiently large, \eqref{eq:adaptive-lower-H-basic}, \eqref{eq:adaptive-lower-holder-scale}, and \eqref{eq:adaptive-lower-amplitude} imply
\begin{equation}
\label{eq:adaptive-lower-membership}
f_{\ell,\theta}\in\cF_{s_\ell}^d(L,M)
\qquad
\text{for all }\ell\in\mathcal L_N,
\quad
\theta\in\{-1,1\}^{\mathcal I_{m_\ell}}.
\end{equation}
Furthermore,
\begin{equation}
\label{eq:adaptive-lower-L1-separation}
\|f_{\ell,\theta}-1\|_1
=A_1a_\ell
=A_1\delta\,B_N^{s_\ell/(4s_\ell+d)}N^{-2s_\ell/(4s_\ell+d)}.
\end{equation}

Let $Q_\ell$ be the mixture distribution
\begin{align*}
Q_\ell
:=2^{-m_\ell^d}
\sum_{\theta\in\{-1,1\}^{\mathcal I_{m_\ell}}}
P_{f_{\ell,\theta}}^{\otimes N},
\end{align*}
and let $L_\ell=dQ_\ell/dP_0^{\otimes N}$.  Define the multiscale mixture
\begin{align*}
\overline Q:=\frac1{J_N}\sum_{\ell\in\mathcal L_N}Q_\ell,
\qquad
\overline L:=\frac1{J_N}\sum_{\ell\in\mathcal L_N}L_\ell.
\end{align*}

\begin{lemma}[Multiscale second-moment bound]
\label{lem:adaptive-lower-second-moment}
If $\delta>0$ is sufficiently small, then
\begin{align*}
\E_0\overline L^2-1\rightarrow0
\qquad\text{as }N\to\infty .
\end{align*}
Consequently, for every $\eta>0$, $\TV(\overline Q,P_0^{\otimes N})\le\eta$ for all sufficiently large $N$.
\end{lemma}

\begin{proof}
Fix $\ell,r\in\mathcal L_N$, and write $m=m_\ell$, $m'=m_r$, $a=a_\ell$, and $a'=a_r$.  Let $\theta$ and $\vartheta$ be independent Rademacher sign vectors at resolutions $m$ and $m'$.  Since $\int H_{m,\theta}=0$,
\begin{align*}
\int f_{\ell,\theta}(x)f_{r,\vartheta}(x)\,dx
=
1+aa'\langle H_{m,\theta},H_{m',\vartheta}\rangle .
\end{align*}
Therefore
\begin{equation*}
\E_0[L_\ell L_r]
=
\E_{\theta,\vartheta}
\left(1+aa'\langle H_{m,\theta},H_{m',\vartheta}\rangle\right)^N .
\end{equation*}
The term inside parentheses is nonnegative because it is the integral of two densities.  Hence $(1+u)^N\le e^{Nu}$ applies, and
\begin{equation*}
\E_0[L_\ell L_r]
\le
\E_{\theta,\vartheta}
\exp\{Naa'\langle H_{m,\theta},H_{m',\vartheta}\rangle\}.
\end{equation*}
Assume without loss of generality that $m\le m'$, and put $h=|\ell-r|$, so that $m/m'=2^{-h}$.  With the Gram matrix $B_{m,m'}$ from \Cref{lem:adaptive-lower-gram-bounds},
\begin{align*}
\langle H_{m,\theta},H_{m',\vartheta}\rangle
=\theta^\top B_{m,m'}\vartheta .
\end{align*}
The amplitude choice \eqref{eq:adaptive-lower-amplitude} gives
\begin{align*}
t:=Naa'=\delta^2B_N^{1/2}m^{d/4}m'^{d/4}.
\end{align*}
Combining this with the operator-norm bound \eqref{eq:adaptive-lower-gram-op}, we obtain
\begin{align*}
t\|B_{m,m'}\|_{\mathrm{op}}
\le
C\delta^2B_N^{1/2}(mm')^{-d/4}2^{-(R+1)h}.
\end{align*}
The resolutions $m_\ell$ grow polynomially in $N$, uniformly over $\ell\in\mathcal L_N$; hence the last display is bounded by the constant $c_0$ from \Cref{lem:adaptive-lower-chaos-mgf} for all large $N$.  Applying \Cref{lem:adaptive-lower-chaos-mgf} and then \eqref{eq:adaptive-lower-gram-frob} gives
\begin{align*}
\E_0[L_\ell L_r]
&\le
\exp\{C_0t^2\|B_{m,m'}\|_F^2\} \\
&\le
\exp\left\{
C\delta^4B_N\left(\frac m{m'}\right)^{d/2+2R+2}
\right\}
=
\exp\{C\delta^4B_N2^{-\alpha_0h}\},
\end{align*}
where $\alpha_0:=d/2+2R+2>2$.  Consequently,
\begin{align*}
\E_0\overline L^2-1
\le
\frac1{J_N^2}
\sum_{\ell,r\in\mathcal L_N}
\left(\exp\{C\delta^4B_N2^{-\alpha_0|\ell-r|}\}-1\right).
\end{align*}
Let $\rho=C\delta^4$, and choose $\delta$ so small that $\rho<1/4$.  For each integer $h\ge0$, there are at most $2J_N$ pairs $(\ell,r)$ with $|\ell-r|=h$.  Therefore
\begin{equation*}
\E_0\overline L^2-1
\le
\frac{2}{J_N}
\sum_{h\ge0}
\left(\exp\{\rho B_N2^{-\alpha_0h}\}-1\right).
\end{equation*}
The term $h=0$ is at most $2J_N^{-1}\{(\log N)^\rho-1\}$, which tends to zero by \eqref{eq:adaptive-lower-J-asymp}.  For $h\ge1$, split the sum at the first $h_0$ such that $B_N2^{-\alpha_0h_0}\le1$.  The part $h<h_0$ has $O(\log B_N)$ summands, each bounded by $(\log N)^\rho-1$, so its contribution is
\begin{align*}
O\left(\frac{\log B_N}{\log N}(\log N)^\rho\right)=o(1).
\end{align*}
For $h\ge h_0$, using $e^x-1\le2x$ gives a contribution bounded by
\begin{align*}
\frac{C}{J_N}B_N\sum_{h\ge h_0}2^{-\alpha_0h}=O(J_N^{-1})=o(1).
\end{align*}
Thus $\E_0\overline L^2-1\to0$.  The total variation claim follows from
\begin{align*}
\TV(\overline Q,P_0^{\otimes N})
\le
\frac12\left(\E_0\overline L^2-1\right)^{1/2}.
\end{align*}
\end{proof}

\begin{proof}[Proof of \Cref{prop:adaptive-gof-lower-loglog}]
Let
\begin{align*}
\eta:=\frac{1-\gamma-\beta}{2}>0.
\end{align*}
By \Cref{lem:adaptive-lower-second-moment}, after choosing $\delta$ sufficiently small and then taking $N$ sufficiently large,
\begin{align*}
\TV(\overline Q,P_0^{\otimes N})\le\eta .
\end{align*}
If $\psi$ satisfies $\E_0\psi\le\gamma$, then
\begin{align*}
\frac1{J_N}\sum_{\ell\in\mathcal L_N}\E_{Q_\ell}\psi
=
\E_{\overline Q}\psi
\le
\E_0\psi+\TV(\overline Q,P_0^{\otimes N})
\le
\gamma+\eta
=
\frac{1+\gamma-\beta}{2}
<1-\beta .
\end{align*}
Hence there exists $\ell\in\mathcal L_N$ such that $\E_{Q_\ell}\psi<1-\beta$.  Since $Q_\ell$ is the uniform mixture over the vertices $f_{\ell,\theta}$, there exists a sign vector $\theta$ with
\begin{align*}
\E_{f_{\ell,\theta}}\psi<1-\beta .
\end{align*}
By \eqref{eq:adaptive-lower-membership}, $f_{\ell,\theta}\in\cF_{s_\ell}^d(L,M)$, with $s_\ell\in[s_-,s_+]$.  By \eqref{eq:adaptive-lower-L1-separation},
\begin{align*}
\|f_{\ell,\theta}-1\|_1
=A_1\delta\,B_N^{s_\ell/(4s_\ell+d)}N^{-2s_\ell/(4s_\ell+d)}.
\end{align*}
The proposition follows with $c=A_1\delta$, after increasing $N_0$ if necessary.
\end{proof}

\begin{proof}[Proof of \Cref{thm:adaptive-lower-loglog}]
Let $N=N_X\wedge N_Y$.  By symmetry in the two samples, it suffices to treat the case $N_X\le N_Y$, so $N=N_X$. Given a one-sample input $U_1,\ldots,U_N$, draw auxiliary observations $V_1,\ldots,V_{N_Y}\stackrel{\iid}{\sim}\Unif([0,1]^d)$, independently of the input and of all internal randomization, and run $\phi$ on $(U_{1:N},V_{1:N_Y})$.  Let $\psi(U_{1:N})$ be the resulting conditional rejection probability.

Under the one-sample null $U_i\stackrel{\iid}{\sim}\Unif([0,1]^d)$, both samples supplied to $\phi$ are uniform, so the assumed type~I control gives $\E_0\psi\le\gamma$.  Applying \Cref{prop:adaptive-gof-lower-loglog}, there exist $s\in[s_-,s_+]$ and $f\in\cF_s^d(L,M)$ such that
\begin{align*}
\|f-1\|_1\ge c(\log\log N)^{s/(4s+d)}N^{-2s/(4s+d)},
\qquad
\E_f\psi<1-\beta .
\end{align*}
Set $g\equiv1$.  Since the uniform density belongs to $\cF_s^d(L,M)$ for $L>1$ and $M>1$, we have $f,g\in\cF_s^d(L,M)$,
\begin{align*}
\|f-g\|_1=\|f-1\|_1
\ge c(\log\log N)^{s/(4s+d)}N^{-2s/(4s+d)},
\end{align*}
and
\begin{align*}
\E_{f,g}\phi=\E_f\psi<1-\beta .
\end{align*}
The case $N_Y<N_X$ is identical after interchanging the two samples, with the alternative $(f,g)=(1,f)$.  The proof did not use privacy, so it applies a fortiori to every $\varepsilon$-differentially private test.
\end{proof}

\clearpage

\subsection{Algorithmic statements}
\label{app:algorithmic-statements}

\begin{algorithm}[H]
\caption{Exact Rao--Blackwellized statistic from a contingency table}
\label{alg:rb-linear-time}
\KwIn{Sample sizes $N_X,N_Y$, split size
$1\le m\le\lfloor(N_X\wedge N_Y)/2\rfloor$, contingency table
$\mathcal P=\{(u,v,n_{u,v}):n_{u,v}>0,\ u+v>0\}$.}
\KwOut{The Rao--Blackwellized statistic $\overline Z$.}
Compute $\mathcal A_X:=\{u:\exists v,\ n_{u,v}>0\}$,
$\mathcal A_Y:=\{v:\exists u,\ n_{u,v}>0\}$,
and multiplicities $r_u:=\sum_v n_{u,v}$, $c_v:=\sum_u n_{u,v}$\;
\BlankLine
Precompute $d_m$\;
Set $\rho_m(0)\gets 1$ and $\rho_m(s)\gets 0$ for odd $s$\;
\For{$j=0,\dots,m-2$}{
    $\displaystyle\rho_m(2j+2)
    \gets
    \rho_m(2j)\,\frac{(m-j)(2j+1)}{(j+1)(2m-2j-1)}$\;
}
Set $d_m(0)\gets 0$\;
\For{$s=0,\dots,2m-1$}{
    $\displaystyle d_m(s+1)
    \gets
    \Bigl(1-\tfrac{1}{2m-s}\Bigr)d_m(s)+\rho_m(s)$\;
}
\BlankLine
Tabulate hypergeometric rows and prefix arrays, where $h_{N,a,\ell}$ denotes
the $\operatorname{Hypergeom}(N,a,\ell)$ pmf and $w_{N,\ell}(a)$ its support
size\;
\For{$u\in\mathcal A_X$}{
    Tabulate $h_{N_X,u,m}$ and compute
    $F_u^{(X)}(x):=\sum_{r\le x}h_{N_X,u,m}(r)$
    and $M_u^{(X)}(x):=\sum_{r\le x}r\,h_{N_X,u,m}(r)$\;
    Tabulate $h_{N_X,u,2m}$ and set
    $D_{N_X,u}\gets\sum_s h_{N_X,u,2m}(s)\,d_m(s)$\;
}
\For{$v\in\mathcal A_Y$}{
    Tabulate $h_{N_Y,v,m}$ and compute $F_v^{(Y)}$, $M_v^{(Y)}$ analogously\;
    Tabulate $h_{N_Y,v,2m}$ and set
    $D_{N_Y,v}\gets\sum_s h_{N_Y,v,2m}(s)\,d_m(s)$\;
}
\BlankLine
Accumulate\;
$\displaystyle\overline Z
\gets
{-}\sum_{u\in\mathcal A_X} r_u\,D_{N_X,u}
-\sum_{v\in\mathcal A_Y} c_v\,D_{N_Y,v}$\;
\For{$(u,v,n_{u,v})\in\mathcal P$}{
    \eIf{$w_{N_X,m}(u)\le w_{N_Y,m}(v)$}{
        $\displaystyle\Gamma_{u,v}
        \gets
        \sum_x h_{N_X,u,m}(x)
        \Bigl[\tfrac{mv}{N_Y}
        -2M_v^{(Y)}(x)+x\{2F_v^{(Y)}(x)-1\}\Bigr]$\;
    }{
        $\displaystyle\Gamma_{u,v}
        \gets
        \sum_y h_{N_Y,v,m}(y)
        \Bigl[\tfrac{mu}{N_X}
        -2M_u^{(X)}(y)+y\{2F_u^{(X)}(y)-1\}\Bigr]$\;
    }
    $\overline Z\gets\overline Z+2n_{u,v}\Gamma_{u,v}$\;
}
\Return{$\overline Z$\;}
\end{algorithm}
\FloatBarrier

\end{document}